\documentclass[11pt]{article}
\usepackage{amsmath}
\usepackage{amsthm}
\usepackage{amssymb}
\usepackage{amsfonts}
\usepackage{subcaption}

\usepackage[margin=1.2in]{geometry}

\usepackage{tikz}
\usetikzlibrary{decorations.pathreplacing,shapes}

\usepackage[colorlinks=true,urlcolor=blue]{hyperref} 

\newtheorem{thm}{Theorem}[section]
\newtheorem{lem}[thm]{Lemma}
\newtheorem{prop}[thm]{Proposition}
\newtheorem{cor}[thm]{Corollary}

\numberwithin{figure}{section}

\newcommand{\refT}[1]{Theorem~\ref{#1}}
\newcommand{\refC}[1]{Corollary~\ref{#1}}
\newcommand{\refL}[1]{Lemma~\ref{#1}}
\newcommand{\refS}[1]{Section~\ref{#1}}
\newcommand{\refP}[1]{Proposition~\ref{#1}}
\newcommand{\refF}[1]{Figure~\ref{#1}}
\newcommand{\refand}[2]{\ref{#1} and~\ref{#2}}

\newcommand\cC{\mathcal C}
\newcommand\cD{\mathcal D}
\newcommand\cF{\mathcal F}
\newcommand\cG{\mathcal G}
\newcommand\cK{\mathcal K}
\newcommand\cN{\mathcal N}
\newcommand\cR{{\mathcal R}}
\newcommand\cU{{\mathcal U}}
\newcommand\cX{{\mathcal X}}

\newcommand{\E}[1]{{\mathbf E}\left[#1\right]}
\newcommand{\p}[1]{{\mathbf P}\left(#1\right)}
\newcommand{\I}[1]{{\mathbf 1}_{[#1]}}
\newcommand{\N}{\mathbb{N}}
\newcommand{\Z}{\mathbb{Z}}

\newcommand{\eps}{\varepsilon}
\newcommand{\sibl}{\overset{\mathrm{s}}{\sim} }

\title{A branching process with deletions and mergers that matches the threshold for hypercube percolation\footnote{Alternative title - Infanticide and Incest: a branching process with deletions and mergers}}
\author{Laura Eslava\thanks{Universidad Nacional Autonoma Mexico, Instituto de investigaciones en matematicas aplicadas y en sistemas, Mexico} 
\and Sarah Penington\thanks{Department of Mathematical Sciences, University of Bath, UK}
\and Fiona Skerman\thanks{
Department of Mathematics, Uppsala University, Sweden}
}
\date{\today}

\begin{document}
\maketitle

\begin{abstract}
    We define a graph process $\cG(p,q)$ based on a discrete branching process with deletions and mergers, which is inspired by the 4-cycle structure of both the hypercube $Q_d$ and the lattice $\Z^d$ for large $d$. Individuals have Poisson offspring distribution with mean $1+p$ and certain deletions and mergers occur with probability $q$; these parameters correspond to the mean number of edges discovered from a given vertex in an exploration of a percolation cluster and to the probability that a non-backtracking path of length four closes a cycle, respectively. 

    We prove survival and extinction under certain conditions on $p$ and $q$ that heuristically match the known expansions of the critical probabilities for bond percolation on the lattice $\Z^d$ and the hypercube $Q_d$. These expansions have been rigorously established by Hara and Slade in 1995, and van der Hofstad and Slade in 2006, respectively. We stress that our method does not constitute a branching process proof for the percolation threshold.
    
    The analysis of the graph process survival is considerably more challenging than for branching processes in discrete time, due to the interdependence between the descendants of different individuals in the same generation.
    In fact, it is left open whether the survival probability of $\cG(p,q)$ is monotone in $p$ or $q$; we discuss this and some other open problems regarding the new graph process. 
\end{abstract}

\section{Introduction} \label{sec:intro}
We propose an extension to inhomogeneous branching processes in discrete generations by defining a process $\cG(p,q)$ in which individuals may have more than one parent and the offspring distribution of each vertex depends on its genealogy which, in turn, is no longer represented by a tree but by a graph. Our main theorem gives sufficient conditions for extinction and survival of this graph process. 

The motivation for the construction of $\cG(p,q)$ lies in the local, structural properties of percolation clusters on the hypercube $Q_d$ and the lattice $\Z^d$. Briefly explained, under suitable choices of $p$ and $q$, the graph process we propose approximates the exploration of a cluster in bond percolation, by taking into account only 4-cycles. The survival and extinction conditions we obtain shed some light on the influence the local structure of these graphs has on the coefficients in the asymptotic expansion of the critical probability for percolation. We postpone the discussion of the heuristic interpretation of our model to \refS{sec:heuristic}.

\subsection{Definition of the graph process $\cG(p,q)$} \label{subsec:Gpqfirstdef}

For $p> -1$ and $q\in [0,1]$, we define a graph process consisting of a branching process in discrete generations with deletions and mergers, and denoted by $\cG(p,q):=(G_n)_{n\ge 0}$. Each graph $G_n$ represents the first $n$ generations of the process, and each vertex corresponds to an individual.

The construction is as follows. For each $k\in \N_0$, let $m_k:=(1+p)(1-q)^k$. The process $\cG(p,q)=(G_n)_{n\ge 0}$ is defined recursively, with $G_n:=(V_n,E_n)$ for each $n$.
We let $I_0:=V_0$ and for each $n\ge 1$, we write
$I_n := V_n \setminus V_{n-1}$, so that for each $n$, $I_n$ is the set of vertices corresponding to individuals in the $n$th generation.

The process starts with $G_0$ consisting of a single vertex with no edges. For $n\in \N$, once the  graph $G_{n-1}$ has been constructed, we construct the $n$th generation, $I_n$, and the graph $G_n$ as follows: 
\begin{enumerate}
    \item For each individual $u\in I_{n-1}$, let $k_u$ be the number of vertices in $G_{n-1}$ at graph distance exactly 3 from $u$. Each $u\in I_{n-1}$ has an independent number of `potential' offspring with Poisson distribution and with mean $1+p$; then, for each of these offspring independently, the offspring is deleted with probability $1-(1-q)^{k_u}$. 

    Form a new graph $\tilde G_n$ by adding new vertices to $G_{n-1}$ for each of the offspring that survived the deletion, with edges joining them to their parents. Note that, in $\tilde G_n$, each vertex $u\in I_{n-1}$ has an independent number of offspring with Poisson distribution and with mean $m_{k_u}$.
    \item Now, suppose there are $Y_n$ new vertices: call these vertices $\tilde I_n$, and give them an arbitrary ordering. We define equivalence classes of vertices in $\tilde{I}_n$ as follows. 
    Let $(B_{i,j})_{1\le i<j\le Y_n}$ be i.i.d.~Bernoulli random variables with mean $q$.
    For each $1\le i< j \le Y_n$, if the graph distance (in $\tilde G_n$) between the $i$th and $j$th individuals in $\tilde I_n$ is exactly 4, and if $B_{i,j}=1$, then identify the two individuals.
    
    Form the graph $G_n$ and the set of vertices $I_n$ from $\tilde G_n$ and $\tilde I_n$ by merging each class of identified vertices into a single vertex, with an edge set given by the union of their edge sets, replacing multi-edges with single edges; see \refF{fig:merger} for an example of a merging event.
\end{enumerate}

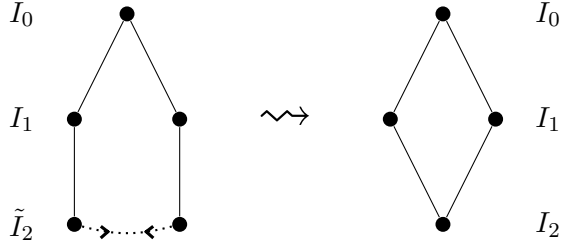
\begin{figure}[ht!]
	\begin{center}
	\begin{tikzpicture}[scale=0.7]
 \begin{scope}[shift={(-2,-2)}]
    		\tikzstyle{vertex}=[circle,fill=black, draw=none, minimum size=4pt,inner sep=2pt]
    		\node[vertex] (a1) at (-1,6){};
    		\node[vertex] (l2) at (-2,4){};
    		\node[vertex] (l3) at (-2,2){};
    		\node[vertex] (a2) at (0,4){};
    		\node[vertex] (a3) at (0,2){};
    		\draw (a1)--(a2)--(a3) (a1)--(l2)--(l3);
    		\draw[dotted, thick] (l3) to [out=-15,in=195] (a3); 
            \node[] at (-3,6){$I_0$};
            \node[] at (-3,4){$I_1$};
            \node[] at (-3,2){$\tilde{I}_2$};
    	\tikzstyle{vertex}=[circle,fill=none, draw=none, inner sep=0pt, minimum size=2]
		\begin{scope}[shift={(-0.7,1.9)}]
    		\node[vertex] (mid) at (0,0.017){}; 
    		\node[vertex] (mid2) at (0,-0.09){};
    		\node[vertex] (mid-below) at (0.245,-0.2){}; 
    		\node[vertex] (mid-above) at (0.245,0.1){};
    		\draw[color=black, very thick] (mid-below) -- (mid) (mid2) --(mid-above);
		\end{scope}

		\begin{scope}[shift={(-1.3,1.9)}]
    		\node[vertex] (mid) at (0,0.017){}; 
    		\node[vertex] (mid2) at (0,-0.09){};
    		\node[vertex] (mid-below) at (-0.245,-0.2){}; 
    		\node[vertex] (mid-above) at (-0.245,0.1){}; 
    		\draw[color=black, very thick] (mid-below) -- (mid) (mid2) --(mid-above);
		\end{scope}
\end{scope}
\node[] at (0,2) {\huge $\leadsto$};

\begin{scope}[shift={(4,-2)}]
    		\tikzstyle{vertex}=[circle,fill=black, draw=none, minimum size=4pt,inner sep=2pt]
    		\node[vertex] (a1) at (-1,6){};
    		\node[vertex] (l2) at (-2,4){};
    		\node[vertex] (a2) at (0,4){};
    		\node[vertex] (m3) at (-1,2){};
    		\draw (a1)--(a2)--(m3) (a1)--(l2)--(m3);
            \node[] at (1,6){$I_0$};
            \node[] at (1,4){$I_1$};
            \node[] at (1,2){$I_2$};
 \end{scope}
    \end{tikzpicture}
    \end{center}
    \caption{Merging of cousins. On the left, $\tilde{G}_2$ contains a pair of cousins that will merge. On the right, the pair of cousins form a vertex with two parents in $G_2$.}\label{fig:merger}
  \end{figure}

Note that if two individuals $u$ and $v$ in $\tilde I_n$ for some $n$ are graph distance 4 apart in $\tilde G_n$, then their parents in $I_{n-1}$ must have a common parent; we say that $u$ and $v$ are `cousins'.

For $n\ge 0$, let
$Z_n:=|I_n|$ be the number of vertices corresponding to individuals in the $n$th generation of $\cG(p,q)$.

\subsection{Main result}\label{sec:mainthm}

Recall that branching processes in discrete generations (Galton-Watson processes) have the following survival threshold: the process survives with positive probability (resp.~dies out almost surely) if the mean number of offspring is greater than (resp.~less than or equal to) one (except for the special case in which each individual has exactly one offspring); see for example \cite{AN72}.

In the case of the graph process $\cG(p,q)$, we can define a notion of survival by saying that $\cG(p,q)$ dies out if there exists $n\ge 1$ for which $Z_n=0$; otherwise we say that the process survives. It is easy to show that the generation sizes $Z_n$ are stochastically dominated by the generation sizes of a Galton-Watson process with Poisson offspring distribution with mean $1+p$, and so if $p\le 0$ then $\cG(p,q)$ trivially dies out almost surely. Our main theorem concerns the survival threshold of $\cG(p,q)$.

\begin{thm}\label{thm:phases}
There exist $C>0$ and $p_0\in (0,1)$ with $Cp_0<1$ such that the graph process $\cG(p,q)$ with $0<p\le p_0$ exhibits (at least) two phases:
\begin{itemize}
    \item If $q< \frac{2}{5}p(1-Cp)$, then the process survives with positive probability:
    $$\p{Z_n>0 \text{ for all } n\ge 1}>0.$$
    \item If $q> \frac{2}{5}p(1+Cp)$, then the process dies out almost surely:
    $$\p{Z_n>0 \text{ for all } n\ge 1}=0.$$    
\end{itemize}
\end{thm}
We do not attempt to obtain an optimal value for $C$, but we conjecture that in fact there is a critical value $q_c=q_c(p)$
such that the process $\mathcal G(p,q)$ survives with positive probability for $q<q_c$ and dies out for $q>q_c$. 
We also leave open the question of monotonicity of the survival probability in $p$ and $q$ (this monotonicity would imply the existence of such a critical $q_c$).
It seems reasonable to expect that the survival probability increases with $p$ (since larger values of $p$ correspond to larger numbers of offspring) and decreases with $q$ (since a larger value of $q$ corresponds to more deletions and mergers).
However, under the most natural coupling of the processes $\mathcal G(p,q)$ and $\mathcal G(p',q)$ with $p>p'$, survival of $\mathcal G(p',q)$ does not imply survival of $\mathcal G(p,q)$, and similarly, under the most natural coupling of the processes $\mathcal G(p,q)$ and $\mathcal G(p,q')$ with $q'>q$, survival of $\mathcal G(p,q')$ does not imply survival of $\mathcal G(p,q)$ (see Figures~\ref{fig:pnonmonotone} and~\ref{fig:qnonmonotone} and the discussion in the Appendix).

Take $\Omega$ a large positive integer and $\rho \in (0,1)$, and let $p(\rho) :=(\Omega -1)\rho -1$ and $q_b := (\Omega -1)^{-2}$.
In \refS{sec:design} below, we will show heuristically that the process $\mathcal G(p(\rho),q_b)$ can be seen as a toy model of an exploration of a bond percolation cluster with edge probability $\rho$ in the hypercube $Q_{\Omega}$ (and the lattice $\Z^{\Omega/2}$, if $\Omega$ is even).
The following corollary, which is a direct consequence of \refT{thm:phases},
shows that the survival and extinction phases of $\mathcal G(p(\rho),q_b)$ match the expansion in $(\Omega -1)^{-1}$ of
the critical probability for percolation up to the first three coefficients
(see~\eqref{pc_exp} in \refS{sec:heuristic} below).
We note in passing that it will become clear from the proof of \refT{thm:phases} that the $5/2$ coefficient in the definition of $\hat \rho_c$ below has a combinatorial reasoning; namely, a 1/2 comes from cousin mergers, while the remaining 4/2 comes from the deletion of edges (see the brief discussion after Proposition~\ref{prop:Zn} below). 

\begin{cor}\label{cor:threshold}
Take $p_0,C>0$ as in \refT{thm:phases}. 
There exists $C'>0$ such that the following holds for $\Omega \in \N$ sufficiently large.
For any $\rho\in (0,(1+p_0\wedge (2C)^{-1})(\Omega-1)^{-1}]$, let $p(\rho):=(\Omega-1)\rho-1$ and let
\begin{align*}
    \hat{\rho}_c&:=(\Omega-1)^{-1}+\frac 52 (\Omega-1)^{-3}.
\end{align*}
Letting $q_b:=(\Omega-1)^{-2}$,
\begin{itemize}
    \item If $\rho > \hat{\rho}_c +C' \Omega^{-5}$, then $\cG(p(\rho),q_b)$ survives with positive probability.
    \item If $\rho < \hat{\rho}_c -C' \Omega^{-5}$, then $\cG(p(\rho),q_b)$ dies out almost surely.
\end{itemize}
\end{cor}

A similar heuristic explanation for a threshold appeared in the case of the $k$-core in Erd\H{o}s-R\'enyi random graphs; that is, the unique maximal subgraph with minimum degree $k$ \cite{B84}. The threshold for the appearance of the giant $k$-core was first proven by Pittel, Spencer and Wormald in 1996 \cite{PSW96}; they observed 
- ``some serious gaps and leaps of faith notwithstanding'' -
that the proportion of vertices that are in the $k$-core should coincide with the probability that a Poisson branching process that mimics the exploration of the connected component of an arbitrary vertex has a $k$-regular tree as a subgraph. However, their proofs did not use any couplings between branching processes and the exploration of the random graph. Later, Riordan in 2008 \cite{R08} provided a proof using a local coupling of the random graph with a suitable Poisson branching process. 

\subsection{Outline}

In the next section we provide the necessary background on bond percolation, and explain the heuristic link between cluster exploration and the graph process $\cG(p(\rho),q_b)$, for $p(\rho)$ and $q_b$ as defined in \refC{cor:threshold}. In \refS{sec:outline} we lay out the main intermediate results (Propositions~\ref{prop:Z4}--\ref{prop:superGW}) of the paper and, assuming those, we prove \refT{thm:phases}. 
Sections~\ref{sec:UlamHarris}--\ref{sec:mainproofs} are then concerned with proving these propositions (see \refS{sec:overview} for an overview of these sections).

\section{Heuristic link to percolation in $Q_d$ and $\Z^d$}\label{sec:heuristic}

For $d\in \N$, the hypercube $Q_d$ and the lattice $\Z^d$ are graphs with vertex sets $\{0,1\}^d$ and $\Z^d$, respectively. Their edge sets are given by nearest neighbours; that is, $\{u,v\}$ is an edge if and only if $|u-v|=1$, where $|w|:=\sum_{i=1}^d |w_i|$. In what follows we write $\mathbb{G}$ to refer to either $Q_d$ or $\Z^d$ (for large values of $d$) and write $\Omega$ for the number of neighbours per vertex; that is, $\Omega=d$ if $\mathbb{G}=Q_d$ and $\Omega=2d$ if $\mathbb{G}=\Z^d$. Results and heuristics are then given in terms of $\Omega$.

For a graph $G:=(V^G,E^G)$ and $\rho\in [0,1]$, we define the bond percolation graph $G_\rho$ as the subgraph of $G$ in which each edge is present with probability $\rho$, independently from all the other edges. 
One of the central objects of study in percolation theory is the critical probability $p_c(G)$, which is rigorously defined for bond percolation on infinite graphs \cite{HH15}, and on (classes of) finite graphs; see e.g. \cite{BR06,BCHSS05a, NP08, JW18}. For an infinite graph $G$,
\begin{equation*}
    p_c(G):=\inf\{\rho\in [0,1]: \exists \text{ an infinite connected component in } G_\rho \text{ a.s.}\}.
\end{equation*}
For finite graphs, we follow the definition in \cite{BCHSS05a} and assume that $G$ is connected and transitive. In this case, the critical probability consists of a critical window rather than a precise point. For $\lambda>0$ let
\begin{equation*}
  p_c(G)=p_c(G,\lambda):=\inf\{\rho \in [0,1]: \mathbb{E}_\rho [|\cC|]\ge \lambda |V^G|^{1/3}\}, 
\end{equation*}
where $|\cC|$ is the size of the connected component containing a fixed root vertex and $\mathbb{E}_\rho$ denotes expectation with respect to $G_{\rho}$. The choice of (small) $\lambda$ is rather flexible in the case $G=Q_d$ for large $d$ \cite{HS05}. 

The first three terms in the asymptotic expansion of $p_c(\mathbb{G})$ in terms of $\Omega^{-1}$ are known, and rigorous bounds on the error terms have been proved. The critical probability is given by 
\begin{equation}\label{pc_exp}
    p_c(\mathbb{G})=\Omega^{-1}+\Omega^{-2}+\frac{7}{2}\Omega^{-3}+O(\Omega^{-4})=(\Omega-1)^{-1}+\frac{5}{2}(\Omega-1)^{-3}+O(\Omega^{-4});
\end{equation}
these expansions were first obtained by Hara and Slade for $\Z^d$ \cite{HS93}, using a delicate application of the lace expansion. Later, van der Hofstad and Slade provided a unified proof for both $Q_d$ and $\Z^d$ \cite{HS06}, again using the lace expansion. In addition, the same authors~\cite{HS05} proved the existence of asymptotic expansions to all orders in $\Omega^{-1}$ (for $Q_d$, these are independent of the choice of $\lambda)$, without computing the numerical values of these coefficients. They conjectured that the coefficient of $\Omega^{-4}$ will differ for $Q_d$ and $\Z^d$ \cite{HS06}. 

The study of the percolation threshold for $\Z^d$, and other infinite lattices, has a long history in both the mathematics and the physics literature, see e.g. \cite{HS93,HM20} and the references therein. The study of percolation in the hypercube started with the analysis of the connectivity probability \cite{ES79,B83} and bounds on the threshold for the emergence of a linear-sized component \cite{AKS82,BKL92}. A formal definition of $p_c$ for general graph classes was first given in work by Borgs et al.~\cite{BCHSS05a,BCHSS05b,BCHSS06c}. For a thorough survey of hypercube percolation see~\cite{HN14}, where van der Hofstad and Nachmias also prove an upper bound on $p_c(\mathbb G)$ corresponding to the first three terms in~\eqref{pc_exp} without using the lace expansion.

In the remainder of this section, we demonstrate the heuristic connection between our graph process $\cG(p,q)$ and an exploration process of a percolation cluster in the hypercube $Q_d$ and the lattice $\Z^d$. We will observe that the only properties of the underlying graph that we model are the transitivity of the vertices, the 4-cycle structure, and the absence of odd cycles. 

\subsection{Exploration of a percolation cluster}\label{sec:exploration}

We first define an exploration algorithm that constructs a percolation cluster in a graph $G$, i.e.~the connected component of $G_\rho$ containing some root vertex.
Each time an edge is explored by the algorithm, the edge will be added to the percolation cluster with probability $\rho \in [0,1]$.
Unlike in the usual exploration algorithms such as Breadth-first search and Depth-first search, 
where the exploration is performed by sequentially selecting  a vertex from a given set and exploring its incident edges, instead we will explore all neighbouring unexplored edges at once, mimicking the generations of offspring in a branching process in discrete time. 

Let $G:=(V^G,E^G)$ be a graph with no cycles of odd length, and take a root vertex $v_\emptyset \in V^G$.
The exploration process is a graph process $(H_n)_{n\ge 0}:=(V_n,E_n)_{n\ge 0}$ defined recursively (with $V_{n-1}\subseteq V_n$ and $E_{n-1}\subseteq E_n$ for each $n$).
Let $S_0:=\{v_\emptyset\}$ and for each $n\ge 1$, we let $S_n :=V_n \setminus V_{n-1}$ denote the set of vertices in the `$n$th generation' of the exploration process.
Initially, we set $H_0:=(\{v_\emptyset \},\emptyset)$ and we say that all edges in $G$ are `unexplored edges'.
For $n\ge 1$, after constructing $H_{n-1}$, we construct $H_n$ as follows:
\begin{enumerate}
    \item We construct $E_n\setminus E_{n-1}$ by first constructing a set of edges $F_n\subseteq E^G$.
    For each vertex $u\in S_{n-1}$, for each edge $e=\{u,v\}\in E^G \setminus E_{n-1}$, we add the edge $e$ to $F_n$ independently with probability $\rho$. 
    Then we let $E_n\setminus E_{n-1}$ be the set of edges in $F_n$ that are unexplored edges.
    (We say that edges in $F_n$ that are explored edges are `deleted' at this stage.)
    From now on, the edges in $E^G$ from each vertex in $S_{n-1}$ are `explored edges'.
    \item Let $V_n:=\{v\in V^G :\exists \{u,v\}\in E_n\}\cup \{v_\emptyset\}$, the set of vertices reached by edges in $E_n$.
    Note that $S_n=V_n\setminus V_{n-1}$ is the set of vertices reached by newly explored edges (edges in $E_n \setminus E_{n-1}$).
\end{enumerate}
See Figure~\ref{fig:perc_expl} for an illustration of this graph process with $G=Q_4$.
Observe that since edges are only added to the graph process if they were not previously explored, in which case they are added with probability $\rho$, the graph $H=\cup_{n\ge 0} H_n$ is a bond percolation cluster with edge probability $\rho$. We claim that the structure of the process $(H_n)_{n\ge 0}$ is analogous to that of $\cG(p,q)$ in the sense that {\it i)} edges only connect vertices in consecutive generations and {\it ii)} in a given generation $n\ge 2$, multiple new edges in $E_n \setminus E_{n-1}$ may `merge' (lead to the same new vertex in $S_n$). 

Indeed, note that each edge in $E_n\setminus E_{n-1}$ was not previously explored, and so cannot have an end-vertex in $V_{n-2}$.
Moreover, since the graph $G$ has no odd cycles, and since for each $u\in S_{n-1}$ there is a path in $H_{n-1}$ from $u$ to $v_\emptyset$ of length $n-1$, an edge in $E^G$ cannot have both end-vertices in $S_{n-1}$.
Hence each edge in $E_n\setminus E_{n-1}$ has one end-vertex in $S_{n-1}$ and one in $S_n$, which verifies property {\it i)} above.

For $n\ge 2$, in step 2 above it may be the case that two (or more) edges $e_1,e_2,\ldots,e_k \in E_n\setminus E_{n-1}$ have the same end-vertex $v\in S_n$ (see Figure~\ref{fig:perc_expl} for examples of this). In this case, we say that edges $e_1,\ldots, e_k$ `merge'.

\begin{figure}
\begin{minipage}[b]{.5\textwidth}
\begin{center}
\begin{tikzpicture}[scale=1]
    		\tikzstyle{vertex}=[circle,fill=black, draw=none, minimum size=4pt,inner sep=2pt]
    		\tikzstyle{vertexRoot}=[circle,fill=red, draw=none, minimum size=5 pt,inner sep=2pt]
    		\tikzstyle{vertexRootb}=[circle,fill=red!60, draw=none, minimum size=5 pt,inner sep=2pt]
		\tikzstyle{vertexDouble}=[rectangle,fill=blue, draw=none, minimum size=5 pt,inner sep=2pt]
		\tikzstyle{vertexDoubleb}=[rectangle,fill=blue!60, draw=none, minimum size=5 pt,inner sep=2pt]

		\tikzstyle{vertexExploreFrom}=[diamond,fill=brown, draw=none, minimum size=5 pt,inner sep=2pt]
    		\node[vertexRoot] (v1000) at (0,4.5){};
    		\node[vertex] (v1010) at (-.75,3){};
    		\node[vertex] (v1100) at (.75,3){};
    		\node[vertexDouble] (v1011) at (-1.5,1.5){};
    		\node[vertex] (v1110) at (0,1.5){};
    		\node[vertex] (v1101) at (1.5,1.5){};
    		\node[vertex] (v0011) at (-2,0){};
    		\node[vertex] (v1111) at (0,0){};
    		\node[vertexExploreFrom] (v1001) at (2,0){};
        		
		\node[vertex] (v0111) at (-2,-1.5){};
    		\node[vertexDoubleb] (v1011b) at (0.5,-1.5){};
    		\node[vertex] (v0001) at (2,-1.5){};
    		\node[vertexRootb] (v1000b) at (3.5,-1.5){};
    		
		\draw (v1010)--(v1000)--(v1100);
        		\draw (v1011)--(v1010)--(v1110)--(v1100)--(v1101);
        		\draw (v0011)--(v1011);
        		\draw (v1011)--(v1111)--(v1101)--(v1001);
        		\draw (v1110)--(v1111);
		
		\draw (v0011)--(v0111);%
		\draw (v1001)--(v0001);%
		\draw[black!60, dashed] (v1001)--(v1011b);%
		\draw[black!60, dashed] (v1001)--(v1000b);%
                \node[] at (.75,4.5){$(1000)$};
                \node[] at (-1.5,3){$(1010)$};
                \node[] at (1.5,3){$(1100)$};
                \node[] at (-2.4,1.95){$w=$};
                \node[] at (-2.4,1.5){$(1011)$};
                \node[] at (.65,1.5){$(1110)$};
                \node[] at (2.25,1.5){$(1101)$};
                \node[] at (-1.25,0){$(0011)$};
                \node[] at (.75,0){$(1111)$};
                \node[] at (3.15,0){$u=(1001)$};    	
                \node[] at (-1.25,-1.5){$(0111)$};
                \node[] at (2.75,-1.5){$(0111)$};
                \node[] at (-3.5,4.5){$S_0:$};
                \node[] at (-3.5,3){$S_1:$};
                \node[] at (-3.5,1.5){$S_2:$};
                \node[] at (-3.5,0){$S_3:$};
                \node[] at (-3.5,-1.5){$S_4:$};
\end{tikzpicture}
\end{center}
    \subcaption{Exploration process}\label{fig:distantRela}
\end{minipage}
\begin{minipage}[b]{.5\textwidth}
\begin{center}
\begin{tikzpicture}[scale=.2]
   	\tikzstyle{vertex}=[circle,fill=black!30, draw=none, minimum size=2pt,inner sep=1pt]
  	\tikzstyle{vertexB}=[circle,fill=black, draw=none, minimum size=2pt,inner sep=1pt]
  	\tikzstyle{vertexRoot}=[circle,fill=red, draw=none, minimum size=5 pt,inner sep=2pt]
	\tikzstyle{vertexDouble}=[rectangle,fill=blue, draw=none, minimum size=5 pt,inner sep=2pt]
	\tikzstyle{vertexExploreFrom}=[diamond,fill=brown, draw=none, minimum size=5 pt,inner sep=2pt]

    \node[vertexExploreFrom] (1001) at (0,0){};   
    \node[vertexDouble] (1011) at (0,18){};   
    \node[vertexRoot] (1000) at (7,6){};   
    \node[vertexB] (1010) at (7,12){};   
    \node[vertexB] (0001) at (9,4){};   
    \node[vertexB] (0011) at (9,22){};   
    \node[vertex] (0000) at (11,8){};   
    \node[vertex] (0010) at (11,14){};   
    \node[vertexB] (1100) at (13,6){};   
    \node[vertexB] (1110) at (13,12){};   
    \node[vertex] (0100) at (17,8){};   
    \node[vertex] (0110) at (17,14){};   
    \node[vertexB] (1101) at (18,0){};   
    \node[vertexB] (1111) at (18,18){};   
    \node[vertex] (0101) at (27,4){};   
    \node[vertexB] (0111) at (27,22){};
    \draw[color=black!30] (0000)--(1000)--(1010)--(1011)--(0011)--(0010)--  
    (0110)--(1110)--(1010)--(0010)--(0000)--(0100)--(1100)--(1000)--(1001)--
    (1011)--(1111)--(1110)--(1100)--(1101)--(1001)--(0001)--(0011)--(0111)--       (1111)--(1101)--(0101)--(0111)--(0110)--(0100)--(0101)--(0001)--(0000);
        		\draw[thick, -latex] (1000)--(1100);
        		\draw[thick, -latex] (1000)--(1010);
        		\draw[thick, -latex] (1010)--(1011);
        		\draw[thick, -latex] (1010)--(1110);
        		\draw[thick, -latex] (1100)--(1110);
        		\draw[thick, -latex] (1100)--(1101);
        		\draw[thick, -latex] (1011)--(0011);
        		\draw[thick, -latex] (1011)--(1111);
        		\draw[thick, -latex] (1101)--(1111);
        		\draw[thick, -latex] (1101)--(1001);
        		\draw[thick, -latex] (1110)--(1111);
		
		\draw[thick, -latex] (1001)--(0001);
		\draw[thick, -latex, black!30, dashed] (1001)--(1011);
		\draw[thick, -latex, black!30, dashed] (1001)--(1000);
		
		\draw[thick, -latex] (0011)--(0111);
        		
    \node[] at (-2,17){$w$};
    \node[] at (-2,0){$u$};
\end{tikzpicture}
\end{center}
    \subcaption{Subgraph $H_4$ of the hypercube $Q_4$}\label{fig:distantRelb}
\end{minipage}
    \caption{Percolating on the set of \emph{all} non-backtracking edges instead of \emph{unexplored} edges can lead to the situation above occurring: while percolating edges from $u$ (${\color{brown} \blacklozenge})$, the edges to the vertices $w$ (${\color{blue} \blacksquare})$ and the root (${\color{red} \bullet})$ may be added even though they have been excluded from the percolation cluster earlier in the exploration process. For this reason, some edges in $F_n$ may be `deleted' (are not added to $E_n\setminus E_{n-1}$).} 
    \label{fig:perc_expl}
\end{figure}
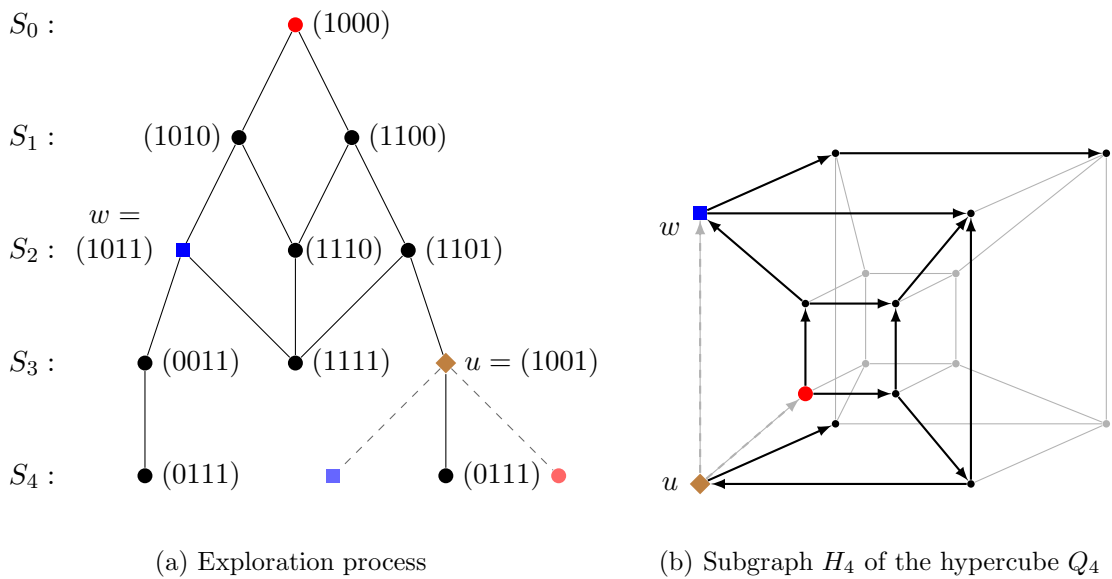  

\subsection{Heuristic link between exploration process and $\mathcal G(p,q)$}\label{sec:design}

In this section, we will show heuristically that for suitable choices of $p$ and $q$, $\mathcal G(p,q)$ is a toy model for the exploration process of a bond percolation cluster defined in Section~\ref{sec:exploration} with $G=Q_d$ or $G=\Z^d$.

Recall from the start of Section~\ref{sec:heuristic} that we let $\mathbb G=Q_d$ (in which case $\Omega =d$) or $\mathbb G=\Z^d$ (in which case $\Omega =2d$) for some large $d\in \N$.
Let us first consider $u_1,\ldots, u_5$ a non-backtracking random walk on $\mathbb{G}$. A key parameter is the probability that vertices $u_1,\ldots, u_5$ form a 4-cycle; let 
\begin{equation} \label{eq:cbdefn}
    c_b:=\p{u_1=u_5}=(\Omega-1)^{-2}+\eps_{\mathbb{G}} =(\Omega-1)^{-2}(1+O(\Omega^{-1})),
\end{equation}
where in particular $\eps_{Q_d}=0$ and $\eps_{\Z^d}=-(\Omega -1)^{-3}$. Note that in a non-backtracking random walk $u_1,\ldots,u_{2\ell+1}$, the probability that  $u_{2\ell+1}=u_1$ is $O(\Omega^{-\ell})$ for $\ell\ge 3$. 

Take $\rho\in (0,1)$; as in Corollary~\ref{cor:threshold} let  
\begin{align*}
    p(\rho):= (\Omega-1)\rho-1\quad \text{and} \quad
    q_b:=(\Omega-1)^{-2}.
\end{align*}
Let $(H_n)_{n\ge 0}$ denote the graph process defined in Section~\ref{sec:exploration} with $G=\mathbb G$ and root vertex $v_{\emptyset}=\bf 0$. We claim that the graph process $\cG(p(\rho),q_b)$ is a toy model of the exploration process $(H_n)_{n\ge 0}=(V_n,E_n)_{n\ge 0}$ (ignoring the labels in $\{0,1\}^d$ or $\Z^d$ of vertices). 

Indeed, for $n\ge 1$ and a vertex $u\in S_{n-1}$, conditional on $(H_m)_{m \le n-1}$, the number of edges from $u$ that are added to $F_n$ has distribution
\begin{equation} \label{eq:hypercubeheur1}
    \text{Bin}(\Omega - \mathrm{deg}_{n-1}(u),\rho) \stackrel{d}{\approx} \text{Poisson}((\Omega - \mathrm{deg}_{n-1}(u))\rho),
\end{equation}
where $\mathrm{deg}_{n-1}(u)$ is the degree of $u$ in the graph $H_{n-1}$. 
For $n\ge 2$ and $u\in S_{n-1}$, we have $\mathrm{deg}_{n-1}(u)\ge 1$.
We will see below that mergers happen with small probability, and so with high probability we have $\mathrm{deg}_{n-1}(u)= 1$ and $(\Omega - \mathrm{deg}_{n-1}(u))\rho=1+p(\rho)$.

For each edge $e=\{u,v\}\in F_n$, the edge $e$ is `deleted' if and only if $v\in V_{n-2}$, i.e.~if and only if there exist $\ell\ge 2$ and a non-backtracking cycle $v=v_1,\ldots, v_{2\ell}=u,v$ in $\mathbb{G}$ such that $\{v_i,v_{i+1}\}\in E_{n-1}$ for each $1\le i \le 2\ell-1$ (since the edge $\{u,v\}$ cannot have both end-vertices in $S_{n-1}$, as observed in Section~\ref{sec:exploration}).
For our toy model, we ignore the effect of cycles of length greater than four, and use~\eqref{eq:cbdefn} to approximate the probability (conditional on the graphs $(H_m)_{m\le n-1}$ without vertex labels) that the edge from $u$ is deleted by
\begin{equation} \label{eq:hypercubeheur2}
    1-(1-c_b)^{k_u}\approx 1-(1-q_b)^{k_u},
\end{equation}
where $k_u$ is the number of vertices at graph distance exactly three from $u$ in $H_{n-1}$.

In step~2 of the construction of $H_n$, note that a pair of edges $e_1,e_2\in E_n \setminus E_{n-1}$ `merge' (have a common end-vertex in $S_n$) if and only if there exist $\ell\ge 2$ and a cycle $e_1,e_2,\ldots, e_{2\ell}, e_{2\ell+1}=e_1$ of edges with $e_3,\ldots, e_{2\ell} \in E_{n-1}$.
As above, we ignore the `lower order' effect of cycles of length greater than four, and use the fact that a non-backtracking random walk of length four on $\mathbb{G}$ forms a cycle with probability $c_b$, to give us the heuristic that a pair of edges $e_1,e_2\in E_n \setminus E_{n-1}$ from $u_1,u_2\in S_{n-1}$ merges only if the graph distance from $u_1$ to $u_2$ in $H_{n-1}$ is two, in which case the edges merge with probability $q_b$. (We assume for our toy model that each pair of edges merges independently from other pairs, neglecting the effect of correlations.)

By comparing with the definition in Section~\ref{subsec:Gpqfirstdef}, this means that $\cG(p(\rho),q_b)$ mimics the exploration process of a percolation cluster on $\mathbb{G}_\rho$ by incorporating the effect of the local 4-cycle structure of $\mathbb{G}$.

\subsection{Open questions}\label{sec:Questions}

The process $\cG(p,q)$ is a generalisation of branching processes in which mergers may occur; the corresponding loss of the tree structure substantially increases the challenges in the analysis of the survival probability. In particular, establishing the monotonicity of the survival probability in $p$ and $q$ is left open (see the discussion after Theorem~\ref{thm:phases}).

As we saw in Section~\ref{sec:design},
the particular forms of the (inhomogeneous) offspring distribution and merger probabilities are tuned so that for suitable $p$ and $q$, the process $\cG(p,q)$ is a toy model for the exploration of bond percolation clusters on both $Q_d$ and $\Z^d$ for large $d$. We expect that if the model
was modified to take into account the effect of
the cycle structure of $Q_d$ or $\Z^d$ up to cycles of length $2k$, for some $k\ge 3$,
then Corollary~\ref{cor:threshold} would still hold for such a model.

For integers $d,n\ge 1$, the Hamming graph $H(d,n)$ has vertex set $\{0,1,\ldots,n-1\}^d$ and two vertices $u$ and $v$ are adjacent if they differ in exactly one coordinate. In other words, $H(d,n)$ is the Cartesian product of $d$ complete graphs on $n$ vertices; in particular $H(d,2)=Q_d$. The first three terms in the expansion of the critical percolation probability for Hamming graphs, as $n\to \infty$, are obtained in \cite{FHHH20}; the first two terms were established in \cite{BCHSS05b}. Similarly, the first three terms in the expansion of the critical probability for site percolation on $\Z^d$ as $d\to \infty$ have been rigorously established in \cite{HM19}. 

It would be interesting to define a branching process with deletions and mergers (with adjusted offspring distribution and merge probabilities) that {\it i)} serves as a toy model of cluster exploration for Hamming graphs $H(d,n)$ for large $n$ (respectively, site percolation of $\Z^d$ for large $d$) and {\it ii)} has a survival threshold that matches the known expansion of the critical probability.

\section{Outline of the proof of Theorem~\ref{thm:phases}}\label{sec:outline}

We shall use an auxiliary graph process $\cG'(p,q):=(G'_n)_{n\ge 0}$ in the proof of Theorem~\ref{thm:phases}.
The process $\cG'(p,q)$ is defined in exactly the same way as the definition of the process $\cG(p,q)$ in Section~\ref{subsec:Gpqfirstdef}, except that the mean offspring number for individuals in generations 0, 1 and 2 is given by $m_{k_u+1}$ (instead of $m_{k_u}$). See Section~\ref{sec:UlamHarris} below for a detailed construction (of $\mathcal G(p,q)$ and $\mathcal G'(p,q)$) using Ulam-Harris notation. For $n\ge 0$, let $Z'_n$ denote the number of individuals in the $n$th generation of $\cG'(p,q)$.

We now state a straightforward bound on the growth of the expected size of generations, and a technical inequality for the size of early generations. 

\begin{prop}\label{prop:Z4}
For any $p,q \in [0,1]$ and $n\ge 1$,
\begin{align} 
    \E{Z_n}&\le (1+p-\I{n\ge 4} q)\E{Z_{n-1}}. \label{eq:Z4}
\end{align}
There exist $C_1>0$ and $p_1\in (0,1)$ such that if $0\le q\le p \le p_1$ then
\begin{align}
    (1+C_1 p)\E{Z_6}&\ge \max_{0\le n\le 5} \E{Z_n}. \label{eq:Z3}
\end{align}
Moreover, \eqref{eq:Z4} and \eqref{eq:Z3} also hold when $Z_m$ is replaced with $Z'_m$ for each $m$. 
\end{prop}

The first main intermediate result in the proof of Theorem~\ref{thm:phases} is
an estimate on the growth of the expected size of the generations of the graph process; up to first order terms, the expected size of the generations grows by a factor of $1+p-\frac{5}{2}q$. The following proposition gives us upper and lower bounds on this growth.

\begin{prop}\label{prop:Zn}
There exists $C_2>0$ such that for $0 \le q\le p\le 1$ and for $n\ge 6$, $m\ge 7$,
\begin{align}
    \E{Z_n}&\le (1+p-\tfrac 52 q)\E{Z_{n-1}}+ C_2 p^2 \max_{2\le k\le 6}\E{Z_{n-k}}, 
    \label{eq:upperZ}\\
    \text{ and} \quad \E{Z_m}&\ge (1+p)\E{Z_{m-1}}-\tfrac{5}{2}q(1+p)^4 \max_{2\le k\le 4} \E{Z_{m-k}}-C_2 q^2 \E{Z_{m-4}}.
    \label{eq:lowerZ}
\end{align}
Moreover, \eqref{eq:lowerZ} also holds when $Z_l$ is replaced with $Z'_l$ for each $l$.
\end{prop}

It will become clear in the proof of this result that of the $5/2$ coefficient in~\eqref{eq:upperZ} and~\eqref{eq:lowerZ}, a 1/2 comes from the effect of cousin mergers, and the remaining 4/2 comes from deletions (see the definitions in \eqref{eq:Mndefn}--\eqref{eq:Yndefn}
together with Lemmas~\ref{lem:EML'} and~\ref{lem:EM'} and \eqref{eq:propZnineq}). 

The second main intermediate step in the proof of survival in Theorem~\ref{thm:phases} is to prove that after a constant number of generations the process contains some individuals that \emph{renew} the graph process dynamics in the sense that their descendants form i.i.d.~subprocesses. (We will prove a lower bound on the expected number of such individuals.) This will mean that we can observe an underlying Galton-Watson process occurring with `long-range' generations. We now make this idea precise. 

An individual $v\in I_n$ for some $n\ge 3$ is called a \emph{renewed vertex} if there are exactly three individuals in $G_{n}$ at graph distance at most three from $v$. In genealogical terms, a renewed vertex has exactly one great-grandparent, one grandparent, and one parent, and it has no aunts or siblings; see Figure~\ref{fig:renewed} for a representation of the ancestry and descendants of a renewed vertex. For $n\ge 3$, let $R_n$ be the number of renewed vertices in generation $n$ for the process $\cG(p,q)$, and let $R'_n$ denote the corresponding quantity for the process $\mathcal G'(p,q)$.

\begin{figure}
\begin{subfigure}[t]{.55\textwidth}
    \hfill
	\begin{center}
	\begin{tikzpicture}[scale=0.7]
    		\tikzstyle{vertex}=[circle,fill=black, draw=none, minimum size=4pt,inner sep=2pt]
    		\node[vertex] (t) at (-3,6){};
    		\node[vertex] (u) at (-1.5,6){};
    		\node[vertex] (v) at (0,6){};
    		\node[vertex] (w) at (1.5,6){};
    		\node[vertex] (u1) at (-2.25,4){}; 
    		\node[vertex] (u2) at (-1,4){};
    		\node[vertex] (v1) at (.5,4){};
    		\node[vertex] (w1) at (2,4){};
    		\node[vertex] (u11) at (-3,2){};
    		\node[vertex] (u21) at (-1.5,2){};
    		\node[vertex] (v11) at (0,2){};
    		\node[vertex] (w11) at (1.25,2){};
    		\node[vertex] (w12) at (2.75,2){};
    		\node[vertex] (u111) at (-2.75,0){};
    		\node[vertex] (u211) at (-1,0){};
    		\node[vertex] (v111) at (0.5,0){};
    		\node[vertex] (w121) at (2.5,0){};
    		\draw 
    		(u1)--(t) 
    		(u)--(u1)--(u11)--(u111)
    		(u)--(u2)--(u21)--(u211)
    		(v)--(v1)--(v11)--(v111)
    		(w)--(w1)--(w11)
    		(w1)--(w12)--(w121);
        	\node[] at (-1,-0.5) {$u$};
        	\node[] at (0.5,-0.5) {$v$};
        	\node[] at (3.5,0) {$I_n$};
    \end{tikzpicture}
    \end{center}
\subcaption{Vertices $u$ and $v$ are the only renewed vertices among the depicted vertices in $I_n$.}\label{fig:renewedcases}
\end{subfigure}
\begin{subfigure}[t]{.45\textwidth}
    \hfill
	\begin{center}
	\begin{tikzpicture}[scale=0.7]
    		\tikzstyle{vertex}=[circle,fill=black, draw=none, minimum size=4pt,inner sep=2pt]
    		\node[vertex, color=gray] (a3) at (-.5,6){};
    		\node[vertex, color=gray] (a2) at (0,5){};
    		\node[vertex, color=gray] (a1) at (-.5,4){};
    		\node[vertex] (v) at (0,3){}; 
    		\node[vertex] (v1) at (-1.5,1.50){};
    		\node[vertex] (v2) at (0,1.5){};
    		\node[vertex] (v3) at (1.5,1.5){};
    		\node[vertex] (v11) at (-2.5,0){};
    		\node[vertex] (v12) at (-.5,0){};
    		\node[vertex] (v21) at (1.25,0){};
    		\node[vertex] (v31) at (2.75,0){};
     		\draw (a3)--(a2)--(a1)--(v)--(v1)--(v11)
    		(v1)--(v12)
    		(v)--(v2)--(v21)--(v3)
    		(v)--(v3)--(v31);
        	\node[] at (0.5,3) {$v$};
        	\node[color=white] at (0,-0.5) {$q$};
    \end{tikzpicture}
    \end{center}
\subcaption{Descendants of a renewed vertex $v$.}\label{fig:renewedlaw}
    \end{subfigure}
    \caption{Renewed vertices in $I_n$ have exactly three relatives in $G_{n}$ at distance at most three. It turns out that the graph process of their descendants has the law of $\mathcal{G}'(p,q)$ (see Lemma~\ref{lem:distGu} below).}
    \label{fig:renewed}
\end{figure}
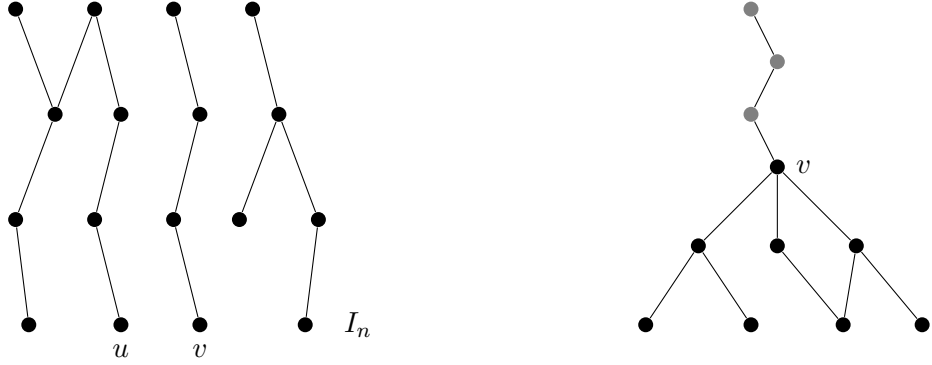
\begin{prop}\label{prop:ERn}
There exist $\eta,p_2\in (0,1)$ such that for $n\ge 9$ and $0\le q\le p\le p_2$,
\begin{align} \label{eq:propERn}
\E{R_n}\ge 2\eta \E{Z_{n-3}}-\eta \max_{3\le k\le 6} \E{Z_{n-k}};
\end{align}
 the same inequality holds when $R_n$, $Z_{n-3}$ and $Z_{n-k}$ are replaced with $R'_n$, $Z'_{n-3}$ and $Z'_{n-k}$.
\end{prop}

Suppose for some $N\ge 3$ that individual $v\in I_N$ is a renewed vertex. 
Then the children of individual $v$ in the graph $\tilde G_{N+1}$ have no cousins and hence cannot merge with any other vertices. Similarly, every descendant of $v$ in generation $N+2$ and subsequent generations can only merge with other individuals whose only ancestor in generation $N$ is $v$. This suggests that the graph processes given by the descendants of different renewed vertices in generation $N$ are independent; in fact, their distribution is given by $\cG'(p,q)$ (a formal statement is given in \refL{lem:distGu}). 
This will enable us to prove the following coupling result. 

\begin{prop}\label{prop:superGW}
Take $N\ge 3$. Define a Galton-Watson process $(X_n)_{n\geq 1}$ as follows.
Let $X_1\stackrel{d}{=}R_N$ and, for $n\geq 1$, let
$$ X_{n+1}:=\sum_{j=1}^{X_n} \xi^{(n)}_j, $$
where $\{\xi^{(n)}_j, n\in \N,j\in \N\}$ are i.i.d.~with $\xi_j^{(n)}\stackrel{d}{=}R'_N$.
We can couple $(Z_n)_{n\geq 1}$ and $(X_n)_{n\geq 1}$ in such a way that 
$Z_{n N} \geq X_{n} \,\, \forall n\in \N$.
\end{prop}

We next prove \refT{thm:phases} assuming Propositions~\ref{prop:Z4}--\ref{prop:superGW}.
We then end this section with an outline of the rest of the paper. 

\begin{proof}[Proof of \refT{thm:phases}]
Let $C_2>0$ be as in \refP{prop:Zn}. Take $0<p\le 2/3$. We first consider the extinction case.
Note that for any $m\in \N$,
$$
\p{Z_n>0 \; \forall n\in \N}\le \p{Z_m\ge 1} \le \E{Z_m}
$$
by Markov's inequality.
Hence it suffices to prove that if $q> \frac{2}{5}p(1+C_2 p)$ then $\E{Z_n}\to 0$ as $n\to \infty$. 

By~\eqref{eq:Z4} in \refP{prop:Z4}, and since $\E{Z_0}=1$, for $n\ge 3$,
$$\E{Z_n}\le (1+p-q)^{n-3}\E{Z_3}\le (1+p-q)^{n-3}(1+p)^3.$$
Therefore, if $q>p$ then $\E{Z_n}\to 0$ as $n\to \infty$.

It remains to consider the case $\frac{2}{5}p(1+C_2 p)<q\le p$. Then $p-\frac 52 q +C_2 p^2<0$, and since $q\le p\le 2/3$, we also have $p-\frac 5 2 q \ge -1$; thus $a:=1+p-\frac{5}{2}q+C_2 p^2\in (0,1)$. For $n\ge 6$, by~\eqref{eq:upperZ} in Proposition~\ref{prop:Zn}, and since $1+p-\frac 5 2 q \ge 0$, we have 
\begin{align}\label{eq:upperZx}
    \E{Z_n}\le (1+p-\tfrac{5}{2}q)\max_{1\le k\le 6}\E{Z_{n-k}} 
    +C_2 p^2 \max_{1\le k\le 6} \E{Z_{n-k}}= a\max_{1\le k\le 6} \E{Z_{n-k}}.
\end{align}
Now we prove by induction on $j$ that for $n\ge 6$, for $j\in \{0,\ldots, 5\}$,
\begin{align}\label{eq:maxs}
    \max_{0\le l\le j} \E{Z_{n+l}} \le a \max_{1\le k \le 6} \E{Z_{n-k}}.
\end{align}
The base case for the induction,~\eqref{eq:maxs} with $j=0$, holds by~\eqref{eq:upperZx}. Suppose \eqref{eq:maxs} holds for some $0\le j<5$; then by~\eqref{eq:upperZx},
\begin{align} \label{eq:thmindpf}
\E{Z_{n+j+1}} 
&\le a \max_{1\le k \le 6} \E{Z_{n+j+1-k}} \notag \\
&= a \, \max\left( \max_{0\le l \le j} \E{Z_{n+l}}, \max_{1\le k \le 5-j} \E{Z_{n-k}} \right)\notag \\
&\le a \max_{1\le k \le 6} \E{Z_{n-k}},
\end{align}
where we use both the induction hypothesis and the fact that $a\in (0,1)$ for the last inequality. 
It follows from~\eqref{eq:thmindpf} and the induction hypothesis that 
\begin{align*}
    \max_{0\le l \le j+1} \E{Z_{n+l}} \le a \max_{1\le k \le 6} \E{Z_{n-k}},
\end{align*}
and so by induction,~\eqref{eq:maxs} holds for each $j\in \{0,1,\ldots , 5\}$.

Now take $m\ge 1$; by~\eqref{eq:maxs} with $j=5$ and $n=6m$, and then using an iterative argument we have
\begin{align*}
        \max_{0\le l \le 5} \E{Z_{6m+l}} \le a \max_{1\le k \le 6} \E{Z_{6m-k}}\le \ldots \le a^m \max_{0\le k \le 5} \E{Z_k}\le a^m (1+p)^5,
\end{align*}
where the last inequality follows by~\eqref{eq:Z4} in \refP{prop:Z4} and since $Z_0=1$.
Letting $m\to \infty$ and using that $a\in (0,1)$, we infer that $\E{Z_n}\to 0$ as $n\to \infty$, as desired.

Now, for the case of survival,
let $C' := \max(C_1,C_2)$.
We can take a constant $C>C_2>0$  sufficiently large that there exists $p_3>0$ such that for $0<p<p_3$, if
$0\le q<\frac 25 p(1-C p)$ then
letting $b:=1+p-(\frac{5}{2}q(1+p)^4+C'q^2)(1+C' p)$, we have $b>1$. 
(This can be seen by substituting $q=\frac 25 p(1-C p)$ into the expression for $b$ and showing that the coefficient of $p^2$ is positive if $C$ is a sufficiently large constant.)
From now on, take $0<p<p_0=\min(p_1,p_2,p_3,2/3)$ and $0 \le q<\frac 25 p(1-C p)$.

Using \eqref{eq:lowerZ} from \refP{prop:Zn} with $m=7$, and then~\eqref{eq:Z3} from \refP{prop:Z4}, we have 
$$\E{Z_7}\ge (1+p)\E{Z_6}-(\tfrac 52 q(1+p)^4 +C' q^2)\max_{3\le k \le 5}\E{Z_k}\ge  b \E{Z_6}.$$
Therefore, since $b>1$, in particular,
\begin{align*}
    (1+C'p)\E{Z_7} \ge (1+C'p)\E{Z_6} \ge \max_{1\le k\le 3} \E{Z_{7-k}}
\end{align*}
by~\eqref{eq:Z3} in \refP{prop:Z4}.
By~\eqref{eq:lowerZ} in Proposition~\ref{prop:Zn} and an induction argument, this implies that for $n\ge 7$, 
\begin{align} \label{eq:thmind2}
    \E{Z_{n}}\ge b\E{Z_{n-1}} \quad \text{and}\quad
    (1+C'p)\E{Z_n}\ge \max_{1\le k \le 3}\E{Z_{n-k}}.
\end{align}
Indeed, for $m>7$, if~\eqref{eq:thmind2} holds with $n=m-1$, then by~\eqref{eq:lowerZ},
\begin{align*}
    \E{Z_m}&\ge (1+p)\E{Z_{m-1}}-(\tfrac{5}{2}q(1+p)^4 +C'q^2) (1+C'p)\E{Z_{m-1}}= b\E{Z_{m-1}}.
\end{align*}
Hence in particular, since $b>1$,
$$
(1+C'p)\E{Z_m}\ge (1+C'p)\E{Z_{m-1}}\ge \max( \E{Z_{m-1}},\max_{2\le k \le 4}\E{Z_{m-k}})\ge  \max_{1\le k \le 3}\E{Z_{m-k}},
$$
where we used~\eqref{eq:thmind2} with $n=m-1$ in the second inequality, which completes the induction argument.

Since $b>1$, we now have that $\E{Z_n}$ is increasing for $n\ge 6$. Moreover, by iterating~\eqref{eq:thmind2}, $\E{Z_n}\ge b^{n-6} \E{Z_6}\ge b^{n-6} (1+C_1 p)^{-1}$ by~\eqref{eq:Z3} in \refP{prop:Z4}, and so $\E{Z_n}\to \infty$ as $n\to \infty$. 
By the same argument, since~\eqref{eq:Z3} and~\eqref{eq:lowerZ} hold when $Z_l$ is replaced with $Z'_l$ for each $l$, we have that for $n\ge 6$, $\E{Z'_n}$ is increasing, and $\E{Z'_n}\to \infty$ as $n\to \infty$.

By \refP{prop:ERn}, and since $\E{Z_m}$ is increasing for $m\ge 6$, we have that for $n\ge 12$,
\begin{align*}
\E{R_n}
&\ge 2\eta \E{Z_{n-3}}-\eta \E{Z_{n-3}}
= \eta \E{Z_{n-3}}>1
\end{align*}
for $n$ sufficiently large. By the same argument for $\E{R'_n}$, it follows that there exists $N\ge 3$ such that $\min\{\E{R_N},\E{R'_N}\}>1$. 

Now, by \refP{prop:superGW} with this choice of $N$, we can couple $(Z_n)_{n\geq 0}$ with $(X_n)_{n\geq 0}$ in such a way that 
$Z_{n N} \geq X_{n} \,\, \forall n\in \N$. Since $\min\{\E{R_N},\E{R'_N}\}>1$, the process $(X_n)_{n\ge 1}$ is a supercritical branching process with $\E{X_1}>1$, and so $\p{X_n>0 \,\, \forall n\in \N}>0$.
Hence
\begin{align*}
\p{Z_n>0 \,\, \forall n\in \N}
&= \p{Z_{nN}>0 \,\, \forall n\in \N } \\
&\ge \p{X_n>0 \,\, \forall n\in \N}>0,
\end{align*}
as required.
\end{proof}

\subsection{Overview of Sections \ref{sec:UlamHarris}-\ref{sec:mainproofs}} \label{sec:overview}

It remains to prove Propositions~\ref{prop:Z4}-\ref{prop:superGW}.
First, in Section~\ref{sec:UlamHarris}, we introduce a construction of the graph processes $\cG(p,q)$ and $\cG'(p,q)$ using Ulam-Harris labelling.
We use this construction to prove Propositions~\ref{prop:Z4} and~\ref{prop:superGW}.
In \refS{sec:upper} we prove almost sure upper and lower bounds for $Z_n$ in terms of new random variables which are analysed in the subsequent sections; see \refL{lem:Z}. 
In Sections~\ref{sec:upper}-\ref{sec:lower}, we establish upper and lower bounds on the expectations of these random variables, and finally in Section~\ref{sec:mainproofs} we complete the proofs of Propositions~\ref{prop:Zn} and~\ref{prop:ERn}.

\section{Ulam-Harris construction, proof of Propositions~\ref{prop:Z4} and \ref{prop:superGW}}\label{sec:UlamHarris}
\subsection{Construction of $\cG(p,q)$ and $\cG '(p,q)$ using Ulam-Harris notation}\label{subsec:UlamHarris}

In this subsection, we construct the graph processes $\cG(p,q)$ and $\cG '(p,q)$ (defined in Sections~\ref{subsec:Gpqfirstdef} and~\ref{sec:outline} resp.) using Ulam-Harris notation to give a labelling of the vertices that will be useful in the rest of the proofs.

We let
$$
\mathcal U_n :=\N^n \text{ for } n\in \N_0, \quad \text{and} \quad \mathcal U :=\bigcup_{n= 0}^\infty \mathcal U_n.
$$
For $u=u_1\ldots u_k$ and $v=v_1 \ldots v_l\in \mathcal U$, we write $uv:=u_1\ldots u_k v_1 \ldots  v_l$. 
For $u =u_1\ldots u_n \in \mathcal U_n$, let $|u|:=n$, and for $m<n$, let $u|_m:=u_1 \ldots u_m.$

The individuals in our graph processes will be labelled using $\mathcal U$.
Take $p>-1$ and $q\in [0,1]$.
We now define three independent families of i.i.d.~random variables $(\xi_u)_{u\in \mathcal U}$, $(\delta_{u,v})_{u,v\in \mathcal U}$ and $(\mu_{\{u,v\}})_{u\neq v\in \mathcal U}$, where each $\xi_u$ has Poisson distribution with mean $1+p$, each $\delta_{u,v}$ has Bernoulli distribution with mean $q$, and each $\mu_{\{u,v\}}$ also has Bernoulli distribution with mean $q$.

Before giving the precise construction of the graph process $\mathcal G(p,q)$ in terms of these random variables, we give a brief overview of how we will construct a new generation.
The random variable $\xi_u$ will determine the initial number of offspring of the individual  in the current generation labelled $u$. Each offspring $ui$ with $i\le \xi_u$ may then be deleted by an individual $v$ at graph distance exactly three from $u$; the deletion occurs if $\delta_{v,ui}=1$. Finally, surviving offspring $u'$ and $v'$ in the new generation that are at graph distance 4 from each other merge into one individual if $\mu_{\{u',v'\}}=1$.

We now give the precise construction.
We construct the process $\cG(p,q)=(G_n)_{n\ge 0}$ in such a way that $G_n=(V_n,E_n)$ is a graph for each $n\in \N_0$, with $V_n \subset \cup_{m\le n}\mathcal U_m$ and $E_n \subset \cup_{1\le m \le n}(\mathcal U_m \times \mathcal U_{m-1})$.
The process starts with $G_0:=(V_0,E_0)$, where $V_0:=\{\emptyset\}$ and $E_0:=\emptyset$.
We let $I_m :=\mathcal U_m \cap V_m$ for each $m\in \N_0$.

The process is constructed iteratively; for each $n\ge 1$, after constructing $G_{n-1}$, we construct $G_n$ as follows.
\begin{enumerate}
    \item For each $u\in I_{n-1}$, let $\cK_u:=\{v\in V_{n-1}:d_{G_{n-1}}(u,v)=3\}$, and let $\mathcal C_u := \{j\le \xi_u : \delta_{v,uj}=0 \; \forall v\in \cK_u\}$.
    Then let $\tilde G_n :=(\tilde V_n , \tilde E_n),$ where
    \begin{align*}
        \tilde V_n &:= V_{n-1} \cup \{ui: u\in I_{n-1}, i \in \mathcal C_u\}\\
        \text{and }\quad \tilde E_n &:= E_{n-1} \cup \{ \{u, ui\}:u\in I_{n-1}, i \in \mathcal C_u\}.
    \end{align*}
    \item Let $\tilde I_n := \tilde V_n \cap \mathcal U_n$.
    For $u,v \in \tilde I_n$, write $u\stackrel{m}{\sim} v$ if and only if $d_{\tilde G_n}(u,v)=4$ and $\mu_{\{u,v\}}=1$.
    Then define an equivalence relation on $\tilde I_n$ by letting $u\sim v$ (for $u,v \in \tilde I_n$) if and only if there exist $k\ge 0$ and $u=u_0, u_1, \ldots, u_k=v \in \tilde I_n$ such that $u_i \stackrel{m}{\sim} u_{i+1}$ for each $0\le i \le k-1$.
    For $u\in \tilde I_n$, let
    $$
    \pi(u):=\min\{v\in \tilde I_{n}: v\sim  u\},
    $$
    where the minimum is with respect to the lexicographical ordering of $\mathcal U$.
    Then we let
    \begin{align*}
    V_n &:= V_{n-1} \cup \{\pi(u): u\in \tilde I_n\}\\
    \text{and } \quad E_n &:= E_{n-1} \cup \{\{v, \pi(vi)\}: v\in I_{n-1}, i \in \mathcal C_v\}.
    \end{align*}
\end{enumerate}
This completes the construction of $G_n$ from $G_{n-1}$.

Define another independent family of i.i.d.~random variables $(\delta_{u})_{u\in \mathcal U}$, where each $\delta_u$ has a Bernoulli distribution with mean $q$.
The construction of $\cG'(p,q)=(G'_n)_{n\ge 0}=(V'_n,E'_n)_{n\ge 0}$ is identical to the construction of $\cG(p,q)$ except that, for $n\le 3$, when constructing $G'_{n}$ from $G'_{n-1}$ we replace the set $\cC_u$ with $\mathcal C'_u := \{j\le \xi_u : \delta_{uj}=0, \delta_{v,uj}=0 \; \forall v\in \cK '_u\}$,
where $\cK '_u: =\{v\in V'_{n-1}:d_{G'_{n-1}}(u,v)=3\}$. 
For $n\ge 4$, the construction of $G'_{n}$ from $G'_{n-1}$ is identical to the construction for the process $\cG(p,q)$.
We write $I'_n :=V'_n \cap \mathcal U_n$ for $n\ge 0$.

\subsection{Notation}
In this subsection, we introduce some notation that will be used throughout Sections~\ref{sec:UlamHarris}-\ref{sec:mainproofs}.
For $n\ge 0$, define the $\sigma$-algebras
\begin{align}
\mathcal F_n &:= \sigma ( (\xi_u)_{u\in \cup_{m\le n-1}\mathcal U_{m}}, (\delta_{u,v})_{u,v\in \cup_{m\le n}\mathcal U_{m}} (\mu_{\{u,v\}})_{u\neq v\in \cup_{m\le n}\mathcal U_{m}} ), \label{dfn:Fn}\\
\text{and }\; \widetilde{\mathcal F}_n &:= \sigma ( (\xi_u)_{u\in \cup_{m\le n-1}\mathcal U_{m}}, (\delta_{u,v})_{u, v\in \cup_{m\le n}\mathcal U_{m}}, (\mu_{\{u,v\}})_{u\neq v\in \cup_{m\le n-1}\mathcal U_{m}} ). \label{dfn:F'n}
\end{align}
Note from our construction in Section~\ref{subsec:UlamHarris} that $G_n$ (respectively, $\tilde G_n$) is $\mathcal F_n$-measurable (respectively, $\widetilde{\cF}_n$-measurable) for each $n\ge 1$.

For $n\ge 0$ and $u\in I_n$, as in Section~\ref{subsec:Gpqfirstdef} we write 
\begin{equation} \label{eq:kudefn}
    k_u :=|\cK_u|=|\{v\in V_n :d_{G_n}(u,v)=3\}|,
\end{equation}
so that $\E{|\cC_u| |\cF_n}=(1+p)(1-q)^{k_u}$.
For $n\ge 0$ and $u\in I_n$, we write $\alpha(u)$ for the set of ancestors of $u$ (including $u$ itself) in $G_n$, i.e.~we let
$$
\alpha(u):= \bigcup_{0\le k \le n}\{v \in I_{n-k}: d_{G_n}(v,u)=k\}.
$$
Then for $m\le n $ and $u\in I_n$, we write
$$
\alpha_m(u):= \alpha(u) \cap I_{n-m},
$$
so that $\alpha_1(u)$ denotes the set of parents of individual $u$, $\alpha_2(u)$ is the set of grandparents, $\alpha_3(u)$ is the set of great-grandparents, and so on. 
(For $m>n$, we let $\alpha_m(u)=\emptyset$.)
Similarly, for $u \in \tilde I_n$, we write $\tilde \alpha(u)$ for the set of ancestors in $\tilde G_n$ of $u$:
$$
\tilde \alpha(u):= \{u\}\cup \bigcup_{1\le k \le n}\{v \in I_{n-k}: d_{\tilde G_n}(v,u)=k\},
$$
and for $1\le m \le n$, we write
\begin{equation} \label{eq:tildealphadefn}
\tilde \alpha_m(u) :=\tilde \alpha(u) \cap I_{n-m}= \alpha_{m-1}(u|_{n-1}) .
\end{equation}
For $m,n\ge 0$, for  a vertex $u\in I_m$, let $I_n^u\subseteq I_{m+n}$ denote the set of vertices which are descendants of $u$ in generation $m+n$, i.e.
\begin{equation} \label{eq:Inudefn}
I_n^u :=\{ v\in I_{m+n} : u\in \alpha (v)\}.
\end{equation}
Note that for $u\in I_m$, by our construction in Section~\ref{subsec:UlamHarris} we have
\begin{equation} \label{eq:offspringbound}
    |I_1^u|=|\{\pi(ui):i\in \cC_u\}|\le |\cC_u|\le \xi_u.
\end{equation}
For $m,n\ge 0$ and $u\in I_m$, let
\begin{align} \label{eq:tildeIudefn}
\tilde{I}_{n+1}^u:=&\{wi :\, w\in I_n^u, \, wi\in \tilde{I}_{m+n+1}\},
\end{align}
the set of offspring in $\tilde I_{m+n+1}$ of individuals in $I^u_n$.
Note that by our construction in Section~\ref{subsec:UlamHarris},
$I^u_{n+1}=\{\pi(wi):wi\in \tilde I^u_{n+1}\}$ and so
\begin{equation} \label{eq:ItildeIbd}
    |I^u_{n+1}|\le |\tilde I^u_{n+1}|.
\end{equation}
For $n\ge 0$ and $v\in I_n$, write $\sigma(v)$ for the equivalence class in $\tilde I_n$ that merges to form $v$, i.e.
\begin{equation} \label{eq:sigmavdefn}
\sigma(v):=\{\tilde v \in \tilde I_n : \tilde v \sim v \}.
\end{equation}
For $n\ge 1$, and for distinct vertices $u,v\in I_n$, we write $u\sibl v$ if these vertices share a common parent in $I_{n-1}$ (that is, $u$ and $v$ are siblings).
More precisely,
\begin{equation} \label{eq:sibldefn}
u\sibl v \quad \text{if and only if} \quad d_{G_n}(u,v)=2.
\end{equation}
For $1\le k \le n$, write 
\begin{equation} \label{eq:Jnkdefn}
J_n^{(k)}:= \{u\in I_n:|\alpha_k(u)|=1\}
\quad \text{and}\quad
\tilde J_n^{(k)}:= \{u\in \tilde I_n:|\tilde \alpha_k(u)|=1\}.
\end{equation}
(For example, $J^{(1)}_n$ is the set of individuals in $I_{n}$ which have exactly one parent.)

\subsection{Proof of \refP{prop:Z4}}

\begin{proof}[Proof of  \refP{prop:Z4}]
We begin by proving~\eqref{eq:Z4}.
For $n\ge 1$ and $u\in I_{n-1}$, by our construction in Section~\ref{subsec:UlamHarris} and by~\eqref{eq:kudefn} we have 
$\E{|\cC_u||\cF_{n-1}}=(1+p)(1-q)^{k_u}.$ For $n\ge 4$, 
for each $u\in I_{n-1}$ we have $\alpha_3(u)\neq \emptyset$, and so we must have
$k_u\ge 1$. For $n\ge 1$, since $Z_n =|I_n|\le |\tilde I_n|$ and $(1+p)(1-q)\le 1+p-q$, it follows that 
\begin{align*}
    \E{Z_n}\le \E{|\tilde I_n|} =\E{\sum_{u\in I_{n-1}} |\cC_u|}\le (1+p-q\I{n\ge 4})\E{Z_{n-1}}, 
\end{align*}
which establishes~\eqref{eq:Z4}. 
By our construction of $\cG'(p,q)$ in Section~\ref{subsec:UlamHarris}, for $n\ge 1$ and $u\in I'_{n-1}$ we have $\E{|\cC'_u||\cF_{n-1}}= (1+p)(1-q)^{|\mathcal K'_u|+\I{n\le 3}}$, and $|\mathcal K'_u|\ge 1$ for $n\ge 4$. Hence for $n\ge 1$,
\begin{align*}
    \E{Z'_n}\le \E{\sum_{u\in I'_{n-1}} |\cC '_u|}\le (1+p)(1-q)\E{Z'_{n-1}}\le (1+p-q)\E{Z'_{n-1}}; 
\end{align*}
in particular, \eqref{eq:Z4} holds when $Z_m$ is replaced by $Z'_m$ for each $m$.

It remains to prove~\eqref{eq:Z3}; suppose $0\le q \le p\le 1$.
Define a random subset $I^* \subset \mathcal U$ by letting
$$
I^* = \{\emptyset\} \cup \{u =u_1 \ldots u_k \in \mathcal U: 1\le k\le 6, \, u_i \le \xi_{u|_{i-1}} \forall 1\le i\le k\}.
$$
Note that $I^*$ contains all the possible labels in $I_0, \ldots, I_6$ and $I'_0, \ldots, I'_6$.
For $0\le n \le 6$, let $W_n =|I^* \cap \mathcal U_n|$. Then
by our construction in Section~\ref{subsec:UlamHarris}, $Z_n \le W_n$ and $Z'_n \le W_n$ for $n=0,1,\ldots, 6$.
Define the event
\begin{align*}
    A = \{\delta_{u,v}=0 \; \forall u,v \in I^*, \, \mu_{\{u,v\}}=0 \; \forall u\neq v \in I^*, \, \delta_u =0 \; \forall u \in I^*
    \}.
\end{align*}
Then on the event $A$, no vertices are deleted or merged in the first six generations of the construction of $\mathcal G(p,q)$ and $\mathcal G'(p,q)$, and so in particular,
$Z_i =Z'_i = W_i$ for $i=0,1,\ldots, 6$. We now have that for $0\le n \le 6,$
\begin{equation} \label{eq:Z123bd}
\E{W_n \I{A}}\le \E{Z_n}\le \E{W_n} \quad \text{and}\quad \E{W_n \I{A}}\le \E{Z'_n}\le \E{W_n}.
\end{equation}
Since $W_n$ is the number of individuals in the $n$th generation of a Galton-Watson process with offspring distribution given by a Poisson distribution with mean $1+p$, we have $\E{W_n}=(1+p)^n$ for $0\le n \le 6$.
Moreover, by conditioning on $(\xi_u)_{u\in \mathcal U}$, and then since $|I^*|=\sum_{k=0}^6 W_k$,
\begin{align} \label{eq:propZ3star}
    \E{W_n \I{A^c}} \le \E{W_n \cdot (2 |I^*|^2+ |I^*| ) q}\le 3q \E{W_n \Big( \sum_{k=0}^6 W_k \Big)^2}.
\end{align}
By the Cauchy-Schwarz inequality, and then by Jensen's inequality and Cauchy-Schwarz again,
for $0\le k_1,k_2 \le 6$,
\begin{align} \label{eq:Wmixedmoments}
    \E{W_n W_{k_1} W_{k_2}}\le \E{W_n^2}^{1/2}\E{W_{k_1}^2W^2_{k_2}}^{1/2}
    &\le \E{W_n^4}^{1/4}\E{W_{k_1}^4}^{1/4}\E{W_{k_2}^4}^{1/4} \notag \\
    &\le \Big(\max_{0\le k \le 6}\E{W_k^4}\Big)^{3/4}.
\end{align}
Let $(\zeta_i)_{i\ge 1}$ be i.i.d.~Poisson random variables with mean $2$.
For $1\le k \le 6$, since $p\le 1$,
$$
\E{W_k^4}\le \E{\Big( \sum_{i=1}^{W_{k-1}}\zeta_i\Big)^4}=\E{\sum_{1\le i_1,i_2,i_3,i_4\le W_{k-1}}\E{\zeta_{i_1}\zeta_{i_2}\zeta_{i_3}\zeta_{i_4}}}\le \E{\zeta_1^4}\E{W_{k-1}^4},
$$
where the last inequality follows since by Jensen's inequality, $\E{\zeta_i^r}\le \E{\zeta_i^4}^{r/4}$ for $1\le r \le 3$.
Therefore, since $W_0=1$, there exists a constant $C'<\infty$ such that 
$\max_{0\le k \le 6}\E{W_k^4}\le C'$.
It follows from~\eqref{eq:propZ3star} and~\eqref{eq:Wmixedmoments} that there exists a constant $C''<\infty$ such that 
\begin{align*}
    \E{W_n \I{A^c}} \le  C'' q  \quad \forall\; 0\le n \le 6.
\end{align*}
Hence since $q\le p$, by~\eqref{eq:Z123bd} we now have that for $0\le n \le 6,$
\begin{equation*} 
(1+p)^n- C'' p \le \E{Z_n}\le (1+p)^n \quad \text{and}\quad (1+p)^n- C'' p\le \E{Z'_n}\le (1+p)^n,
\end{equation*}
and the result follows by taking $C_1>0$ sufficiently large.
\end{proof}

\subsection{Proof of \refP{prop:superGW}}

Take $k,n\ge 0$ and $u\in I_k$. Recall the definitions of $I_n^u$ and $\tilde I_{n+1}^u$ in~\eqref{eq:Inudefn} and~\eqref{eq:tildeIudefn}. 
Let $V_n^u:=\cup_{m\le n} I_m^u$ and let 
\begin{align*}
    E_n^u:=\cup_{1\le m\le n} \{\{v,w\}:\, v\in \alpha_1(w), v\in I_{m-1}^u \}=E_{n+k}\cap (V^u_n \times V^u_n);
\end{align*}
define the graph $G_n^u:=(V_n^u, E_n^u)$ and the graph process $\cG^u:=(G_n^u)_{n\ge 0}$. For $n\ge 0$, let 
\begin{align*}
    \cU_n^u:=\{uv:\, v\in \cU_n\}.
\end{align*}
Recall the definition of a renewed vertex from Section~\ref{sec:outline}.
For $k\ge 3$, we write $\cR_k$ for the set of renewed vertices in generation $k$, i.e. 
\begin{equation} \label{eq:Rkdefn}
    \cR_k:=\{v\in I_k: |\{w\in V_k: 1\le d_{G_k}(v,w)\le 3\}|=3 \}.
\end{equation}
Note that since $|\alpha_i(v)|\ge 1$ for each $v\in I_k$ and $i\le k$, we have that if $v\in \cR_k$ then $\{w\in V_k: 1\le d_{G_k}(v,w)\le 3\}=\cup_{i=1}^3 \alpha_i(v)$.
The next lemma says that the set of descendants in generation $n+k$ of a renewed vertex $u$ in generation $k$ is contained in $\cU^u_n$. 
Recall the definition of $\sigma(v)$ in~\eqref{eq:sigmavdefn}.
\begin{lem}\label{lem:Iu}
Let $k\ge 3$ and $u\in \cR_k$. The following holds for $n\ge 1$, for all $v\in I_n^u$:
\begin{enumerate}
    \item[\it i)] $\sigma(v)\subseteq \tilde{I}_{n}^u$,
    \item[\it ii)] $\alpha_i(v)\subseteq I_{n-i}^u$ for each $0\le i\le n$.
\end{enumerate}
Moreover,  $I_n^u\subseteq \cU_n^u$ for each $n\ge 0$,
and 
for $m\ge k$, $v\in V^u_{m-k}$ and $v' \in V_m \setminus V_{m-k}^u$,
any path in $G_m$ from $v$ to $v'$ passes through $u$.
\end{lem}

\begin{proof}
First observe that $I_0^u=\{u\}$. We will prove {\it i)} and {\it ii)} by induction on $n\ge 1$. Let $v\in I_1^u$. We have that $u\in \alpha_1(v)$ and so, by our construction in Section~\ref{subsec:UlamHarris}, $ui\in \sigma(v)$ for some $i\ge 1$. Suppose that $wj\stackrel{m}{\sim} ui$ for some $w\in I_k$ and $j\ge 1$; then $d_{G_k}(w,u)=2$, which contradicts the assumption that $u\in \cR_k$. It follows that $\sigma(v)=\{ui\}$, and so $v=ui$ and $\alpha_1(v)=u$, which implies conditions {\it i)} and {\it ii)} hold for $n=1$. 

We next prove the induction step. Let $m>1$ and suppose conditions {\it i)} and {\it ii)} hold for $n= m-1$. Then suppose, aiming for a contradiction, that condition {\it i)} is not satisfied for $n=m$. Take $v\in I_m^u$ such that there exists $w_-j_-\in \sigma(v)$ with $w_-\in I_{k+m-1}\setminus I_{m-1}^u$ and $j_-\in \N$. Since $v\in I_m^u$, there also exists $w_+j_+\in \sigma(v)$ with $w_+\in I_{m-1}^u$ and $j_+\in \N$. 
Then by the definition of $\sigma(v)$, there exist $s\ge 1$ and $w_+j_+=v_0, v_1,\ldots, v_s=w_-j_-\in \tilde{I}_{k+m}$ such that $v_\ell \stackrel{m}{\sim} v_{\ell+1}$ for each $\ell=0,1,\ldots, s-1$. In particular, there exists $0\le \ell<s$ such that $v_\ell=wj$, $v_{\ell+1}=w' j'$ with $j,j'\in \N$, $w \in I_{m-1}^u$ and $w' \in I_{k+m-1}\setminus I_{m-1}^u$. Using that $v_\ell \stackrel{m}{\sim} v_{\ell+1}$, we infer that $d_{G_{k+m-1}}(w,w')=2$ and so $\alpha_1(w)\cap \alpha_1(w')\neq \emptyset$. By the induction hypothesis, $\alpha_1(w)\subseteq I^u_{m-2}$, and so there exists $z\in I^u_{m-2}\cap \alpha_1(w')$. This contradicts the assumption that $w'\notin I_{m-1}^u$. Therefore, for all $v\in I_m^u$ we must have $\sigma(v)\subseteq \tilde{I}_{m}^u$. By our construction in Section~\ref{subsec:UlamHarris}, it follows that $\alpha_1(v)\subseteq I_{m-1}^u$; by our assumption that condition {\it ii)} holds for $n= m-1$, it then follows that
$\alpha_i(v)\subseteq I_{m-i}^u$ for each $0\le i\le m$.
This completes the proof that conditions {\it i)} and {\it ii)} hold for $n= m$,
and so completes the induction argument.

Finally, by the definition of $\tilde I^u_{n+1}$ in~\eqref{eq:tildeIudefn}, it is straightforward to see that if $I^u_n\subseteq \cU_n^u$ then $\tilde I^u_{n+1}\subseteq \cU_{n+1}^u$.
Hence, by induction, condition {\it i)} implies that $I_n^u\subseteq \cU_n^u$ for each $n\ge 0$. 
Moreover, for $m\ge k$, by condition {\it ii)} with $i=1$, 
and since for $n\ge 1$, $v\in I^u_n$ if $\alpha_1(v)\cap I^u_{n-1}\neq \emptyset$,
there is no edge in $G_m$ between a vertex in $V^u_{m-k}\setminus \{u\}$ and a vertex in $V_m \setminus V^u_{m-k}$. The last claim of the lemma follows.
\end{proof}

Let $u\in \cR_m$ for some $m\ge 3$. By \refL{lem:Iu}, for any $n\ge 1$, removing vertex $u$ from $G_{m+n}$ disconnects the graph $G^u_n$ from the rest of the graph; therefore, we may refer to $\cG^u$ as a graph process that \emph{restarts} at $u$. The next lemma specifies, conditional on $\cF_m$, the distribution of $\cG^u$.

\begin{lem}\label{lem:distGu}
For $m\ge 3$, conditional on $\cF_m$, the graph processes $(\cG^u)_{u\in \cR_m}$ (ignoring the Ulam-Harris labels of vertices) form a collection of i.i.d. processes with $\cG^u\stackrel{d}{=} \cG'(p,q)$ for each $u\in \mathcal R_m$.
\end{lem}

\begin{proof}
For $u\in \cR_m$, we write $u_i$, $0\le i\le 3$, for the unique ancestor of $u$ in $I_{m-i}$; that is, $\alpha_i(u)=\{u_i\}$.
Recall the definition of $\cF_m$ in~\eqref{dfn:Fn}.
We claim that for any $u\in \cR_m$, the graph process $\cG^u$ can be constructed using the following collection of random variables (which are independent of $\cF_m$):
\begin{equation} \label{coll_rvs}
    (\xi_{uv})_{v\in \cU}, (\delta_{uv,uv'})_{v,v'\in \cU \setminus \{\emptyset\}}, (\mu_{\{uv,uv'\}})_{v\neq v'\in \cU\setminus\{\emptyset\}} \text{ and } (\delta_{u_i,uv})_{v\in \cU\setminus\{\emptyset\},i\in \{0,1,2,3\}}.
\end{equation}
Note that conditional on $\cF_m$, for distinct vertices in $\cR_m$, the collections of random variables corresponding to these vertices are independent. 
Hence it follows from the claim that conditional on $\cF_m$, the graph processes $(\cG^u)_{u\in \cR_m}$ are independent.

To see that the claim holds, suppose that $G^u_n$ has been constructed for some $n\ge 0$.
(Note that $G^u_0=(\{u\},\emptyset)$.)
By Lemma~\ref{lem:Iu}, we have $I^u_n \subseteq \cU^u_n$.
Moreover, for $v\in I^u_n$, if $d_{G_{m+n}}(v,w)=3$ for some $w\in V_{m+n}\setminus (V^u_n \setminus\{u\})$ then by Lemma~\ref{lem:Iu},
and by the definition of $\cR_m$,
we must have $n\le 3$ and $w=u_{3-n}$.
If instead $d_{G_{m+n}}(v,w)=3$ for some $w\in V_n^u
\setminus \{u\}$, then since removing the vertex $u$ from $G_{m+n}$ disconnects the graph $G^u_n$ from the rest of the graph, we must have $d_{G^u_n}(v,w)=3$.
Hence
$$
\cK_v =\begin{cases} \{w\in V^u_n \setminus \{u\} :d_{G^u_n}(v,w)=3\}\cup \{u_{3-n}\} \quad &\text{if }n\le 3\\
\{w\in V^u_n \setminus \{u\} :d_{G^u_n}(v,w)=3\} \quad &\text{otherwise,}
\end{cases}
$$
and $\cC_v=\{j\le \xi_v : \delta_{w,vj}=0 \; \forall w\in \cK_v\}$ can be constructed using the collection of random variables in~\eqref{coll_rvs}.

Let $\tilde G^u_{n+1}:=(\tilde V^u_{n+1},\tilde E^u_{n+1})$, where $\tilde V^u_{n+1}:=V^u_n \cup \{vi:v\in I^u_n, i \in \cC_v\}$ and $$\tilde E^u_{n+1}:=E^u_n \cup \{\{v,vi\}:v\in I^u_n, i\in \mathcal C_v\}.$$
Take $vi \in \tilde I^u_{n+1}$ with $v\in I^u_n$ and $i\in \N$, and $v'i'\in \tilde I_{m+n+1}$ with
$v'\in  I_{m+n}$ and $i'\in \N$ such that
 $d_{\tilde{G}_{m+n+1}}(vi,v'i')=4$.
 Then $d_{G_{m+n}}(v,v')=2$, and so by the definition of $\cR_m$ we must have $n\ge 1$. Moreover,
 by {\it{ii)}} in Lemma~\ref{lem:Iu} we must have  $v'\in  I^u_{n}$ and $d_{\tilde G^u_{n+1}}(vi,v'i')=4$. 
 Therefore, for $vi\in \tilde I^u_{n+1}$ and $v'i' \in \tilde I_{m+n+1}$ we have that $vi \stackrel{m}{\sim}v'i'$ if and only if $v'i'\in \tilde I^u_{n+1}$, $d_{\tilde G^u_{n+1}}(vi,v'i')=4$ and $\mu_{\{vi,v'i'\}}=1$.
 Hence $G^u_{n+1}$ can be constructed from $G^u_n$ using the collection of random variables in~\eqref{coll_rvs}.
This completes the proof of the claim.

From the description of the construction above, and the construction of $\cG'(p,q)$ in Section~\ref{subsec:UlamHarris}, it is not difficult to verify that conditional on $\cF_m$, $\cG^u\stackrel{d}{=} \cG'(p,q)$.
\end{proof}

\begin{proof}[Proof of \refP{prop:superGW}]
Fix $N\ge 3$; then define a sequence $(\cX_n)_{n\ge 1}$ of sets of renewed vertices iteratively as follows.
Let $\cX_1:=\cR_N$; for $n\ge 1$ let
\begin{align*}
    \cX_{n+1}:=\{v\in \cR_{(n+1)N}: \alpha_N(v)\subseteq \cX_n\},
\end{align*}
the set of renewed vertices in generation $(n+1)N$ whose generation-$nN$ ancestors are in $\cX_n$.
For each $n\ge 1$, let $X_n := |\cX_n|$.

Note that for each $v\in I_{(n+1)N}$ with $\alpha_N(v)\subseteq \cX_n \subseteq \cR_{nN}$, we have $v\in I^u_N$ for some $u\in \cR_{nN}$ and so
by property {\it ii)} of \refL{lem:Iu}, we have $\alpha_N(v)=\{u\}$.
Therefore for each $n\ge 1$, we can write 
$$
X_{n+1}= \sum_{u\in \cX_{n}} X^u_{n+1},
\quad \text{where } X^u_{n+1} = |\{v\in \cR_{(n+1)N}: \alpha_N(v)=\{u\}\}|.
$$
Moreover, by Lemma~\ref{lem:distGu}, conditional on $\cF_{nN}$, $(X^u_{n+1})_{u \in \cX_n}$ are i.i.d.~with the same distribution as $R'_N$.
Since $I_{nN}\supseteq \cX_n$ and so $Z_{nN}\ge X_n$ for each $n\ge 1$, the result follows.
\end{proof}

\section{Merger stage and upper bounds}\label{sec:upper}
From now on, we assume $p,q\in [0,1]$.
In this section, we establish upper and lower bounds on $Z_n$, and then prove upper bounds on the expectations of some relevant quantities.
For $n\ge 2$, let $M_n$ denote the number of mergers that occur when $I_n$ is constructed from $\tilde I_n$, i.e.~let

\begin{equation} \label{eq:Mndefn}
M_n := |\{ \{u,v\}: u,v \in \tilde I_n, u\stackrel{m}{\sim} v\}|.
\end{equation}
Recall the definition of $\tilde J^{(2)}_n$ in~\eqref{eq:Jnkdefn}.
For $n\ge 2$, let
\begin{align}
    \widetilde M_n &:= |\{ \{u,v\}: u,v \in \tilde J_n^{(2)}, u\stackrel{m}{\sim} v\}| \label{eq:tildeMdefn} \\
    \text{and }\quad \widetilde L_n &:=|\{(u,v,w)\in \tilde J_n^{(2)}\times \tilde J_n^{(2)} \times \tilde J_n^{(2)}:u\stackrel{m}{\sim} v,v\stackrel{m}{\sim} w, u\neq w\}|. \label{eq:tildeLdefn}
\end{align}
For $n\ge 1$, let
\begin{equation} \label{eq:Yndefn}
Y_n :=\sum_{u\in I_{n-1}}|\mathcal C_u |=|\tilde I_n |.
\end{equation}

\begin{lem}\label{lem:Z}
For all $n\ge 2$,
\begin{align}\label{eq:Z}
    \widetilde M_n-\widetilde L_n \le Y_n-Z_n \le M_n.
\end{align}
\end{lem}

\begin{proof}
For $u\in I_n$, recall from~\eqref{eq:sigmavdefn} that we let
$\sigma(u):=\{v\in \tilde I_n :u\sim v\}$, the equivalence class of offspring that merges into the single individual $u$. 
Let
$$
M_{n,u}:=|\{ \{v,v'\}: v,v' \in \sigma(u), v\stackrel{m}{\sim} v'\}|,
$$
the number of mergers within the equivalence class $\sigma (u)$.
Then for each $u\in I_n$, by the definition of the equivalence relation $\sim$ in Section~\ref{subsec:UlamHarris} we have $M_{n,u}\ge |\sigma(u)|-1$, and so
$$
M_n =\sum_{u\in I_n}M_{n,u}\ge \sum_{u\in I_n}|\sigma(u)|-|I_n|=Y_n-Z_n,
$$
by the definition of $Y_n$ in~\eqref{eq:Yndefn},
which establishes the second inequality in~\eqref{eq:Z}.

For the first inequality in~\eqref{eq:Z}, for $u\in I_n$, let
$$
\widetilde L_{n,u} :=|\{(v_1,v_2,v_3):v_i \in \sigma(u) \cap \tilde J_n^{(2)}\; \forall 1\le i\le 3,\, v_1\stackrel{m}{\sim} v_2,v_2\stackrel{m}{\sim} v_3, v_1 \neq v_3\}|,
$$
and let 
$$
\widetilde M_{n,u}:=|\{ \{v,v'\}: v,v' \in \sigma(u)\cap \tilde J^{(2)}_n, v\stackrel{m}{\sim} v'\}|.
$$
If $|\sigma(u)\cap \tilde{J}_n^{(2)}|=0$, then $\widetilde M_{n,u}-\widetilde L_{n,u}=0\le |\sigma(u)|-1$.
Suppose instead that $|\sigma(u)\cap \tilde{J}_n^{(2)}|> 0$.
Then consider the graph $G^u=(V^u,E^u)$, where
\begin{align*}
    V^u &= \sigma(u) \cap \tilde J^{(2)}_n\\
    \text{and } \quad E^u &= \{ \{v,v'\}:v,v' \in V^u , v\stackrel{m}{\sim} v'\}.
\end{align*}
Note that $G^u$ contains exactly $\widetilde L_{n,u}$ directed paths of length two.
It is a straightforward fact that a simple, non-empty graph with $t$ vertices, $m$ edges and $\ell$ directed paths of length two satisfies $m-\ell<t$ (see \refL{lem:ML} in the Appendix). Thus, $\widetilde M_{n,u}- \widetilde L_{n,u}\le |\sigma(u)\cap \tilde{J}_n^{(2)}|-1\le |\sigma(u)|-1.$
It follows that 
$$
\widetilde M_n -\widetilde L_n =\sum_{u\in I_n}(\widetilde M_{n,u}-\widetilde L_{n,u})\le \sum_{u\in I_n}|\sigma(u)|-|I_n|= Y_n-Z_n,
$$
as required.
\end{proof}
Note that we could have proved by the same argument that $Y_n-Z_n \ge M_n-L_n$, where 
$L_n := |\{(u,v,w)\in \tilde I_n \times \tilde I_n \times \tilde I_n: u\stackrel{m}{\sim}v, v\stackrel{m}{\sim}w, u\neq w\}|, $ but it will turn out that proving a suitable lower bound on $\E{\widetilde M_n -\widetilde L_n}$ is easier than proving such a bound on $\E{M_n-L_n}$.

\subsection{Preliminary results} \label{subsec:prelimNn}

Let $j_1 , \ldots , j_k$ be non-negative integers. For $n\ge 0$ and a vertex $u\in I_n$, let
\begin{equation} \label{eq:Nsetdefn}
    \cN^{(j_1,\ldots, j_k)}(u):=\{(v_1,\ldots, v_k)\in \prod_{i=1}^k I_{n+j_i}:\,  u\in \alpha(v_i), \alpha(v_i)\cap \alpha(v_{i'})\subseteq \cup_{m\le n}I_m 
    \text{ for all } i\neq i'\},
\end{equation}
the set of ordered $k$-tuples of vertices $(v_1,\ldots, v_k)$ with $v_i\in I_{n+j_i}$ for each $i$ which do not have any pairwise ancestors more recently than their common ancestor $u$.
Let  
\begin{align} \label{eq:Nnumdefns}
N^{(j_1,\ldots,j_k)}(u):=|\cN^{(j_1,\ldots, j_k)}(u)|
\quad \text{and}\quad
    N_n^{(j_1,\ldots,j_k)}:=\sum_{u\in I_n} N^{(j_1,\ldots,j_k)}(u).
\end{align}
The following simple upper bound on the expectation of $N_n^{(j_1,\ldots, j_k)}$ will be used several times in the rest of the paper.
\begin{lem}\label{lem:EN}
Let $j_1,\ldots, j_k$ be non-negative integers. For any $n\ge 0$,  
\begin{align*}
    \E{N_n^{(j_1,\ldots,j_k)}}\le (1+p)^{\sum_{i=1}^k j_i} \E{Z_n}. 
\end{align*}
\end{lem}
\begin{proof}
Without loss of generality, we may assume $j_1 \ge \ldots \ge j_k\ge 0.$
Note first that it suffices to prove the uniform bound  
\begin{align}\label{eq:ENu}
    \E{N^{(j_1,\ldots,j_k)}(u)\Big|\mathcal F_n}\le (1+p)^{\sum_{i=1}^k j_i}
\end{align}
for any $n\ge 0$ and $u\in I_n$. 
Indeed, assuming~\eqref{eq:ENu}, 
\begin{align*}
\E{N^{(j_1,\ldots,j_k)}_n} =\E{\E{N^{(j_1,\ldots,j_k)}_n \Big|\mathcal F_n}}
&= \E{\sum_{u \in I_n}\E{N^{(j_1,\ldots,j_k)}(u)\Big|\mathcal F_n}}\\
&\le (1+p)^{\sum_{i=1}^k j_i}\E{Z_n},
\end{align*}
as required.
We will prove \eqref{eq:ENu} by induction on $j_1=\max_\ell\{j_\ell\}$.

We begin by considering the case $j_k=0$; note that by the definition of $\mathcal N^{(j_1,\ldots,j_k)}(u)$ in~\eqref{eq:Nsetdefn}, in this case we have
$
N^{(j_1,\ldots,j_k)}(u)=N^{(j_1,\ldots,j_{k-1})}(u),
$
and $N^{(0)}(u)=1$.
Hence from now on we may assume that $j_k\ge 1$.

Let us now consider the case $j_1=j_k=1$, where $N^{(1,\ldots,1)}(u)$ counts the number of ordered $k$-tuples of distinct offspring of $u$. 
By~\eqref{eq:Inudefn} and~\eqref{eq:offspringbound}, the number of offspring of $u$ is
$|I^u_1|\le \xi_u.$
Since $\xi_u$ has Poisson distribution with mean $1+p$, and $\xi_u$ is independent of $\mathcal F_n$, we have 
\begin{align*}
    \E{\left. N^{(1,\ldots,1)}(u)\right|\mathcal F_n}\le \E{(\xi_u)_k}= (1+p)^k.
\end{align*}
(Here, we write $(x)_k:=x(x-1)\ldots (x-k+1)$ for the $k$th falling factorial.)

For the induction step, suppose that $j_1\ge 2$ and let $\ell:=\max\{i:\, j_i=j_1\}$. Set $j'_i:=j_i-\I{i\le \ell}\ge 1$ for $i\le k$. 
Then for each $(v_1,\ldots, v_k) \in \cN^{(j_1,\ldots, j_k)}(u)$, 
there exists
$(w_1,\ldots, w_k)\in \cN^{(j'_1,\ldots, j'_k)}(u)$
such that $w_i=v_i$ $\forall i>\ell$ and $w_i\in \alpha_1(v_i)$ $\forall i\le \ell$. 
Moreover, for each $(w_1,\ldots, w_k)\in \cN^{(j'_1,\ldots, j'_k)}(u)$,
by~\eqref{eq:offspringbound} again,
$$
|\{(v_1,\ldots, v_k) \in \cN^{(j_1,\ldots, j_k)}(u):w_i=v_i\;\forall i>\ell, w_i\in \alpha_1(v_i) \; \forall i\le \ell \}|
\le \prod_{i=1}^\ell \xi_{w_i}.
$$
It follows that 
\begin{align*}
    \E{\left. N^{(j_1,\ldots,j_k)}(u)\right|\cF_{n+j_1-1}}
    &\le \sum_{(w_1,\ldots,w_k)\in\cN^{(j'_1,\ldots, j'_k)}(u)} 
    \E{\left. \prod_{i=1}^\ell \xi_{w_i} \right| \mathcal F_{n+j_1-1}}\\
    &= N^{(j'_1,\ldots,j'_k)}(u)(1+p)^{\ell},
\end{align*}
since $w_1,\ldots ,w_k\in I_{n+j_1-1}$ are distinct for each $(w_1,\ldots,w_k)\in\cN^{(j'_1,\ldots, j'_k)}(u)$.
By taking expectations, this establishes the induction step for \eqref{eq:ENu}. 
\end{proof}

We now define a similar quantity to $N_n^{(j_1,\ldots,j_k)}$ which will also be used several times later in the paper.
Let $j_1 , \ldots , j_k$ be non-negative integers. For $n\ge 0$ and a vertex $u\in I_n$, let
\begin{equation} \label{eq:barNsetdefn}
    \overline{\cN}^{(j_1,\ldots, j_k)}(u):=\{(v_1,\ldots, v_k)\in \prod_{i=1}^k I_{n+j_i}:\,  u\in \alpha(v_i) 
    \text{ for all } i\},
\end{equation}
the set of ordered $k$-tuples of vertices $(v_1,\ldots, v_k)$ with $v_i\in I_{n+j_i}$ for each $i$ and with common ancestor $u$.
Let  
\begin{align} \label{eq:barNnumdefn}
\overline{N}^{(j_1,\ldots,j_k)}(u):=|\overline{\cN}^{(j_1,\ldots, j_k)}(u)|
\quad \text{and}\quad 
    \overline{N}_n^{(j_1,\ldots,j_k)}:=\sum_{u\in I_n} \overline{N}^{(j_1,\ldots,j_k)}(u).
\end{align}
We can prove a similar upper bound to Lemma~\ref{lem:EN} for the expectation of $\overline{N}_n^{(j_1,\ldots,j_k)}$.
\begin{lem}\label{lem:ENbar}
Let $j_1,\ldots, j_k$ be non-negative integers with $m = \max_{i}\{j_i\}$. Let $\xi$ be a Poisson random variable with mean $1+p$. For any $n\ge 0$,  
\begin{align*}
    \E{\overline{N}_n^{(j_1,\ldots,j_k)}}\le \E{\xi^{k}}^m \E{Z_n}. 
\end{align*}
\end{lem}
\begin{proof}
Without loss of generality, we may assume $j_1 \ge \ldots \ge j_k\ge 0.$
By the same argument as in the proof of \refL{lem:EN}, it suffices to prove the uniform bound  
\begin{align}\label{eq:ENubar}
    \E{\overline{N}^{(j_1,\ldots,j_k)}(u)\Big|\mathcal F_n}\le \E{\xi^{k}}^m
\end{align}
for any $n\ge 0$ and $u\in I_n$. 
We will prove \eqref{eq:ENubar} by induction on $j_1=\max_\ell\{j_\ell\}$.

In the case $j_k=0$, by the definition of $\overline{\mathcal N}^{(j_1,\ldots,j_k)}(u)$ in~\eqref{eq:barNsetdefn}, we have 
$
\overline{N}^{(j_1,\ldots,j_k)}(u)=\overline{N}^{(j_1,\ldots,j_{k-1})}(u),
$
and $\overline{N}^{(0)}(u)=1$.
From now on we will assume that $j_k\ge 1$.

We now consider the case $j_1=j_k=1$, where $\overline{N}^{(1,\ldots,1)}(u)$ counts the number of ordered $k$-tuples of (not necessarily distinct) offspring of $u$. 
By~\eqref{eq:Inudefn} and~\eqref{eq:offspringbound},
the number of offspring of $u$ is 
$|I^u_1|\le \xi_u,$
and, conditional on $\mathcal F_n$, $\xi_u \stackrel{d}{=} \xi$. Therefore
\begin{align*}
    \E{\left. \overline{N}^{(1,\ldots,1)}(u)\right|\mathcal F_n}\le \E{\left. (\xi_u)^k \right| \cF_n }= \E{\xi^k}.
\end{align*}

As in the proof of Lemma~\ref{lem:EN},
for the induction step, suppose that $j_1\ge 2$ and let $\ell=\max\{i:\, j_i=j_1\}$. Set $j'_i=j_i-\I{i\le \ell}\ge 1$ for $i\le k$. 
Then for each $(v_1,\ldots, v_k) \in \overline{\cN}^{(j_1,\ldots, j_k)}(u)$, 
there exists
$(w_1,\ldots, w_k)\in \overline{\cN}^{(j'_1,\ldots, j'_k)}(u)$
such that $w_i=v_i$ $\forall i>\ell$ and $w_i\in \alpha_1(v_i)$ $\forall i\le \ell$. 
Moreover, for each $(w_1,\ldots, w_k)\in \overline{\cN}^{(j'_1,\ldots, j'_k)}(u)$, by~\eqref{eq:offspringbound},
$$
|\{(v_1,\ldots, v_k) \in \overline{\cN}^{(j_1,\ldots, j_k)}(u):w_i=v_i\;\forall i>\ell, w_i\in \alpha_1(v_i) \; \forall i\le \ell \}|
\le \prod_{i=1}^\ell \xi_{w_i}.
$$
Since $j_1',\ldots , j_k'\le j_1-1$, it follows that 
\begin{equation} \label{eq:barNcalc}
    \E{\left. \overline{N}^{(j_1,\ldots,j_k)}(u)\right|\cF_{n+j_1-1}}
    \le \sum_{(w_1,\ldots,w_k)\in \overline{\cN}^{(j'_1,\ldots, j'_k)}(u)} 
    \E{\left. \prod_{i=1}^\ell \xi_{w_i} \right| \mathcal F_{n+j_1-1}}.
\end{equation}
(Note that for $(w_1,\ldots, w_k) \in \overline{\cN}^{(j'_1,\ldots, j'_k)}(u)$, $w_1,\ldots, w_k$ need not be distinct.) By Jensen's inequality, we have
$$
\E{\xi^{\ell'}}\le \E{\xi^\ell}^{\ell'/\ell} \; \text{for each }\ell ' \le \ell,
$$
and therefore, for $w_1,\ldots , w_\ell \in I_{n+j_1-1}$,
\begin{equation} \label{eq:useJensen}
\E{\left. \prod_{i=1}^\ell \xi_{w_i} \right| \mathcal F_{n+j_1-1}}\le
\E{\xi^\ell}\le \E{\xi^k}^{\ell/k}\le \E{\xi^k},
\end{equation}
where the second inequality follows by Jensen's inequality and since $\ell\le k$, and the last inequality follows since $\E{\xi^k}\ge (1+p)^k\ge 1$.
By substituting into~\eqref{eq:barNcalc} and taking expectations, this establishes the induction step for \eqref{eq:ENubar}. 
\end{proof}

\subsection{Upper bounds} 
We now use the results in Section~\ref{subsec:prelimNn} to prove some useful upper bounds on expectations of the random variables in Lemma~\ref{lem:Z}.
\begin{lem}\label{lem:EML'}
For $n\ge 2$, 
\begin{align}
    \E{M_n}&\le \frac{q}{2}(1+p)^4\E{Z_{n-2}} \label{eq:EM}\\
    \text{and }\quad \E{\widetilde L_n}&\le 4q^2(1+p)^6 \E{Z_{n-2}}. \label{eq:EL'}
\end{align}
\end{lem}

\begin{proof}
By the definition of $M_n$ in~\eqref{eq:Mndefn} and our construction in Section~\ref{subsec:UlamHarris}, and then since $\E{\left. \mu_{\{u,v\}}\right|\tilde {\cF}_n}=q$ for each $u\neq v \in \mathcal U_n$,
\begin{align*}
    \E{M_n} &= \E{\sum_{\{\{u,v\}:u,v\in \tilde I_n, d_{\tilde G_n}(u,v)=4\}}\mu_{\{u,v\}}} \\
    &= q\E{|\{\{u,v\}: u,v \in \tilde I_n , d_{\tilde G_n}(u,v)=4\}|}.
    \end{align*}
By the construction of $\tilde G_n$ in Section~\ref{subsec:UlamHarris} (and recalling the definition of $\sibl$ in~\eqref{eq:sibldefn}), it follows that
    \begin{align*}
    \E{M_n} &=q \E{|\{\{u,v\}: u,v \in \tilde I_n , u|_{n-1}\sibl v|_{n-1} \}|}\\
    &=\tfrac 12 q \E{\sum_{\{(u',v')\in I_{n-1}\times I_{n-1}:u'\sibl v'\}}
    |\mathcal C_{u'}||\mathcal C_{v'}|}.
\end{align*}
Then since for each $u\in I_{n-1}$, $|\mathcal C_u|\le \xi_u$, and since $(\xi_u)_{u\in \mathcal U_{n-1}}$ is independent of $\mathcal F_{n-1}$,
it follows that
\begin{align*}
    \E{M_n} &\le \tfrac 12 q \E{\sum_{\{(u',v')\in I_{n-1}\times I_{n-1}:u'\sibl v'\}}
    \E{ \xi_{u'}\xi_{v'}\Big| \cF_{n-1}}}\\
    &= \tfrac 12 q(1+p)^2 \E{|\{(u',v')\in I_{n-1}\times I_{n-1}:u'\sibl v'\}|}\\
    &\le \tfrac 12 q(1+p)^2 \E{N^{(1,1)}_{n-2}}\\
    &\le \tfrac 12 q(1+p)^4 \E{Z_{n-2}},
\end{align*}
where the third line follows by the definition of $N_n^{(j_1,\ldots, j_k)}$ in~\eqref{eq:Nnumdefns}, and the last inequality follows
by \refL{lem:EN}.

Now recall the definitions of $\widetilde L_n$ in~\eqref{eq:tildeLdefn} and $\tilde J^{(2)}_n$ in~\eqref{eq:Jnkdefn}, and note that for each triple $(u,v,w)\in \tilde J_n^{(2)}\times \tilde J_n^{(2)} \times \tilde J_n^{(2)}$ with $u\stackrel{m}{\sim} v$ and $v\stackrel{m}{\sim} w$,
we must have $u|_{n-1}\sibl v|_{n-1}$, $v|_{n-1}\sibl w|_{n-1}$, and $\mu_{\{u,v\}}=1=\mu_{\{v,w\}}$. Since each of $u,v$ and $w$ has exactly one grandparent, we must have $|\alpha_1(u|_{n-1})|=1$ and
$\alpha_1(u|_{n-1})=\alpha_1(v|_{n-1})=\alpha_1(w|_{n-1})$.
Therefore, since conditional on $\tilde{\cF}_n$, $(\mu_{\{u_1,u_2\}})_{u_1 \neq u_2 \in \cU_n}$ are i.i.d.~Bernoulli random variables with mean $q$,
\begin{align} \label{eq:tildeLncalc}
    \E{\widetilde L_n}
    &\le q^2 \E{|\{(u,v,w)\in \tilde J_n^{(2)}\times \tilde J_n^{(2)} \times \tilde J_n^{(2)}:
    u|_{n-1}\sibl v|_{n-1}, v|_{n-1}\sibl w|_{n-1}\}|
    } \notag \\
    &= q^2 \E{\sum_{\{(u',v',w')\in J^{(1)}_{n-1}: \alpha_1(u')=\alpha_1(v')=\alpha_1(w'), u'\neq v', v'\neq w'\}} |\mathcal C_{u'}| |\mathcal C_{v'}| |\mathcal C_{w'}| }.
\end{align}
Note that for $u',v',w'\in I_{n-1}$ distinct,
$$
\E{\left. |\cC_{u'}||\cC_{v'}||\cC_{w'}| \right| \cF_{n-1}}\le \E{\left. \xi_{u'}\xi_{v'}\xi_{w'} \right| \cF_{n-1}}=(1+p)^3,
$$
and for $u'\neq v'\in I_{n-1}$,
$$
\E{\left. |\cC_{u'}|^2 |\cC_{v'}| \right| \cF_{n-1}}\le \E{\left. (\xi_{u'})^2\xi_{v'} \right| \cF_{n-1}}=(1+p)^2(2+p).
$$
Hence by conditioning on $\mathcal F_{n-1}$ in~\eqref{eq:tildeLncalc} and splitting the sum according to whether or not we have $u'=w'$,
\begin{align*}
    \E{\widetilde L_n}
    &\le q^2 (1+p)^2 \Big((1+p)\E{|\{(u',v',w')\in J^{(1)}_{n-1}\text{ distinct}: \alpha_1(u')=\alpha_1(v')=\alpha_1(w')\}|}\\
    &\hspace{2.5cm} + (2+p)\E{|\{(u',v',u')\in J^{(1)}_{n-1}:u'\neq v', \alpha_1(u')=\alpha_1(v')\}|}\Big)\\
    &\le q^2 (1+p)^2 \Big((1+p)\E{N^{(1,1,1)}_{n-2}}+(2+p)\E{N^{(1,1)}_{n-2}}\Big),
\end{align*}
and the result follows by \refL{lem:EN} and since $0\le p\le 1$.
\end{proof}
We will also need to control the number of deletions that occur when $\tilde I_{n+1}$ is constructed, i.e.~the size of the set $\{i\le \xi_u : i\notin \cC_u\}$ for each $u\in I_n$.
Recall from~\eqref{eq:kudefn} that for $u\in I_n$,
$$
k_u = |\cK_u|=|\{v\in V_n :d_{G_n}(u,v)=3\}|.
$$
We now split $k_u$ into three quantities that we will control separately: for $u\in I_n$, let
\begin{align}
    k_u^{(g)}&:=|\alpha_3(u) |, \label{eq:kug}\\
    k_u^{(a)}&:=|\{v\in I_{n-1}\setminus \alpha_1(u):\, \alpha_1(v)\cap \alpha_2(u)\neq \emptyset \}|,\label{eq:kua}\\
    k_u^{(r)}&:=|\{v\in I_{n-1}:\, \alpha_1(v)\cap \alpha_2(u)= \emptyset,\, d_{G_n}(u,v)=3 \}|. \label{eq:kur}
\end{align}
Note that $k_u=k_u^{(g)}+k_u^{(a)}+k_u^{(r)}$.
We refer to the individuals that contribute to $k_u^{(g)},$ $k_u^{(a)}$, $k_u^{(r)}$ as `great-grandparents', `aunts' and `distant relatives' of $u$ resp.~(see \refF{fig:kv3}).
We end this section with two upper bounds on the numbers of great-grandparents and aunts; the next section will be devoted to upper bounds on the number of vertices that have distant relatives.

\begin{figure}[ht!]
\begin{subfigure}[t]{.3\textwidth}
    \hfill
	\begin{center}
	\begin{tikzpicture}
	\tikzstyle{vertex}=[circle,fill=black, draw=none, minimum size=4pt,inner sep=2pt]
	\tikzstyle{vertexG}=[circle,fill=none, draw=none, minimum size=4pt,inner sep=2pt]
    		\node[vertex] (a1) at (0,3){};
    		\node[vertex] (a2) at (0,2){};
    		\node[vertex] (a3) at (0,1){};
    		\node[vertex] (a4) at (0,0){};
    		\draw (a1)--(a2)--(a3)--(a4);
        \node[] at (-.5,3){$w$};
        \node[] at (-.5,0){$u$};
        \node[vertexG] (G) at (0,-.2){};
\end{tikzpicture}
	\end{center}
\subcaption{Great-grandparent}
    \end{subfigure}	
\begin{subfigure}[t]{.3\textwidth}
    \hfill
	\begin{center}
	\begin{tikzpicture}
	\tikzstyle{vertex}=[circle,fill=black, draw=none, minimum size=4pt,inner sep=2pt]
	\tikzstyle{vertexG}=[circle,fill=none, draw=none, minimum size=4pt,inner sep=2pt]
    		\node[vertex] (l1) at (-2,1.5){};
    		\node[vertex] (a2) at (-1,3){};
    		\node[vertex] (a3) at (0,1.5){};
    		\node[vertex] (a4) at (0,0){};
        		\draw (l1)--(a2)--(a3)--(a4);
        \node[] at (-2.5,1.5){$w$};
        \node[] at (-.5,0){$u$};
        \node[vertexG] (G) at (0,-.3){};
\end{tikzpicture}
	\end{center}
\subcaption{Aunt}
    \end{subfigure}	
\begin{subfigure}[t]{.3\textwidth}
    \hfill
	\begin{center}
	\begin{tikzpicture}
	\tikzstyle{vertex}=[circle,fill=black, draw=none, minimum size=4pt,inner sep=2pt]
       		\node[vertex] (a2) at (-1,1.5){};
    		\node[vertex] (a3) at (0,0){};
    		\node[vertex] (b3) at (-2,0){};
    		\node[vertex] (c2) at (-3,1.5){};
    		\draw (a3)--(a2)--(b3)--(c2);
		\tikzstyle{vertex}=[circle,fill=black, draw=none, minimum size=4pt,inner sep=2pt]
		\node[vertex, color=gray] (r1) at (-1.7,3){};
    		\node[vertex, color=gray] (l1) at (-2.3,3){};
		\draw[color=gray] (r1)--(a2) (l1)--(c2); 
		\node[gray] at (-2,3.1) {\tiny $\neq$};
        \node[] at (-3.5,1.5){$w$};
        \node[] at (-.5,0){$u$};
	\end{tikzpicture}
	\end{center}
\subcaption{Distant relative}
    \end{subfigure}	
\caption{Types of vertices $w\in V_n$ at distance three from $u\in I_n$.}\label{fig:kv3}
\end{figure}
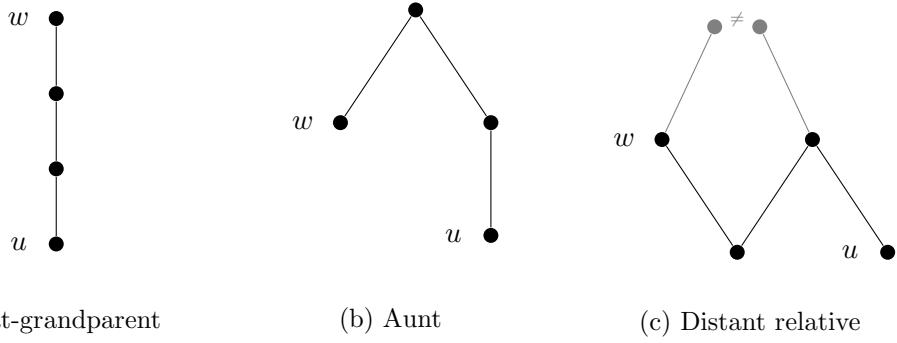

\begin{lem}\label{lem:kga}
For $n\ge 2$,
\begin{align}
\E{\sum_{u\in I_{n}} (k_u^{(g)}+k_u^{(a)})}&\le  (1+p)^3(\E{Z_{n-2}}+\I{n\ge 3}\E{Z_{n-3}}) \label{eq:kga}\\
\text{and }\quad \E{\sum_{\{u, v\in J^{(1)}_{n}:u\sibl v\}} (k_u^{(g)}+k_u^{(a)})} &\le  (1+p)^4(\E{Z_{n-2}}+\I{n\ge 3}\E{Z_{n-3}}). 
    \label{eq:kgasib}
\end{align}
\end{lem}
\begin{proof}
For $n\ge 2$, by the definition of $N_n^{(j_1,\ldots, j_k)}$ in~\eqref{eq:Nnumdefns}, we have
$$
\sum_{u\in I_{n}} k_u^{(g)}= \I{n\ge 3}N_{n-3}^{(3,0)}
\quad \text{ and }\quad
\sum_{u\in I_n} k_u^{(a)}\le N_{n-2}^{(2,1)}.
$$
(The inequality above is not an equality since a pair $u\in I_n$, $v\in I_{n-1}\setminus \alpha_1(u)$ with $\alpha_1(v)\cap \alpha_2(u)\neq \emptyset$ may have $|\alpha_1(v)\cap \alpha_2(u)|>1$ and so count more than once towards $N_{n-2}^{(2,1)}$.) By \refL{lem:EN}, the first statement~\eqref{eq:kga} follows.

For the second statement, note first that for $n\ge 3$, for $u,v\in J^{(1)}_{n}$ with $u\sibl v$, we have $\alpha_1(u)= \alpha_1(v)$ and so $\alpha_3(v)=\alpha_3(u)$. Therefore, by counting triples consisting of a pair of siblings with a common great-grandparent,
\begin{align*}
 \sum_{\{u , v\in J^{(1)}_{n}:u\sibl v\}} k_u^{(g)} 
 &\le \I{n\ge 3}\sum_{w\in I_{n-3}}|\{(u,v)\in J^{(1)}_n \times J^{(1)}_n : u\sibl v, w\in \alpha_3 (v)\cap \alpha_3(u)\}|
 \\
 &\le \I{n\ge 3} \sum_{\{w\in I_{n-3}, w'\in I_{n-1}:w\in \alpha_2(w')\}}(|I^{w'}_1 |)_2.
\end{align*}
Now for $w'\in I_{n-1}$, by~\eqref{eq:offspringbound},
\begin{align} \label{eq:bdoffspringch2}
    \E{\left. (|I^{w'}_1 |)_2 \right|\cF_{n-1}}
    \le \E{(\xi_{w'})_2 |\cF_{n-1}}=(1+p)^2.
\end{align}
Hence, taking expectations and conditioning on $\mathcal F_{n-1}$,
\begin{align*}
 \E{\sum_{\{u, v\in J^{(1)}_{n}:u\sibl v\}} k_u^{(g)} }
 &\le \I{n\ge 3} (1+p)^2 \E{N^{(2,0)}_{n-3}}
 \le \I{n\ge 3} (1+p)^4 \E{Z_{n-3}}
\end{align*}
by \refL{lem:EN}.
Similarly, for $n\ge 2$, for $u,v\in J^{(1)}_{n}$ with $u\sibl v$, we have $\alpha_2(u)= \alpha_2(v)$, and so 
if $z$ is an aunt of $u$ (i.e.~if $\alpha_1(z)\cap \alpha_2(u)\neq \emptyset$), then $\alpha_1(z)\cap \alpha_2(u)\cap \alpha_2(v)\neq \emptyset$. Therefore,
by counting triples consisting of a pair of siblings in generation $n$ and an aunt in generation $n-1$, where all three individuals have a common ancestor in generation $n-2$,
we have 
\begin{align*}
 \sum_{\{u , v\in J^{(1)}_{n}:u\sibl v\}} k_u^{(a)} 
 &\le \sum_{w\in I_{n-2}, (z,w')\in \mathcal N^{(1,1)}(w)}(|I^{w'}_1|)_2.
\end{align*}
Hence by taking expectations, conditioning on $\mathcal F_{n-1}$ and using~\eqref{eq:bdoffspringch2} again,
\begin{align*}
 \E{\sum_{\{u, v\in J^{(1)}_{n}:u\sibl v\}} k_u^{(a)} }
 &\le (1+p)^2 \E{N^{(1,1)}_{n-2}}
 \le (1+p)^4 \E{Z_{n-2}}
\end{align*}
by \refL{lem:EN}, and the result follows.
\end{proof}

\section{Distant relatives}\label{sec:relatives}
Recall from the discussion after~\eqref{eq:kug}-\eqref{eq:kur} that for $n\ge 1$, we say that $w\in I_{n-1}$ is a {\it{distant relative}} of $v\in I_n$ if the distance in $G_n$ between $w$ and $v$ is exactly 3 and these vertices do not share a common ancestor in $I_{n-2}$, i.e.~if $d_{G_n}(v,w)=3$ and $\alpha_1(w)\cap \alpha_2(v)=\emptyset$. In this section, we will see that for $n\ge 3$, if $v\in I_n$ has a distant relative $w$,  then two mergers can be associated to descendants of a great-grandparent 
of $v$; see \refL{lem:count_snakes} below. We use this correspondence to bound the expectations of numbers of vertices with distant relatives (see Lemma~\ref{lem:ERel} below); these bounds will be used in our lower bounds on $\E{\widetilde M_n}$ and $\E{Y_n}$ in Sections~\ref{sec:lower} and~\ref{sec:mainproofs} respectively (recall Lemma~\ref{lem:Z}).

Recall from~\eqref{eq:sigmavdefn} that for each $v\in I_n$, we let $\sigma(v)$ denote the equivalence class of offspring in $\tilde I_n$ that merge into the single individual $v$. In what follows, vertices in $\tilde I_n$ will be denoted with a tilde superscript, so for $v\in I_n$ we will have $\sigma(v)=\{\tilde{v}_1, \ldots, \tilde v_\ell\} \subseteq \tilde I_n$ for some $\ell\geq 1$. To aid in reading, vertices in the same generation will be denoted by the same letter, e.g.\ $v, v'$ for vertices in $I_n$. Recall from~\eqref{eq:tildealphadefn} that  for $m\ge 1$ and $\tilde v_i \in \tilde I_n$, $\alpha_m(\tilde v_i)$ denotes the set of vertices in $I_{n-m}$ which are at distance $m$ from $v_i$ in $\tilde G_n$. 

\begin{lem}\label{lem:snake}
Let $n\ge 3$ and let $w\in I_{n-1}$ be a distant relative of $v\in I_n$. Let $(v,w',v',w)$ be a path of length three in $G_n$. Then $v'\in I_n$ and $|\sigma(v')|\ge 3$. 
\end{lem}

\begin{proof} (See Figure~\ref{fig:snake}.)
If $w$ is a distant relative of $v$, then $\alpha_1(w)\cap \alpha_2(v)=\emptyset$, i.e.~$v$ and $w$ do not share an ancestor in $I_{n-2}$, and so it follows that all paths of length three in $G_n$ connecting $v$ and $w$ go through vertices in $I_{n-1}\cup I_n$. In particular, $v'\in I_{n}$ and $v'$ has at least two parents; namely $w$ and $w'$. It follows that there are distinct $\tilde v'_1,\tilde v'_2\in \sigma(v')$ such that $\{w\}=\tilde \alpha_1(\tilde v'_1)$ and $\{w'\}=\tilde \alpha_1(\tilde v'_2)$, i.e.\ $w$ is the parent of $\tilde v'_1$ and $w'$ is the parent of $\tilde v'_2$ in $\tilde G_n$ before the merging stage in the $n$th generation.

\begin{figure}
\begin{center}
\begin{tikzpicture}
        \tikzstyle{vertex}=[circle,fill=black, draw=none, minimum size=4pt,inner sep=2pt]
        \node[vertex] (w) at (-1.5,2){};
        \node[vertex] (w') at (1.5,2){};
        \node[vertex] (v1) at (-1,0){};
        \node[vertex] (v2) at (.5,0){};        
        \node[vertex] (v) at (2.5,0){};
        \draw (w)--(v1)
        (w')--(v2)
        (w')--(v);
        \node[] at (-2.5,.6){$\sigma(v')$};
        \draw (-.25,0) ellipse (1.75 and .75);
		\tikzstyle{vertex}=[circle,fill=black, draw=none, minimum size=4pt,inner sep=2pt]
		\node[vertex, color=gray] (r1) at (-.25,4){};
    	\node[vertex, color=gray] (l1) at (.25,4){};
		\draw[color=gray] (r1)--(w) (l1)--(w'); 
		\node[gray] at (0,4) {\tiny $\neq$};
		\node[] at (-1,2){$w$};
        \node[] at (2,2){$w'$};
        \node[] at (-.5,0){$\tilde{v}'_1$};
        \node[] at (1,0){$\tilde{v}'_2$};        
        \node[] at (3,0){$v$};
\end{tikzpicture}
\end{center}
    \caption{If $v$ has a distant relative $w$, then there is $v'$, sibling of $v$, such that $|\sigma(v')|\ge 3$. This figure illustrates the argument made in Lemma~\ref{lem:snake}.}
    \label{fig:snake}
\end{figure}
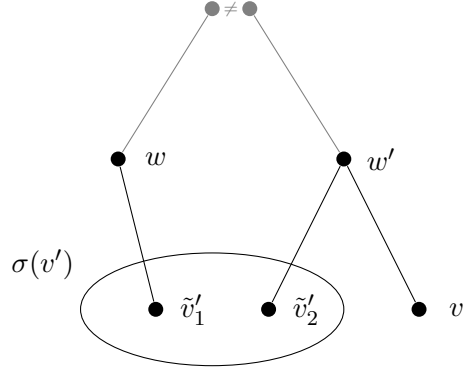

Suppose (aiming for a contradiction) that $|\sigma(v')|=2$; then as $\tilde v'_1 \neq \tilde v'_2$ we have $\sigma(v')=\{ \tilde v'_1, \tilde v'_2 \}$ and hence $\tilde v'_1 \stackrel{m}{\sim} \tilde v'_2$. Therefore $\tilde\alpha_2(\tilde v'_1)\cap \tilde\alpha_2(\tilde v'_2)\neq \emptyset$, i.e.\ $\tilde v'_1$ and $\tilde v'_2$ must have a common grandparent in $\tilde G_n$ before the merging stage in the $n$th generation. This common grandparent in $\tilde G_n$ must then be a parent of both $w$ and $w'$, and hence, since $w'$ is a parent of $v$, $v$ and $w$ have a common ancestor in generation $I_{n-2}$. But we assumed $w$ is a distant relative of $v$, which gives us a contradiction. It follows that $|\sigma(v')|\geq 3$ as claimed.
\end{proof}

We now state two simple facts about the structure of merging pairs within an equivalence class of a vertex $v$. Suppose $v\in I_n$ and $\sigma(v)=\{\tilde v_1,\ldots, \tilde v_\ell\} \subseteq \tilde I_n$. 
First, if $\ell \ge 2$, then for each $\tilde v_i$ there is $\tilde v_j \in \sigma(v)$ such that
$\tilde v_i \stackrel{m}{\sim} \tilde v_j$. Additionally if $\ell \ge 3$, in the graph with vertices given by $\tilde v_1, \ldots,\tilde v_\ell$ and edges given by merging pairs, for each $\tilde v_i$, there is a path of length two containing $\tilde v_i$. In other words, for each $\tilde v_i \in \sigma(v) \subseteq \tilde I_n$ with $|\sigma(v)|\geq 3$ there are vertices $\tilde v_{j},\tilde v_{k} \in \sigma(v)$ such that $\tilde v_i \stackrel{m}{\sim} \tilde v_j$ and either $\tilde v_i \stackrel{m}{\sim} \tilde v_k$ or $\tilde v_j \stackrel{m}{\sim} \tilde v_k$.

In genealogical terms, the following lemma says that a vertex $v$ that has a distant relative must have a great-grandparent $u$ such that among the descendants of $u$ there are at least two merging events, either both in the same generation as $v$ or one in the same generation as $v$ and one in the preceding generation. 

\begin{lem}\label{lem:count_snakes} 
Let $n\ge 3$. If $v\in I_n$ has $k_v^{(r)}>0$ then there exists $u \in \alpha_3(v)$ such that the following holds.
There exist $\tilde a, \tilde b, \tilde c, \tilde d$ with $\tilde a \stackrel{m}{\sim} \tilde b$, $\tilde c \stackrel{m}{\sim} \tilde d$,  $u \in \tilde\alpha(\tilde a) \cap \tilde \alpha(\tilde b) \cap \tilde\alpha(\tilde c) \cap \tilde \alpha(\tilde d)$ and either 
\begin{itemize}
    \item $\tilde a,\tilde b, \tilde c, \tilde d \in \tilde I_{n}$ with $\tilde a, \tilde b, \tilde d$ distinct 
\end{itemize}
\noindent or
\begin{itemize}
    \item $\tilde a, \tilde b \in \tilde I_{n-1}$ and $\tilde c, \tilde d \in \tilde I_n.$  
\end{itemize}
\end{lem}

\begin{proof}
(See Figure~\ref{fig:count_snakes}.) Suppose $v\in I_n$ has a distant relative (i.e.~$k^{(r)}_v>0$); by \refL{lem:snake} there exist $w'\in I_{n-1}$ and $v'\in I_n$ with $w' \in \alpha_1(v) \cap \alpha_1(v')$ and $|\sigma(v')|\geq 3$. 

Let $\tilde v'_1 \in \sigma(v')$ be one of the offspring of $w'$ before the merging stage, i.e.\ such that $\{w'\}=\tilde \alpha_1(\tilde v'_1)$; we know such a vertex exists as $w' \in \alpha_1(v')$. Since $|\sigma(v')|\geq 3$, there exist $\tilde v'_2, \tilde v'_3 \in \sigma(v')$ with $\tilde v'_1, \tilde v'_2$ and $\tilde v'_3$ all distinct, $\tilde v'_1 \stackrel{m}{\sim} \tilde v'_2$ and either (case 1) $\tilde v'_2 \stackrel{m}{\sim} \tilde v'_3$ or (case 2) $\tilde v'_1 \stackrel{m}{\sim} \tilde v'_3$.

In both cases $\tilde v'_1 \stackrel{m}{\sim} \tilde v'_2$, so we may take $x \in \tilde\alpha_2(\tilde v'_1) \cap \tilde \alpha_2(\tilde v'_2)$, i.e.\ let $x$ be a grandparent of $\tilde v'_1$ and $\tilde v'_2$ in $\tilde G_n$ before the merging stage. As $\tilde v'_1$ has a unique parent $w'$ before the merging stage, $x \in \tilde \alpha_2(\tilde v'_1)$ implies that $x \in \alpha_1(w')$ and hence $x \in \alpha_2(v)$. Observe also that if we had $x \in \tilde\alpha_2(\tilde v'_3)$, then any $u\in \alpha_1(x)$ would satisfy the requirements of the lemma and we would be done (i.e.\ we would have $u \in \alpha_3(v)$ and $\tilde v'_1, \tilde v'_2, \tilde v'_3 \in \tilde I_n$ distinct with $u \in \tilde\alpha(\tilde v'_1) \cap \tilde\alpha(\tilde v'_2) \cap \tilde\alpha(\tilde v'_3)$, $\tilde v'_1 \stackrel{m}{\sim} \tilde v'_2$ and either $\tilde v'_2 \stackrel{m}{\sim} \tilde v'_3$ or $\tilde v'_1 \stackrel{m}{\sim} \tilde v'_3$). Hence we may assume from now on that $x \notin \tilde \alpha_2(\tilde v'_3)$.\\

\noindent \emph{Case 1: $\tilde v'_2 \stackrel{m}{\sim} \tilde v'_3$.} 
First observe that $x \notin \tilde\alpha_2(\tilde v'_3)$ and $\tilde v'_2 \stackrel{m}{\sim} \tilde v'_3$ implies we may take $x' \in \tilde\alpha_2(\tilde v'_2)\cap \tilde\alpha_2(\tilde v'_3)$, i.e.\ let $x'$ be a common grandparent of $\tilde v'_2$ and $\tilde v'_3$ in $\tilde G_n$ before the merging stage, and we must have $x'\neq x$.

Take $w''$ such that $\{w''\}=\tilde \alpha_1(\tilde v'_2)$, i.e.\ let $w''$ be the parent of $\tilde v'_2$ in $\tilde G_n$ before the merging stage; as $\tilde v'_1 \stackrel{m}{\sim} \tilde v'_2$ we have that $\tilde v'_1$ and $\tilde v'_2$ are not siblings and so $w' \neq w''$. But now $x,x'\in \tilde \alpha_2(\tilde v'_2)$ implies $x, x' \in \alpha_1(w'')$ and so $|\sigma(w'')|\geq 2$. Let $\tilde w''_1 \in \sigma(w'')$ be an offspring of $x$ before merging, i.e. such that~$\tilde \alpha_1(w''_1)=\{x\}$, and take $\tilde w''_2 \in \sigma(w'')$ such that $\tilde w''_1 \stackrel{m}{\sim} \tilde w''_2$.

Thus there exists $u \in \tilde\alpha_2(\tilde w''_1) \cap \tilde\alpha_2(\tilde w''_2)$. Since $\{x\} = \tilde\alpha_1(\tilde w''_1)$, this implies $u \in \alpha_1(x)$ and hence $u \in \alpha_3(v)$. To complete this case it remains to show that $u$ satisfies the claims of the lemma.
Note that since $u\in \alpha_1(x)$ and $x\in \tilde \alpha_2(\tilde v'_1)\cap \tilde \alpha_2(\tilde v'_2)$, we have $u\in \tilde \alpha (\tilde v'_1)\cap \tilde \alpha(\tilde v'_2)$.
Consider $\tilde w''_1, \tilde w''_2,\tilde v'_1, \tilde v'_2$. We have $\tilde w''_1 \stackrel{m}{\sim} \tilde w''_2$, $\tilde v'_1 \stackrel{m}{\sim} \tilde v'_2$, $u \in \tilde\alpha(\tilde w''_1)\cap \tilde\alpha(\tilde w''_2)\cap \tilde\alpha(\tilde v'_1)\cap \tilde\alpha(\tilde v'_2)$, $\tilde w''_1, \tilde w''_2 \in \tilde I_{n-1}$ and $\tilde v'_1, \tilde v'_2 \in \tilde I_n$, which satisfies the claims of the lemma.\\

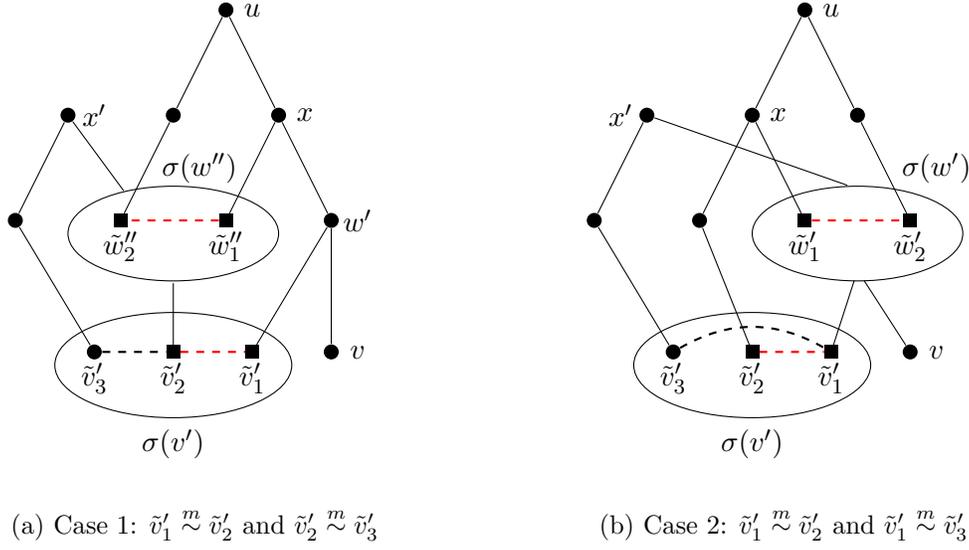
\begin{figure}
\begin{minipage}[b]{.5\textwidth}
\begin{center}
\begin{tikzpicture}[scale=.7]
        \tikzstyle{vertex}=[circle,fill=black, draw=none, minimum size=4pt,inner sep=2pt]
        \tikzstyle{vertex1}=[rectangle,fill=black, draw=none, minimum size=5 pt,inner sep=2pt]
        \node[vertex] (u) at (1,6){};
        \node[vertex] (x') at (-2,4){};
        \node[vertex] (x0) at (0,4){};
        \node[vertex] (x) at (2,4){};        
        \node[vertex] (w0) at (-3,2){};
        \node[vertex1] (w2) at (-1,2){};
        \node[] (wup) at (-0.8,2.37){};
        \node[vertex1] (w1) at (1,2){};
        \node[vertex] (w') at (3,2){};
        \node[vertex] (v3) at (-1.5,-0.5){};        
        \node[vertex1] (v2) at (0,-0.5){};        
        \node[vertex1] (v1) at (1.5,-0.5){};        
        \node[vertex] (v) at (3,-0.5){};        
        \draw (v3)--(w0)--(x')
        (x')--(wup)
        (w2)--(x0)--(u)--(x)--(w1)
        (x)--(w')--(v)
        (w')--(v1);
        \node[] (w) at (0,1){};
        \draw (w)--(v2);
        \draw[dashed, thick, color=red] (v2)--(v1)
        (w1)--(w2);
        \draw[dashed, thick] (v3)--(v2);
        \node[] at (0.5,3){$\sigma(w'')$};
        \draw (0,1.75) ellipse (2 and .9);
        \node[] at (0,-2.25){$\sigma(v')$};
        \draw (0,-0.75) ellipse (2.25 and 1);
        \node[] at (1.5,6){$u$};
        \node[] at (-1.5,4){$x'$};
        \node[] at (2.5,4){$x$};        
        \node[] at (-1,1.5){$\tilde{w}''_2$};
        \node[] at (1,1.5){$\tilde{w}''_1$};
        \node[] at (3.5,2){$w'$};
        \node[] at (-1.5,-1){$\tilde{v}'_3$};        
        \node[] at (0,-1){$\tilde{v}'_2$};        
        \node[] at (1.5,-1){$\tilde{v}'_1$};        
        \node[] at (3.5,-0.5){$v$};        
\end{tikzpicture}
\end{center}
    \subcaption{Case 1: $\tilde v'_1 \stackrel{m}{\sim} \tilde v'_2$ and $\tilde v'_2 \stackrel{m}{\sim} \tilde v'_3$}\label{fig:count_snakes1}
\end{minipage}
\begin{minipage}[b]{.5\textwidth}
\begin{center}
\begin{tikzpicture}[scale=.7]
        \tikzstyle{vertex}=[circle,fill=black, draw=none, minimum size=4pt,inner sep=2pt]
        \tikzstyle{vertex1}=[rectangle,fill=black, draw=none, minimum size=5pt,inner sep=2pt]
        \node[vertex] (u) at (1,6){};
        \node[vertex] (x') at (-2,4){};
        \node[vertex] (x0) at (0,4){};
        \node[vertex] (x) at (2,4){};        
        \node[vertex] (w0) at (-3,2){};
        \node[vertex] (w00) at (-1,2){};
        \node[vertex1] (w1) at (1,2){};
        \node[vertex1] (w2) at (3,2){};
        \node[vertex] (v3) at (-1.5,-0.5){};        
        \node[vertex1] (v2) at (0,-0.5){};        
        \node[vertex1] (v1) at (1.5,-0.5){};        
        \node[vertex] (v) at (3,-0.5){};        
        \draw (v3)--(w0)--(x')
        (v2)--(w00)--(x0)--(u)--(x)--(w2)
        (x0)--(w1);
        \node[] (wup) at (2,2.6){};
        \node[] (wdo) at (2,1.05){};
        \draw (wup)--(x')
        (v)--(wdo)--(v1);
        \draw[dashed, thick, color=red] (v2)--(v1)
        (w1)--(w2);
        \draw[dashed, thick] (v3) to[out=30,in=150] (v1);
        \node[] at (3.5,3){$\sigma(w')$};
        \draw (2,1.75) ellipse (2 and .9);
        \node[] at (0,-2.25){$\sigma(v')$};
        \draw (0,-0.75) ellipse (2.25 and 1);
        \node[] at (1.5,6){$u$};
        \node[] at (-2.5,4){$x'$};
        \node[] at (0.5,4){$x$};        
        \node[] at (1,1.5){$\tilde w'_1$};
        \node[] at (3,1.5){$\tilde w'_2$};
        \node[] at (-1.5,-1){$\tilde{v}'_3$};        
        \node[] at (0,-1){$\tilde{v}'_2$};        
        \node[] at (1.5,-1){$\tilde{v}'_1$};        
        \node[] at (3.5,-0.5){$v$};        
\end{tikzpicture}
\end{center}
    \subcaption{Case 2: $\tilde v'_1 \stackrel{m}{\sim} \tilde v'_2$ and $\tilde v'_1 \stackrel{m}{\sim} \tilde v'_3$}\label{fig:count_snakes2}
\end{minipage}
    \caption{Illustration to accompany the proof of \refL{lem:count_snakes}. We have omitted the simple case covered at the start of the proof in which $\tilde v'_1, \tilde v'_2, \tilde v'_3$ share a grandparent in $\tilde G_n$ before the merging stage in generation $n$; notice that in both cases above $x \notin \tilde \alpha_2(v'_3)$.}
    \label{fig:count_snakes}
\end{figure}

\noindent \emph{Case 2: $\tilde v'_1 \stackrel{m}{\sim} \tilde v'_3$.} 
Similarly to case 1, $x \notin \tilde \alpha_2(\tilde v'_3)$ and $\tilde v'_1 \stackrel{m}{\sim} \tilde v'_3$ together imply we may take $x' \in \tilde \alpha_2(\tilde v'_1)\cap \tilde \alpha_2(\tilde v'_3)$ with $x'\neq x$. But now, since $\tilde \alpha_1(\tilde v'_1)=\{w'\}$, we have $x, x' \in \alpha_1(w')$ and so $|\sigma(w')|\geq 2$. Let $\tilde w'_1 \in \sigma(w')$ be an offspring of $x$ before merging, i.e.\ $\tilde \alpha_1(\tilde w'_1)=\{x\}$, and let $\tilde w'_2\in \sigma(w')$ be such that $\tilde w'_1 \stackrel{m}{\sim} \tilde w'_2$.

Now we take $u \in \tilde \alpha_2(\tilde w'_1) \cap \tilde \alpha_2(\tilde w'_2)$, and it remains only to show that $u$ satisfies the claims of the lemma. First note that $\{x\} = \tilde\alpha_1(\tilde w'_1)$ implies $u \in \alpha_1(x)$ and hence $u \in \alpha_3(v)$. 
Moreover, since $u\in \alpha_1(x)$ we have $u\in \tilde \alpha (\tilde v'_1)\cap \tilde \alpha(\tilde v'_2)$.
Consider $\tilde w'_1, \tilde w'_2,\tilde v'_1, \tilde v'_2$. We have $\tilde w'_1 \stackrel{m}{\sim} \tilde w'_2$, $\tilde v'_1 \stackrel{m}{\sim} \tilde v'_2$, $u \in \tilde\alpha(\tilde w'_1)\cap\tilde\alpha(\tilde w'_2)\cap\tilde\alpha(\tilde v'_1)\cap\tilde\alpha(\tilde v'_2)$, $\tilde w'_1, \tilde w'_2 \in \tilde I_{n-1}$ and $\tilde v'_1, \tilde v'_2 \in \tilde I_n$, satisfying the claims of the lemma.
\end{proof}

To control distant relatives' contribution to the amount of deletion that occurs in the construction of the graph process, we bound the expectation of the number of individuals that have a distant relative; this will allow us to show that deletion due to distant relatives makes a negligible contribution when $q$ is small. For each $n\ge 1$, let $\cD_n\subseteq I_n$ be the set of vertices that have a distant relative, i.e.
\begin{equation} \label{eq:cDdefn}
    \cD_n :=\{v\in I_n : k^{(r)}_v>0\}.
\end{equation}
In the proof of \refL{lem:EM'} below we will need to bound pairs of siblings in which one sibling has a distant relative. Let
\begin{align} \label{eq:Relnsibldefn}
D_n^{s}:= |\{(v,v'):\; v'\in I_n, \, v\in J^{(1)}_n\cap \cD_n, \, v \sibl v'\}|;    
\end{align}
the constraint $v\in J^{(1)}_n$ will allow us to apply \refL{lem:count_snakes} to obtain \eqref{eq:Eksibl} below. Finally, let $D_n:=|\cD_n|$.
\begin{lem}\label{lem:ERel}
There exists $C_3>0$ such that for $p,q\in [0,1]$ and $n\ge 3$,
\begin{align}
\E{D_n}
&\le C_3 q^2\E{Z_{n-3}} \label{eq:ERel}\\
    \text{and }\quad  \E{D_n^s}
    &\le     C_3 q^2\E{Z_{n-3}}. \label{eq:Eksibl}
\end{align}
\end{lem}

\begin{proof}
Let $A_1\subseteq I_{n-3}$ denote the set of individuals such that at least two pairs of descendants in $\tilde I_n$ merge, i.e.
$$
A_1 :=\{u\in I_{n-3}: \exists \tilde v_1, \tilde v_2, \tilde v_3, \tilde v_4 \in \tilde I_n, \tilde v_1, \tilde v_2, \tilde v_4 \text{ distinct s.t.~}u\in \cap_{i=1}^4 \tilde \alpha(\tilde v_i), \tilde v_1 \stackrel{m}{\sim} \tilde v_2, \tilde v_3 \stackrel{m}{\sim} \tilde v_4\}.
$$
Let $A_2\subseteq I_{n-3}$ denote the set of individuals such that at least one pair of descendants in $\tilde I_n$ and at least one pair of descendants in $\tilde I_{n-1}$ merge, i.e.
\begin{align*}
A_2 &:=\{u\in I_{n-3}: \exists \tilde v_1, \tilde v_2 \in \tilde I_n, \tilde w_1, \tilde w_2 \in \tilde I_{n-1} \text{ s.t.~}u\in \tilde \alpha(\tilde v_1)\cap \tilde \alpha(\tilde v_2)\cap\tilde \alpha(\tilde w_1)\cap\tilde \alpha(\tilde w_2), \\
&\hspace{11.5cm} \tilde v_1 \stackrel{m}{\sim} \tilde v_2, \tilde w_1 \stackrel{m}{\sim} \tilde w_2\}.
\end{align*}
By \refL{lem:count_snakes}, if $v\in I_n$ has a distant relative (i.e.~if $k^{(r)}_v>0$) then it has an ancestor $u\in \alpha_3(v)\cap (A_1\cup A_2)$. 
Therefore, recalling the definition of $I^u_m$ in~\eqref{eq:Inudefn},
\begin{align} \label{eq:reln}
   D_n= \sum_{v\in I_n} \I{v\in \cD_n} \le \sum_{u\in I_{n-3}} |I^u_3| \I{u\in A_1\cup A_2}.
   \end{align}
 Moreover, for a pair of siblings $v,v'\in I_n$ with $v\sibl v'$ and $v\in \cD_n\cap J_n^{(1)}$, $v$ must have an ancestor $u\in \alpha_3(v)\cap (A_1\cup A_2)$, and since $v$ has a unique parent, we must also have that $u\in \alpha_3(v')$. Therefore
   \begin{align} \label{eq:relnsib}
   D_n^s=\sum_{v,v'\in I_n} \I{v\in \cD_n\cap J_n^{(1)}}\I{v\sibl v'} \le \sum_{u\in I_{n-3}} |I^u_3|(|I^u_3|-1) \I{u\in A_1\cup A_2}.
\end{align}
We will now bound the expectation of $|I^u_3|^2 \I{u\in A_1\cup A_2}$ for $u\in I_{n-3}$, conditional on $\cF_{n-3}$.
Note first that by~\eqref{eq:ItildeIbd}, we have $|I^u_3|\le |\tilde I^u_3|$.
Then by the definition of $A_1$ and a union bound,
\begin{align} 
    \E{|\tilde I^u_3|^2 \I{u\in A_1} \Big| \tilde{\cF}_{n}}&\le 
    |\tilde I^u_3|^2 \sum_{\tilde v_1, \ldots, \tilde v_4 \in \tilde I^u_3, \tilde v_1, \tilde v_2, \tilde v_4 \text{ distinct}}\E{\left. \I{\tilde v_1 \stackrel{m}{\sim}\tilde v_2,\tilde v_3 \stackrel{m}{\sim}\tilde v_4 } \right| \tilde{\cF}_n } \notag \\
    &\le |\tilde I^u_3|^2 \sum_{\tilde v_1, \ldots, \tilde v_4 \in \tilde I^u_3, \tilde v_1, \tilde v_2, \tilde v_4 \text{ distinct}}\E{\left. \mu_{\{\tilde v_1, \tilde v_2\}}\mu_{\{\tilde v_3, \tilde v_4\}} \right| \tilde{\cF}_n } \notag \\
    &\le q^2 |\tilde I^u_3|^6.
\end{align}
Hence by conditioning on $\tilde{\cF}_n$ and using that $|I^u_3|\le |\tilde I^u_3|$,
\begin{align*} 
    \E{|I^u_3|^2 \I{u\in A_1} \Big| \cF_{n-1}}&\le 
    \E{q^2 |\tilde I^u_3|^6 \Big| \cF_{n-1}}
    \le q^2 \sum_{w_1, \ldots, w_6 \in I^u_2} \E{\left. \prod_{i=1}^6 |\cC_{w_i}|\right| \cF_{n-1}}.
\end{align*}
Let $\xi$ denote a Poisson random variable with mean $1+p$.
For $w_1, \ldots, w_6 \in I^u_2$, we have
\begin{equation} \label{eq:jensenxi}
\E{\left. \prod_{i=1}^6 |\cC_{w_i}|\right| \cF_{n-1}}
\le \E{\left. \prod_{i=1}^6 \xi_{w_i}\right| \cF_{n-1}}
=\E{\prod_{i=1}^6 \xi_{w_i}}\le \E{\xi^6}
\end{equation}
by Jensen's inequality (by the same argument as in~\eqref{eq:useJensen}).
Hence
\begin{equation} \label{eq:K1}
    \E{|I^u_3|^2 \I{u\in A_1} \Big| \cF_{n-1}}\le q^2 \E{\xi^6} |I^u_2|^6.
\end{equation}
Let $A'_2\subseteq I_{n-3}$ denote the set of individuals such that at least one pair of descendants in $\tilde I_{n-1}$ merge, i.e.
\begin{align*}
A'_2 &:=\{u\in I_{n-3}: \exists \tilde w_1, \tilde w_2 \in \tilde I_{n-1} \text{ s.t.~}u\in \tilde \alpha(\tilde w_1)\cap\tilde \alpha(\tilde w_2), \tilde w_1 \stackrel{m}{\sim} \tilde w_2\}.
\end{align*}
Then by the definition of $A_2$ and a union bound,
$$
\E{\left. |\tilde I^u_3|^2 \I{u\in A_2} \right| \tilde{\cF}_n}
\le \I{u\in A'_2} |\tilde I^u_3|^2 \sum_{\tilde v_1 \neq \tilde v_2 \in \tilde I^u_3} \E{\left. \I{\tilde v_1 \stackrel{m}{\sim} \tilde v_2} \right| \tilde{\cF}_n }
\le q \I{u\in A'_2} |\tilde I^u_3|^4.
$$
Hence by conditioning on $\tilde {\cF}_n$, and since $|I^u_3|\le |\tilde I^u_3|$,
\begin{align*}
    \E{|I^u_3|^2 \I{u\in A_2} \Big| \cF_{n-1}}&\le q \I{u\in A'_2} \sum_{w_1,\ldots, w_4 \in I^u_2} \E{\left. \prod_{i=1}^4 |\cC_{w_i}| \right| \cF_{n-1}}\\
    &\le q \I{u\in A'_2}\E{\xi^4}|I^u_2|^4,
\end{align*}
by Jensen's inequality and the same argument as in~\eqref{eq:jensenxi}. 
Then by conditioning on $\cF_{n-1}$, and since $|I^u_2|\le |\tilde I^u_2|$ by~\eqref{eq:ItildeIbd}, we have
\begin{align*}
    \E{|I^u_3|^2 \I{u\in A_2} \Big| \tilde{\cF}_{n-1}}
    &\le q \E{\xi^4} |\tilde I^u_2|^4 \E{\I{u\in A'_2} \Big| \tilde{\cF}_{n-1}}\\
    &\le q \E{\xi^4} |\tilde I^u_2|^4 \sum_{\tilde w_1 \neq \tilde w_2 \in \tilde I^u_2} \E{\left. \I{\tilde w_1 \stackrel{m}{\sim}\tilde w_2} \right| \tilde{\cF}_{n-1}}\\
    &\le q^2 \E{\xi^4} |\tilde I^u_2|^6.
\end{align*}
Therefore by conditioning on $\tilde{\cF}_{n-1}$,
\begin{align*}
    \E{|I^u_3|^2 \I{u\in A_2} \Big| \cF_{n-2}}&\le q^2 \E{\xi^4}
    \sum_{x_1, \ldots, x_6 \in I^u_1}\E{\left. \prod_{i=1}^6 |\cC_{x_i}| \right| \cF_{n-2}}\\
    &\le q^2 \E{\xi^4}\E{\xi^6}|I^u_1|^6
\end{align*}
by Jensen's inequality (as in~\eqref{eq:jensenxi}).
Hence by~\eqref{eq:K1},
\begin{align*}
    &\E{|I^u_3|^2 \I{u\in A_1 \cup A_2} \Big| \cF_{n-3}}\\
    &\le q^2 \Big( \E{\xi^6}\E{|I^u_2|^6 \Big| \cF_{n-3}}+\E{\xi^4}\E{\xi^6}\E{|I^u_1|^6\Big| \cF_{n-3}}\Big)\\
    &= q^2 \Big( \E{\xi^6}\E{\overline{N}^{(2,2,2,2,2,2)}(u) \Big| \cF_{n-3}}+\E{\xi^4}\E{\xi^6}\E{\overline{N}^{(1,1,1,1,1,1)}(u)\Big| \cF_{n-3}}\Big),
\end{align*}
by the definition of $\overline N^{(j_1,\ldots, j_k)}(w)$ in~\eqref{eq:barNnumdefn}.
Therefore, summing over $I_{n-3}$,
$$
\E{\sum_{u\in I_{n-3}} |I^u_3|^2 \I{u\in A_1 \cup A_2} }
\le q^2 \Big( \E{\xi^6}\E{\overline{N}_{n-3}^{(2,2,2,2,2,2)}}+\E{\xi^4}\E{\xi^6}\E{\overline{N}_{n-3}^{(1,1,1,1,1,1)}}\Big),
$$
and the result follows by~\eqref{eq:reln},~\eqref{eq:relnsib} and \refL{lem:ENbar}.
\end{proof}

\section{Lower bounds on sums restricted to $\tilde J_n^{(2)}$ and $J_n^{(2)}$}\label{sec:lower}

In this section, we prove lower bounds on the expectations of $\widetilde M_n$ and $\sum_{u\in J^{(2)}_n}k_u^{(a)}$, which will be used in Section~\ref{sec:mainproofs} to give lower bounds on the expectations of the number of mergers and deletions in each generation.

The following decomposition and lower bound will be used several times. 
Recall~\eqref{eq:kug}-\eqref{eq:kur}, and recall that $k_u=k^{(g)}_u+k^{(a)}_u+k^{(r)}_u$.
For any $n\ge 0$, $u\in I_n$ and $c\in \N$, using that $(1-q)^a\ge 1-aq$ for $a\in \N\cup \{0\}$, 
\begin{align}
    (1-q)^{c k_u}&=(1-q)^{c(k_u^{(g)}+k_u^{(a)})}-(1-q)^{c(k_u^{(g)}+k_u^{(a)})}\left(1-(1-q)^{c k_u^{(r)}}\right) \nonumber\\
    &\ge 1-qc(k_u^{(g)}+k_u^{(a)})-\I{k_u^{(r)}\neq 0} \nonumber\\
    &= 1-qc(k_u^{(g)}+k_u^{(a)})-\I{u\in \cD_{n}} \label{eq:k2}    
\end{align}
by the definition of $\cD_n$ in~\eqref{eq:cDdefn}.
For $n\ge 1$, let
\begin{equation} \label{eq:Sndefn}
  S_{n}:=|\{(v,v')\in J_{n}^{(1)}\times J_{n}^{(1)}:\, v\sibl v'\}|,
\end{equation}
the number of ordered pairs of siblings in the $n$th generation, each of which has only one parent.
In order to prove a lower bound on $\E{\widetilde{M}_n}$ in Lemma~\ref{lem:EM'} below, we need a lower bound on $\E{S_{n-1}}$.
\begin{lem}\label{lem:ES}
There exists $C_4>0$ such that for $p,q\in [0,1]$ and $n\ge 4$,  
\begin{align}
    \E{S_n}&\ge (1+p)^2\E{Z_{n-1}} -C_4 q\max_{2\le k\le 4} \E{Z_{n-k}}. \label{eq:ES}
\end{align}
\end{lem}

\begin{proof}
Note first that by our construction in Section~\ref{subsec:UlamHarris}, for $u\in I_{n-1}$ and $i\neq i'\in \cC_u$, the pair $(ui,u i')$ contributes to $S_n$ unless there exists $\tilde v\in \tilde I_n$ such that $ui \stackrel{m}{\sim} \tilde v$ or $ui' \stackrel{m}{\sim} \tilde v$.
Hence
\begin{equation} \label{eq:Sn1stlower}
    S_n \ge \sum_{u\in I_{n-1}} |\cC_u| (|\cC_u|-1)
    -\sum_{u\in I_{n-1}}2|\{(i,i',\tilde v)\in \cC_u\times \cC_u \times \tilde I_n : i\neq i', ui \stackrel{m}{\sim}\tilde v\}|.
\end{equation}
We now need an upper bound on the second term on the right hand side.
For $u\in I_{n-1}$ and $i,i'\in \cC_u$, $\tilde v \in \tilde I_n$, if $ui\stackrel{m}{\sim}\tilde v$ then we must have $\alpha_1(u)\cap \tilde \alpha_2(\tilde v) \neq \emptyset$ and $u\notin \tilde \alpha_1(\tilde v)$. Therefore
\begin{align*}
    \sum_{u\in I_{n-1}}|\{(i,i',\tilde v)\in \cC_u\times \cC_u \times \tilde I_n : i\neq i', ui \stackrel{m}{\sim}\tilde v\}|
    &\le \sum_{w\in I_{n-2}}\sum_{u\neq u' \in I^w_1} \sum_{i\neq i'\in \cC_u, \tilde v\in \cC_{u'}}\mu_{\{ui,\tilde v\}}.
\end{align*}
Therefore, by conditioning first on $\tilde{\cF}_n$ and then on $\cF_{n-1}$, and since, for $u\neq u'\in I_{n-1}$, we have
$\E{|\cC_u|(|\cC_u|-1)|\cC_{u'}| | \cF_{n-1}}\le \E{\xi_u(\xi_u-1)\xi_{u'} | \cF_{n-1}}=(1+p)^3$,
\begin{align*}
    \E{\sum_{u\in I_{n-1}}|\{(i,i',\tilde v)\in \cC_u\times \cC_u \times \tilde I_n : i\neq i', ui \stackrel{m}{\sim}\tilde v\}|}
    &\le q(1+p)^3 \E{\sum_{w\in I_{n-2}}|I^w_1|(|I^w_1|-1)}\\
    &= q(1+p)^3 \E{N^{(1,1)}_{n-2}}\\
    &\le q(1+p)^5 \E{Z_{n-2}},
\end{align*}
where the second line follows by the definition of $N^{(j_1,\ldots, j_k)}_n$ in~\eqref{eq:Nnumdefns}, and the last line follows by Lemma~\ref{lem:EN}.

Recall from our construction in Section~\ref{subsec:UlamHarris} that conditional on $\cF_{n-1}$, for $u\in I_{n-1}$, $|\mathcal C_u|$ is Poisson distributed with mean $(1+p)(1-q)^{k_u}$. It follows from~\eqref{eq:Sn1stlower} that
\begin{align*}
    \E{S_n}
    &\ge \E{\sum_{u\in I_{n-1}} \E{|\mathcal C_u|(|\mathcal C_u|-1)|\cF_{n-1}}}-2q(1+p)^5\E{Z_{n-2}}\\
    &\ge (1+p)^2\E{\sum_{u\in I_{n-1}} (1-q)^{2k_u}}-64q\E{Z_{n-2}}
\end{align*}
since $p\le 1$.
Therefore it suffices to show that there exists a constant $C'>0$ such that 
\begin{align} \label{eq:Snclaim}
    \E{\sum_{u\in I_{n-1}} (1-q)^{2k_u}}\ge  \E{Z_{n-1}} -C'q\max_{2\le k\le 4} \E{Z_{n-k}}.
\end{align}
Since $|I_{n-1}|=Z_{n-1}$ and $|\cD_{n-1}|=D_{n-1}$, we obtain from \eqref{eq:k2} with $c=2$ that
\begin{align*}
   \E{\sum_{u\in I_{n-1}} (1-q)^{2k_u}}
   &\ge \E{Z_{n-1}}-\E{D_{n-1}}-2q \E{\sum_{u\in I_{n-1}}(k_u^{(a)}+k_u^{(g)})}\\
   &\ge \E{Z_{n-1}}-C_3 q^2\E{Z_{n-4}}-2q(1+p)^3(\E{Z_{n-3}}+\E{Z_{n-4}}),
\end{align*}
where in the last inequality we use~\eqref{eq:ERel} from \refL{lem:ERel} and~\eqref{eq:kga} from Lemma~\ref{lem:kga} (using that $n\ge 4$); this establishes~\eqref{eq:Snclaim} and completes the proof.
\end{proof}
We can now prove a lower bound on the expectation of $\widetilde M_n$; using the same argument, we establish a lower bound on the expectation of another quantity, which will be used in the proof of Lemma~\ref{lem:ES21} below.
\begin{lem}\label{lem:EM'}
There exists $C_5>0$ such that for $p,q\in [0,1]$ and $n\ge 5$,
\begin{align}
\E{\widetilde M_n} &\ge \frac{q}{2}\E{Z_{n-2}} -C_5 q^2\max_{3\le k\le 5} \E{Z_{n-k}} \label{eq:lemEM'1}\\
\text{and }\quad \E{\sum_{\{(v,v')\in J^{(1)}_{n-1}\times J^{(1)}_{n-1}: v\sibl v'\}}|\mathcal C_v|}&\ge \E{Z_{n-2}} -C_5 q\max_{3\le k\le 5} \E{Z_{n-k}}. \label{eq:lemEM'2}
\end{align}
\end{lem}

\begin{proof}
Recall from~\eqref{eq:tildeMdefn} that $\widetilde M_n$ counts the number of (unordered) pairs  $\{vi,v'i'\}\subseteq \tilde J_n^{(2)}$ such that $vi\stackrel{m}{\sim} v'i'$, and recall from~\eqref{eq:Jnkdefn} that
$\tilde J_n^{(2)}\subseteq \tilde I_n$ is the set of vertices $vi$ such that $v\in I_{n-1}$ has exactly one parent and $i\in \cC_v$.
Note that a pair $\{vi,v'i'\}$ with $v,v'\in I_{n-1}$, $i\in \cC_v$ and $i'\in \cC_{v'}$ contributes to $\widetilde M_n$ if and only if $v,v'\in J^{(1)}_{n-1}$, $v\sibl v'$ and $\mu_{\{vi,v'i'\}}=1$.
Therefore 
$$
\widetilde M_n = \frac 12 \sum_{\{(v,v')\in J^{(1)}_{n-1}\times J^{(1)}_{n-1}:v\sibl v'\}} \sum_{i\in \cC_v, i'\in \cC_{v'}}\mu_{\{vi,v'i'\}}.
$$
By conditioning on $\tilde{\cF}_n$, it follows that
$$
\E{\widetilde M_n}=\frac{q}{2}\E{\sum_{\{(v,v')\in J^{(1)}_{n-1}\times J^{(1)}_{n-1}:v\sibl v'\}} |\mathcal C_v| \cdot |\mathcal C_{v'}|}.
$$
Recall from our construction in Section~\ref{subsec:UlamHarris} that conditional on $\cF_{n-1}$, $(\cC_v)_{v\in I_{n-1}}$ are independent, and $|\cC_v|$ has Poisson distribution with mean $(1+p)(1-q)^{k_v}$, where $k_v =|\{u\in V_{n-1}:d_{G_{n-1}}(u,v)=3\}|$.
For $v,v'\in J_{n-1}^{(1)}$ with $v\sibl v'$, we infer that $v$ and $v'$ have a shared unique parent $w\in I_{n-2}$, and all paths to $v$ (resp. $v'$) in $G_{n-1}$ go through $w$. This means that for $u\notin \{v,v'\}$, $d_{G_{n-1}}(u,v)=d_{G_{n-1}}(u,w)+1=d_{G_{n-1}}(u,v')$. In particular, $k_v=k_{v'}$ and so $\E{|\mathcal C_v| \cdot |\mathcal C_{v'}||\mathcal F_{n-1}}=(1+p)^2 (1-q)^{2k_v}$. By conditioning on $\cF_{n-1}$, it follows that 
\begin{align} \label{eq:tildeMEsum}
    \E{\widetilde M_n}=\frac{q}{2}(1+p)^2\E{\sum_{\{(v,v')\in J^{(1)}_{n-1}\times J^{(1)}_{n-1}:v\sibl v'\}}(1-q)^{2k_v}}.
\end{align}
Since $\E{|\mathcal C_v| |\mathcal F_{n-1}}=(1+p) (1-q)^{k_v}$ for each $v\in I_{n-1}$, we also have 
\begin{align} \label{eq:sumCvE}
    \E{\sum_{\{(v,v')\in J^{(1)}_{n-1}\times J^{(1)}_{n-1}:v\sibl v'\}}|\mathcal C_v|}=(1+p)\E{\sum_{\{(v,v')\in J^{(1)}_{n-1}\times J^{(1)}_{n-1}:v\sibl v'\}}(1-q)^{k_v}}.
\end{align}
By~\eqref{eq:k2}, and then using the definition of $S_{n-1}$ in~\eqref{eq:Sndefn} and the definition of $D^s_{n-1}$ in~\eqref{eq:Relnsibldefn}, we have that for $m\in \{1,2\}$,
\begin{align*}
    \sum_{\{(v,v')\in J^{(1)}_{n-1}\times J^{(1)}_{n-1}:v\sibl v'\}}(1-q)^{m k_v}&\ge \sum_{\{(v,v')\in J^{(1)}_{n-1}\times J^{(1)}_{n-1}:v\sibl v'\}}(1-mq(k_v^{(g)}+k_v^{(a)}) -\I{v\in \cD_{n-1}})\\
    &\ge S_{n-1}-\sum_{\{(v,v')\in J^{(1)}_{n-1}\times J^{(1)}_{n-1}:v\sibl v'\}} mq(k_v^{(g)}+k_v^{(a)})-D^s_{n-1}.
\end{align*}
Taking expectations, and using \refL{lem:ES}, \eqref{eq:kgasib} in Lemma~\ref{lem:kga} and \eqref{eq:Eksibl} in \refL{lem:ERel}, it follows that for $m\in \{1,2\}$,
\begin{align*}
    &\E{\sum_{\{(v,v')\in J^{(1)}_{n-1}\times J^{(1)}_{n-1}:v\sibl v'\}}(1-q)^{m k_v}}\\&\ge \
    (1+p)^2 \E{Z_{n-2}}-C_4 q \max_{3\le k \le 5}\E{Z_{n-k}}
    -mq(1+p)^4(\E{Z_{n-3}}+\E{Z_{n-4}})-C_3 q^2 \E{Z_{n-4}}\\
    &\ge \E{Z_{n-2}} -(C_4+2^6+C_3)q \max_{3\le k\le 5}\E{Z_{n-k}}.
\end{align*}
By substituting into~\eqref{eq:tildeMEsum} and~\eqref{eq:sumCvE}, this completes the proof.
\end{proof}
We can now use~\eqref{eq:lemEM'2} in Lemma~\ref{lem:EM'} to prove a lower bound that will be used in Section~\ref{sec:mainproofs}.
\begin{lem}\label{lem:ES21}
There exists $C_6>0$ such that for $p,q\in [0,1]$ and $n\ge 5$,
\begin{align*}
    \E{\sum_{u\in J_n^{(2)}} k_u^{(a)}}\ge \E{Z_{n-2}}-C_6 q\max_{2\le k\le 5}\E{Z_{n-k}}.
\end{align*}
\end{lem}

\begin{proof}
Note that by the definition of $k^{(a)}_u$ in~\eqref{eq:kua},
$$
\sum_{u\in J^{(2)}_n}k^{(a)}_u = |\{(u,w)\in J^{(2)}_n \times I_{n-1} : \alpha_2(u) \cap \alpha_1(w) \neq \emptyset, w \notin \alpha_1(u)\}|.
$$
By our construction in Section~\ref{subsec:UlamHarris}, for $v,v'\in J^{(1)}_{n-1}$ with $v\sibl v'$ and $i\in \cC_v$, we have that $(vi,v')$ contributes to $\sum_{u\in J^{(2)}_n}k^{(a)}_u$ if there does not exist $\tilde u \in \tilde I_n$ such that $\tilde u \stackrel{m}{\sim} vi$. Therefore
\begin{equation} \label{eq:kua1stlower}
    \sum_{u\in J^{(2)}_n}k^{(a)}_u
    \ge \sum_{\{(v,v')\in J^{(1)}_{n-1}\times J^{(1)}_{n-1}:v\sibl v'\}}|\cC_v|
    -\sum_{\{(v,v')\in J^{(1)}_{n-1}\times J^{(1)}_{n-1}:v\sibl v'\}}|\{(i,\tilde u)\in \cC_v \times \tilde I_n: \tilde u \stackrel{m}{\sim}vi \}|.
\end{equation}
For $v,v'\in J^{(1)}_{n-1}$ with $v\sibl v'$ and $i\in \cC_v$, $\tilde u \in \tilde I_n$, if $\tilde u\stackrel{m}{\sim}vi$ then we must have 
$\tilde \alpha_2(\tilde u)\cap \alpha_1(v)\neq \emptyset$ and $v\notin \tilde \alpha_1(\tilde u)$.
Since $\alpha_1(v)=\alpha_1(v')$, it follows that 
$\tilde \alpha_2(\tilde u)\cap \alpha_1(v)\cap \alpha_1(v')\neq \emptyset$. Therefore
\begin{align*}
   &\E{\sum_{\{(v,v')\in J^{(1)}_{n-1}\times J^{(1)}_{n-1}:v\sibl v'\}}|\{(i,\tilde u)\in \cC_v \times \tilde I_n: \tilde u \stackrel{m}{\sim}vi \}|}\\
   &\le \E{\sum_{w\in I_{n-2}}\sum_{\{(v,v',v'')\in I^w_1\times I^w_1 \times I^w_1: v\neq v''\}}\sum_{i\in \cC_v, i''\in \cC_{v''}}\mu_{\{vi,v''i''\}}}\\
   &=q\E{\sum_{w\in I_{n-2}}\sum_{\{(v,v',v'')\in I^w_1\times I^w_1 \times I^w_1: v\neq v''\}}| \cC_v| |\cC_{v''}|}\\
   &\le q(1+p)^2\E{\sum_{w\in I_{n-2}}|I^w_1|^3 },
\end{align*}
where the third line follows by conditioning on $\tilde{\cF}_n$, and the last line by conditioning on $\cF_{n-1}$ and since for $v\neq v''\in I_{n-1}$, $\E{|\cC_v | \cdot |\cC_{v''}| |\cF_{n-1}}\le (1+p)^2.$
By the definition of $\overline{N}^{(j_1,\ldots, j_k)}_n$ in~\eqref{eq:barNnumdefn}, and then by Lemma~\ref{lem:ENbar},
$$
\E{\sum_{w\in I_{n-2}}|I^w_1|^3 }=\E{\overline{N}^{(1,1,1)}_{n-2}}\le \E{\xi^3} \E{Z_{n-2}},
$$
where $\xi$ is Poisson distributed with mean $1+p$.
By~\eqref{eq:kua1stlower} and by~\eqref{eq:lemEM'2} in Lemma~\ref{lem:EM'}, we now have
$$
\E{\sum_{u\in J^{(2)}_n}k^{(a)}_u}
\ge \E{Z_{n-2}}-C_5 q \max_{3\le k \le 5} \E{Z_{n-k}}-q(1+p)^2 \E{\xi^3}\E{Z_{n-2}},
$$
which completes the proof.
\end{proof}

\section{Proof of Propositions \refand{prop:Zn}{prop:ERn}}\label{sec:mainproofs}
In this section, we use results from Sections~\ref{sec:upper}-\ref{sec:lower} to complete the proof of Propositions \refand{prop:Zn}{prop:ERn}.
Recall the definition of $Y_n$ in~\eqref{eq:Yndefn}, and recall from Lemma~\ref{lem:ERel} that $C_3$ is a positive constant.
We begin by proving upper and lower bounds on $\E{Y_n}$.
\begin{lem}\label{lem:EY}
There exists $C_7>0$ such that for $p,q\in [0,1]$, $n\ge 6$ and $m\ge 4$,
\begin{align}
    \E{Y_n}\le& (1+p-q)\E{Z_{n-1}}-q\E{Z_{n-3}} + C_7 q^2\max_{3\le k\le 6} \E{Z_{n-k}} \label{eq:EYupper}\\
    \text{and}\quad \E{Y_m}\ge& (1+p)\E{Z_{m-1}}-q(1+p)^4\left( \E{Z_{m-3}}+\E{Z_{m-4}}\right)
    -C_3(1+p)q^2\E{Z_{m-4}}.
    \label{eq:EYlower}
\end{align}
\end{lem}

\begin{proof}
We begin with a proof of the lower bound~\eqref{eq:EYlower}.
Take $m\ge 4$. Recall from our construction in Section~\ref{subsec:UlamHarris} that for $u\in I_{m-1}$, we have
$\E{|\mathcal C_u| | \mathcal F_{m-1}}=(1+p)(1-q)^{k_u}$.
Since $Y_m = \sum_{u\in I_{m-1}}|\mathcal C_u|$, by conditioning on $\cF_{m-1}$, and then using \eqref{eq:k2} with $c=1$ and that $|I_{m-1}|=Z_{m-1}$ and $|\cD_{m-1}|=D_{m-1}$, we have
\begin{align*} 
   \E{Y_m}&= (1+p)\E{\sum_{u\in I_{m-1}} (1-q)^{k_u}}\\
    &\ge (1+p)\left(\E{Z_{m-1}}-q \E{\sum_{u\in I_{m-1}} (k_u^{(g)}+k_u^{(a)})} -\E{D_{m-1}}\right)\\
    &\ge (1+p)\E{Z_{m-1}}-q(1+p)^4 (\E{Z_{m-3}}+\E{Z_{m-4}})-C_3(1+p)q^2 \E{Z_{m-4}},
\end{align*}
where the last line follows by~\eqref{eq:kga} in \refL{lem:kga} and~\eqref{eq:ERel} in \refL{lem:ERel}.
This completes the proof of~\eqref{eq:EYlower}.

It remains to prove the upper bound~\eqref{eq:EYupper}. Take $n\ge 6$.
For $u\in I_{n-1}$, since $u$ has at least one great-grandparent we have $k_u^{(g)}\ge 1$, and so $k_u\ge k_u^{(a)}+1$. It follows that 
\begin{align} \label{eq:Ynupperstart}
   \E{Y_n}&\le (1+p)\E{\sum_{u\in I_{n-1}} (1-q)^{k_u^{(a)}+1}} \notag \\
   &\le (1+p)(1-q)\E{|I_{n-1}\setminus J^{(2)}_{n-1}|}+(1+p)\E{\sum_{u\in J^{(2)}_{n-1}} (1-q)^{k_u^{(a)}+1}}.
\end{align}
Now, by a simple induction argument, for $k\in \N\cup\{0\}$ we have  
$(1-q)^{k}\le 1-kq+{k \choose 2}q^2$. In particular, for $u\in J^{(2)}_{n-1}$,
\begin{align}\label{eq:kup}
    (1-q)^{k_u^{(a)}+1} 
    &\le (1-q)-q(1-q)k^{(a)}_u +q^2 (1-q) {k_u^{(a)}\choose 2};
\end{align}
for each of the different terms on the right-hand side of~\eqref{eq:kup}, we will bound the expectation of the sum over $u\in J^{(2)}_{n-1}$. 

First, for the third term on the right hand side of~\eqref{eq:kup}, 
recall from~\eqref{eq:kua} that for $u\in I_{n-1}$, 
$$ k^{(a)}_u=|\{v\in I_{n-2}\setminus \alpha_1(u): \alpha_2(u)\cap \alpha_1(v)\neq \emptyset \}|.$$
For $u\in J^{(2)}_{n-1}$, we have $\alpha_2(u)=\{w\}$ for some $w\in I_{n-3}$, and so if $v\neq v'\in I_{n-2}\setminus \alpha_1(u)$ with $\alpha_2(u)\cap \alpha_1(v)\neq \emptyset$ and $\alpha_2(u)\cap \alpha_1(v')\neq \emptyset$, we must have $w\in \alpha_1(v)\cap \alpha_1(v')$.
Therefore we can write
$$
\sum_{u\in J^{(2)}_{n-1}} {k_u^{(a)}\choose 2}
\le \tfrac 12 \sum_{w\in I_{n-3}}|\{(u,v,v')\in I^w_2 \times I^w_1 \times I^w_1: v\neq v', v,v'\notin \alpha_1(u)\}|
=\tfrac 12 N^{(2,1,1)}_{n-3}
$$
by the definition of $N^{(j_1,\ldots, j_k)}_n$ in~\eqref{eq:Nnumdefns}.
By \refL{lem:EN}, it follows that 
\begin{align*}
    q^2\E{\sum_{u\in J^{(2)}_{n-1}} {k_u^{(a)}\choose 2}}\le \frac{q^2}{2}\E{N_{n-3}^{(2,1,1)}}\le \frac{q^2}{2}(1+p)^4\E{Z_{n-3}}
    \le 8 q^2 \E{Z_{n-3}}.
\end{align*}
For the second term on the right hand side of~\eqref{eq:kup}, using \refL{lem:ES21}, we have
\begin{align*}
    q(1-q)\E{\sum_{u\in J^{(2)}_{n-1}} k_u^{(a)}}
    &\ge q(1-q)\E{Z_{n-3}}-C_6 q^2(1-q)\max_{3\le k\le 6}\E{Z_{n-k}}\\
    &\ge q\E{Z_{n-3}}-q^2(1+ C_6)\max_{3\le k\le 6}\E{Z_{n-k}}.
\end{align*}
Putting these estimates together with~\eqref{eq:Ynupperstart} and~\eqref{eq:kup} yields
\begin{align*}
   \E{Y_n}&\le (1+p)(1-q) \E{|I_{n-1}\setminus J_{n-1}^{(2)}|}\\
   &\quad +(1+p)\left((1-q)\E{|J^{(2)}_{n-1}|}-q(1-q)\E{\sum_{u\in J^{(2)}_{n-1}} k_u^{(a)}} + q^2\E{\sum_{u\in J^{(2)}_{n-1}} {k_u^{(a)}\choose 2}}
   \right)\\
   &\le (1+p)\left((1-q)\E{Z_{n-1}}-q\E{Z_{n-3}} + (C_6 +9) q^2\max_{3\le k\le 6} \E{Z_{n-k}}
   \right)\\
   &\le (1+p-q)\E{Z_{n-1}}-q\E{Z_{n-3}} + 2(C_6+9) q^2\max_{3\le k\le 6} \E{Z_{n-k}},
\end{align*}
which completes the proof.
\end{proof}

\begin{proof}[Proof of \refP{prop:Zn}]
Rearranging \eqref{eq:Z} from Lemma~\ref{lem:Z}, we have that for $n\ge 2$, 
\begin{equation} \label{eq:propZnineq}
   Y_n-M_n\le Z_n\le Y_n-\widetilde M_n+\widetilde L_n. 
\end{equation}
We have bounds on the expectations of $M_n$ and $\widetilde L_n$ in Lemma~\ref{lem:EML'},  $\widetilde M_n$ in \refL{lem:EM'} and $Y_n$ in Lemma~\ref{lem:EY} respectively; these bounds will now allow us to complete the proof of Proposition~\ref{prop:Zn}.

Take $0\le q \le p\le 1$.
For $n\ge 6$, by~\eqref{eq:propZnineq} and then by~\eqref{eq:EYupper} in Lemma~\ref{lem:EY},~\eqref{eq:lemEM'1} in \refL{lem:EM'} and~\eqref{eq:EL'} in Lemma~\ref{lem:EML'},
\begin{align*}
\E{Z_n}
&\le \E{Y_n}-\E{\widetilde M_n}+\E{\widetilde L_n}\\     
&\le (1+p-q)\E{Z_{n-1}}-q\E{Z_{n-3}} + C_7 q^2\max_{3\le k\le 6} \E{Z_{n-k}} \\
& \qquad -\frac{q}{2}\E{Z_{n-2}} +C_5 q^2\max_{3\le k\le 5} \E{Z_{n-k}}+4q^2(1+p)^6 \E{Z_{n-2}}\\
&\le (1+p-q)\E{Z_{n-1}}-\frac{q}{2}\E{Z_{n-2}}-q\E{Z_{n-3}} + C' q^2 \max_{2\le k\le 6}\E{Z_{n-k}},
\end{align*}
where $C'=C_7+C_5+2^8$. 
By~\eqref{eq:Z4} in \refP{prop:Z4}, we have
$$
\E{Z_{n-1}}\le (1+p)\E{Z_{n-2}}
\quad \text{and}\quad
\E{Z_{n-1}}\le (1+p)^{2}\E{Z_{n-3}}.
$$
Therefore
\begin{align*}
\E{Z_n}
&\le (1+p-\tfrac 52 q)\E{Z_{n-1}}+\frac{q}{2}((1+p)-1)\E{Z_{n-2}}\\
&\qquad +q((1+p)^{2}-1)\E{Z_{n-3}} + C' q^2 \max_{2\le k\le 6}\E{Z_{n-k}},
\end{align*}
which completes the proof of~\eqref{eq:upperZ}, since we are assuming $q\le p$. 

Now, for $n\ge 4$, by~\eqref{eq:propZnineq}, and then by~\eqref{eq:EYlower} in Lemma~\ref{lem:EY} and~\eqref{eq:EM} in Lemma~\ref{lem:EML'},
\begin{align*}
\E{Z_n}
&\ge \E{Y_n}-\E{M_n}\\     
&\ge (1+p)\E{Z_{n-1}}-q(1+p)^4\left( \E{Z_{n-3}}+\E{Z_{n-4}}\right)
    -C_3(1+p)q^2\E{Z_{n-4}}\\
& \qquad -\frac{q}{2}(1+p)^4\E{Z_{n-2}}\\
&\ge (1+p)\E{Z_{n-1}} -\tfrac{5}{2}q(1+p)^4\max_{2\le k\le 4}\E{Z_{n-k}} - 2C_3 q^2\E{Z_{n-4}},
\end{align*}
which completes the proof of~\eqref{eq:lowerZ}.

Finally, note that since for $n\ge 4$, the construction of $G'_n$ from $G'_{n-1}$ is the same as the construction of $G_n$ from $G_{n-1}$, we have that when the graph process $\cG(p,q)$ is replaced with $\cG'(p,q)$, Lemma~\ref{lem:Z} holds for $n\ge 4$,~\eqref{eq:EM} in Lemma~\ref{lem:EML'} holds for $n\ge 5$, and~\eqref{eq:EYlower} in Lemma~\ref{lem:EY} holds for $n\ge 7$
(because~\eqref{eq:kga} in Lemma~\ref{lem:kga} holds for $n\ge 6$, and~\eqref{eq:ERel} in Lemma~\ref{lem:ERel} holds for $n\ge 6$).
Therefore~\eqref{eq:lowerZ} holds for $m\ge 7$ when $Z_l$ is replaced with $Z'_l$ for each $l$, which completes the proof.
\end{proof}
It remains to prove Proposition~\ref{prop:ERn}.
For $n\ge 1$, let $K_n$ denote the number of deletions that occur when $\tilde I_n$ is constructed, i.e.~let
\begin{equation} \label{eq:Kndefn}
K_n := |\{ui:u\in I_{n-1}, i\le \xi_u, i \notin \cC_u\}|
=\sum_{u\in I_{n-1}}(\xi_u - |\mathcal C_u|)=\sum_{u\in I_{n-1}}\xi_u -Y_n.
\end{equation}
Recall from~\eqref{eq:Rkdefn} that $R_n:=|\cR_n|$, where $\cR_n:=\{v\in I_n:|\{w\in V_n :1\le d_{G_n}(v,w)\le 3\}|=3\}.$
We will use the following lower bound on $R_n$ to prove Proposition~\ref{prop:ERn}.
\begin{lem}\label{lem:RnCoupling}
For $n\ge 3$,  
\begin{align}\label{eq:R}
    R_n\ge T_n - \sum_{k=0}^2 (K_{n-k}+2\I{n-k\ge 2}M_{n-k}),
\end{align}
where 
$T_n:=|\{u\in I_{n-3}: \xi_u = \xi_{u1}=\xi_{u11}=1\}|$.
\end{lem}
\begin{proof}
Suppose $u\in I_{n-3}$ with $\xi_u = \xi_{u1}=\xi_{u11}=1$.
Suppose also that $u1\in \cC_u$, $u11\in \cC_{u1}$, $u111\in \cC_{u11}$, and for each $v\in \{u1,u11,u111\}$ we have $|\sigma(v)|=1$, i.e.
$\{v'\in I_{|v|}:v'\stackrel{m}{\sim}v\}=\emptyset$.
Then by our construction in Section~\ref{subsec:UlamHarris}, we have $u111 \in \mathcal R_n.$
Hence by the definitions of $M_n$ in~\eqref{eq:Mndefn} and $K_n$ in~\eqref{eq:Kndefn},
\begin{align*}
    R_n &\ge |\{u\in I_{n-3}: \xi_u = \xi_{u1}=\xi_{u11}=1,
\, u1\in \cC_u,\, u11\in \cC_{u1}, \, |\sigma(u1)|=1, \, |\sigma(u11)|=1 \}|\\
&\qquad \qquad -2M_n-K_n\\
&\ge |\{u\in I_{n-3}: \xi_u = \xi_{u1}=\xi_{u11}=1,
\, u1\in \cC_u, |\sigma(u1)|=1  \}|\\
&\qquad \qquad -2M_n-K_n-2M_{n-1}-K_{n-1}\\
&\ge |\{u\in I_{n-3}: \xi_u = \xi_{u1}=\xi_{u11}=1\}|- \sum_{k=0}^2 (K_{n-k}+2\I{n-k\ge 2} M_{n-k}).
\end{align*}
By the definition of $T_n$, this completes the proof.
\end{proof}

\begin{proof}[Proof of Proposition~\ref{prop:ERn}]
Take $n\ge 6$ and $0\le q\le p \le 1$. 
Recall from the statement of \refL{lem:RnCoupling} that we let
$T_n:=|\{u\in I_{n-3}: \xi_u = \xi_{u1}=\xi_{u11}=1\}|$.
Since for $u\in I_{n-3}$, conditional on $\mathcal F_{n-3}$, $\xi_u, \xi_{u1},\xi_{u11}$ are i.i.d.~with Poisson distribution with mean $1+p$,
\begin{align}\label{eq:ETn}
    \E{T_n} = ((1+p)e^{-(1+p)})^3\E{Z_{n-3}} \ge e^{-6}\E{Z_{n-3}},
\end{align}
where the inequality follows since $0\le p \le 1$.

We will use \refL{lem:RnCoupling} to obtain a lower bound on $\E{R_n}$; this will require an upper bound on $\E{K_{n-k}}+2\E{M_{n-k}}$ for $0\le k \le 2$. 
Take $m\ge 4$; by~\eqref{eq:Kndefn} we have
\begin{align*}
    \E{K_m}&=\E{\sum_{u\in I_{m-1}}\xi_u }-\E{Y_m}\\
    &\le (1+p)\E{Z_{m-1}} \\
    &\qquad - ((1+p)\E{Z_{m-1}}-q(1+p)^4\left( \E{Z_{m-3}}+\E{Z_{m-4}}\right)
    -C_3(1+p)q^2\E{Z_{m-4}})
\end{align*}
by~\eqref{eq:EYlower} in Lemma~\ref{lem:EY} and since $\E{\xi_u | \cF_{m-1}}=1+p$ for $u\in I_{m-1}$.
By \eqref{eq:EM} in Lemma~\ref{lem:EML'}, it follows that
\begin{align*}
    \E{K_m}+2\E{M_m} 
    &\le q(1+p)^4 (\E{Z_{m-3}}+\E{Z_{m-4}}+qC_3 \E{Z_{m-4}}+\E{Z_{m-2}})\\
    &\le 16(3+C_3)p\max_{2\le k\le 4} \E{Z_{m-k}}
\end{align*}
since $q\le p \le 1$.
Therefore, since $\E{Z_{n-2}}\leq (1+p) \E{Z_{n-3}}\le 2\E{Z_{n-3}}$ by~\eqref{eq:Z4} in \refP{prop:Z4}, we have 
\begin{align}\label{eq:KMm}
    \sum_{k=0}^2 (\E{K_{n-k}}+2\E{M_{n-k}}) &\le 48(3+C_3)p \max_{2\le k\le 6}\E{Z_{n-k}}\le 96(3+C_3)p \max_{3\le k\le 6}\E{Z_{n-k}}. 
\end{align}
By \eqref{eq:R} in \refL{lem:RnCoupling} together with~\eqref{eq:ETn} and~\eqref{eq:KMm},
it follows that
$$
\E{R_n}\ge e^{-6}\E{Z_{n-3}}-96(3+C_3)p \max_{3\le k\le 6}\E{Z_{n-k}},
$$
and by setting $2\eta=e^{-6}$ and $p_2= (192 e^6(3+C_3))^{-1}$, this establishes~\eqref{eq:propERn}.

Finally, 
since for $n\ge 4$, the construction of $G'_n$ from $G'_{n-1}$ is the same as the construction of $G_n$ from $G_{n-1}$, we have that when the graph process $\cG(p,q)$ is replaced with $\cG'(p,q)$,
Lemma~\ref{lem:RnCoupling} holds for $n\ge 6$,~\eqref{eq:EYlower} in Lemma~\ref{lem:EY} holds for $n\ge 7$,~\eqref{eq:EM} in Lemma~\ref{lem:EML'} holds for $n\ge 5$, and~\eqref{eq:Z4} in Proposition~\ref{prop:Z4} holds.
It follows that for $n\ge 9$, when $R_n$, $Z_{n-3}$ and $Z_{n-k}$ are replaced with $R'_n$, $Z'_{n-3}$ and $Z'_{n-k}$,~\eqref{eq:propERn} holds, which completes the proof.
\end{proof}

\section*{Acknowledgements}
We thank Joel Spencer for suggesting cousin merging in relation to a branching process to recover the coefficients of the critical probability for hypercube percolation. Part of this research was undertaken at the 2018 and 2019 Bellairs Workshops on Probability, held at the Bellairs Research Institute of McGill University. LE was partially supported by PAPIIT TA100820. FS was partially supported by the project AI4Research at Uppsala University and by the Wallenberg AI, Autonomous Systems and Software Program (WASP) funded by the Knut and Alice
Wallenberg Foundation.

\newpage

\appendix
\section{Appendix}\label{sec:Appendix}

\begin{lem}\label{lem:ML}
Let $H=(V,E)$ be a simple, non-empty graph with $t$ vertices and $m$ edges. Let $\ell$ be the number of directed paths of length two in $H$. Then $m-\ell<t$.
\end{lem}

\begin{proof}
If $m=0$ then we have $m-\ell \le 0 <t$. Now assume that $m\ge 1$.
For every vertex $v\in V$, denote the degree of $v$ by $d_v$. It is straightforward to see that $2m=\sum_{v\in V} d_v$ and $\ell=\sum_{v\in V} d_v(d_v-1)$. Recall that $|V|=t$ and $m\ge 1$; the result follows since 
\begin{align*}
    t+\ell =\sum_{v\in V} (1+d_v(d_v-1))\ge \sum_{v\in V} d_v=2m >m.
\end{align*}

\vspace{-1cm}
\end{proof}

\subsection{Natural couplings}

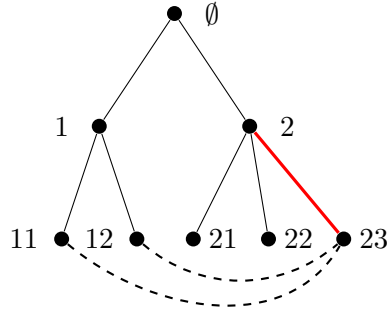
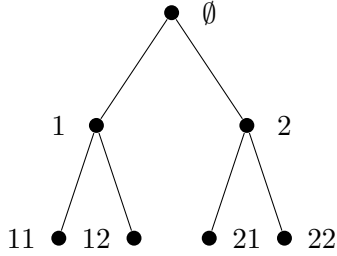
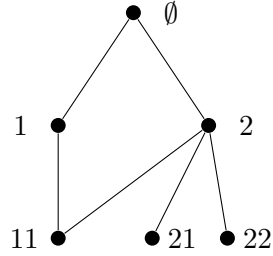
\begin{figure} 
\begin{center}
\begin{minipage}[b]{.9\textwidth}
\begin{center}
\begin{tikzpicture}[scale=1]
    		\tikzstyle{vertex}=[circle,fill=black, draw=none, minimum size=4pt,inner sep=2pt]
            \node[vertex] (0) at (0,3){};
    		\node[vertex] (1) at (-1,1.5){};
    		\node[vertex] (2) at (1,1.5){};
    		\node[vertex] (11) at (-1.5,0){};
    		\node[vertex] (12) at (-.5,0){};
    		\node[vertex] (21) at (0.25,0){};
    		\node[vertex] (22) at (1.25,0){};
    		\node[vertex] (23) at (2.25,0){};
        		\draw (1)--(0)--(2);
        		\draw (11)--(1)--(12);
        		\draw (21)--(2)--(22);
        		\draw[very thick, red] (2)--(23);
                \draw[dashed, thick] (11) to[out=-40,in=240] (23);
                \draw[dashed, thick] (12) to[out=-40,in=220] (23);
                \node[] at (.5,3){$\emptyset$};
                \node[] at (-1.5,1.5){$1$};
                \node[] at (1.5,1.5){$2$};
                \node[] at (-2,0){$11$};
                \node[] at (-1,0){$12$};
                \node[] at (.65,0){$21$};
                \node[] at (1.65,0){$22$};
                \node[] at (2.65,0){$23$};
    \end{tikzpicture}
\end{center}
\vspace{-0.5cm}
    \subcaption{Ulam-Harris tree. Each vertex $u$ in generations 0 and 1 has $\xi '_u$ offspring in the next generation linked with black edges, and $\xi_u-\xi '_u$ offspring linked with red edges. Dotted lines between pairs of vertices $u$ and $v$ in generation 2 indicate that $\mu_{\{u,v\}}=1$.}
\end{minipage}
\end{center}

\begin{minipage}[b]{.5\textwidth}
\begin{center}
\begin{tikzpicture}[scale=1]
    		\tikzstyle{vertex}=[circle,fill=black, draw=none, minimum size=4pt,inner sep=2pt]
            \node[vertex] (0) at (0,3){};
    		\node[vertex] (1) at (-1,1.5){};
    		\node[vertex] (2) at (1,1.5){};
    		\node[vertex] (11) at (-1.5,0){};
    		\node[vertex] (12) at (-.5,0){};
    		\node[vertex] (21) at (0.5,0){};
    		\node[vertex] (22) at (1.5,0){};
        		\draw (1)--(0)--(2);
        		\draw (11)--(1)--(12);
        		\draw (21)--(2)--(22);
                \node[] at (.5,3){$\emptyset$};
                \node[] at (-1.5,1.5){$1$};
                \node[] at (1.5,1.5){$2$};
                \node[] at (-2,0){$11$};
                \node[] at (-1,0){$12$};
                \node[] at (1,0){$21$};
                \node[] at (2,0){$22$};
    \end{tikzpicture}
\end{center}
    \subcaption{Realisation of $G_2(p',q)$}
\end{minipage}
\begin{minipage}[b]{.5\textwidth}
\begin{center}
\begin{tikzpicture}[scale=1]
    		\tikzstyle{vertex}=[circle,fill=black, draw=none, minimum size=4pt,inner sep=2pt]
            \node[vertex] (0) at (0,3){};
    		\node[vertex] (1) at (-1,1.5){};
    		\node[vertex] (2) at (1,1.5){};
    		\node[vertex] (11) at (-1,0){};
    		\node[vertex] (21) at (0.25,0){};
    		\node[vertex] (22) at (1.25,0){};
        		\draw (1)--(0)--(2)--(22);
        		\draw (1)--(11)--(2)--(21);
                \node[] at (.5,3){$\emptyset$};
                \node[] at (-1.5,1.5){$1$};
                \node[] at (1.5,1.5){$2$};
                \node[] at (-1.45,0){$11$};
                \node[] at (.65,0){$21$};
                \node[] at (1.65,0){$22$};
    \end{tikzpicture}
\end{center}
    \subcaption{Realisation of $G_2(p,q)$}
\end{minipage}
    \caption{The number of individuals in the second generation is smaller in the graph process $\mathcal G(p,q)$ than in the graph process $\mathcal G(p',q)$, while $p>p'>-1$.}
    \label{fig:pnonmonotone}
\end{figure}

As mentioned in the introduction, we now briefly discuss natural couplings of the graph processes with different values of $p$ and $q$, using our Ulam-Harris construction in Section~\ref{sec:UlamHarris}.
Take $p>p'>-1$ and $q'>q>0$.
Let $(\xi_u)_{u\in \mathcal U}$ and $(\xi'_u)_{u\in \mathcal U}$ be families of i.i.d.~random variables where $\xi_u$ has Poisson distribution with mean $1+p$ for each $u$, and $\xi'_u$ has Poisson distribution with mean $1+p'$ for each $u$, coupled in such a way that $\xi_u\ge \xi'_u$ $\forall u\in \mathcal U$.
Let $(\delta_{u,v})_{u,v\in \mathcal U}$, $(\mu_{\{u,v\}})_{u \neq v\in \mathcal U}$ be independent families of i.i.d.~random variables with Bernoulli distribution with mean $q$,
    and let $(\delta'_{u,v})_{u,v\in \mathcal U}$, $(\mu'_{\{u,v\}})_{u\neq v\in \mathcal U}$ be independent families of i.i.d.~random variables with Bernoulli distribution with mean $q'$, 
    coupled in such a way that $\delta_{u,v}\le \delta'_{u,v}$ and $\mu_{\{u,v\}}\le \mu'_{\{u,v\}}$ $\forall u,v\in \mathcal U$.
    
    Construct the graph process $\mathcal G(p,q)$ according to the Ulam-Harris construction in Section~\ref{sec:UlamHarris}, using the random variables $(\xi_u)_{u\in \mathcal U}$, $(\delta_{u,v})_{u,v\in \mathcal U}$ and $(\mu_{\{u,v\}})_{u \neq v\in \mathcal U}$. Similarly, construct the graph process $\mathcal G(p',q)$ using the random variables $(\xi'_u)_{u\in \mathcal U}$, $(\delta_{u,v})_{u,v\in \mathcal U}$ and $(\mu_{\{u,v\}})_{u \neq v\in \mathcal U}$,
    and construct the graph process $\mathcal G(p,q')$ using the random variables $(\xi_u)_{u\in \mathcal U}$, $(\delta '_{u,v})_{u,v\in \mathcal U}$ and $(\mu '_{\{u,v\}})_{u \neq v\in \mathcal U}$.
    We illustrate possible realisations of these graph processes in Figures~\ref{fig:pnonmonotone} and~\ref{fig:qnonmonotone}, showing that these couplings do not trivially imply monotonicity of the survival probability in $p$ or $q$.

\begin{figure} 
\begin{center}
\begin{minipage}[b]{.9\textwidth}
\begin{center}
\begin{tikzpicture}[scale=1]
    		\tikzstyle{vertex}=[circle,fill=black, draw=none, minimum size=4pt,inner sep=2pt]
            \node[vertex] (0) at (0,4.5){};
    		\node[vertex] (1) at (-1,3){};
    		\node[vertex] (2) at (1,3){};
    		\node[vertex] (11) at (-1,1.5){};
    		\node[vertex] (21) at (.25,1.5){};
    		\node[vertex] (22) at (1.75,1.5){};
    		\node[vertex] (211) at (-.25,0){};
    		\node[vertex] (221) at (.75,0){};
    		\node[vertex] (222) at (1.75,0){};
    		\node[vertex] (223) at (2.75,0){};
        		\draw (11)--(1)--(0)--(2);
        		\draw (211)--(21)--(2)--(22)--(223);
        		\draw (221)--(22)--(222);
                \draw[dashed, ultra thick, red] (11) to[out=-40,in=220] (21);        \draw[dashed, thick] (211) to[out=-40,in=210] (221);
                \draw[dashed, thick] (211) to[out=-50,in=230] (222);
                \node[] at (.5,4.5){$\emptyset$};
                \node[] at (-1.5,3){$1$};
                \node[] at (1.5,3){$2$};
                \node[] at (-1.5,1.5){$11$};
                \node[] at (.65,1.5){$21$};
                \node[] at (2.15,1.5){$22$};
                \node[] at (-.75,0){$211$};
                \node[] at (1.25,0){$221$};
                \node[] at (2.25,0){$222$};
                \node[] at (3.25,0){$223$};
    \end{tikzpicture}
\end{center}
\vspace{-0.3cm}
    \subcaption{Ulam-Harris tree. Each vertex $u$ in generations 0, 1 and 2 has $\xi_u$ offspring in the next generation. Black dotted lines between pairs of vertices $u$ and $v$ indicate that $\mu_{\{u,v\}}=1$; red dotted lines indicate that $\mu'_{\{u,v\}}=1$ and $\mu_{\{u,v\}}=0$.}
\end{minipage}
\end{center}

\begin{minipage}[b]{.5\textwidth}
\begin{center}
\begin{tikzpicture}[scale=1]
    		\tikzstyle{vertex}=[circle,fill=black, draw=none, minimum size=4pt,inner sep=2pt]
            \node[vertex] (0) at (0,4.5){};
    		\node[vertex] (1) at (-1,3){};
    		\node[vertex] (2) at (1,3){};
    		\node[vertex] (11) at (0,1.5){};
    		\node[vertex] (22) at (1.75,1.5){};
    		\node[vertex] (221) at (.75,0){};
    		\node[vertex] (222) at (1.75,0){};
    		\node[vertex] (223) at (2.75,0){};
        		\draw (11)--(1)--(0)--(2);
        		\draw (11)--(2)--(22)--(223);
        		\draw (221)--(22)--(222);
                \node[] at (.5,4.5){$\emptyset$};
                \node[] at (-1.5,3){$1$};
                \node[] at (1.5,3){$2$};
                \node[] at (-.5,1.5){$11$};
                \node[] at (2.15,1.5){$22$};
                \node[] at (1.25,0){$221$};
                \node[] at (2.25,0){$222$};
                \node[] at (3.25,0){$223$};
\end{tikzpicture}
\end{center}
    \subcaption{Realisation of $G_3(p,q')$}
\end{minipage}
\begin{minipage}[b]{.5\textwidth}
\begin{center}
\begin{tikzpicture}[scale=1]
    		\tikzstyle{vertex}=[circle,fill=black, draw=none, minimum size=4pt,inner sep=2pt]
            \node[vertex] (0) at (0,4.5){};
    		\node[vertex] (1) at (-1,3){};
    		\node[vertex] (2) at (1,3){};
    		\node[vertex] (11) at (-1,1.5){};
    		\node[vertex] (21) at (.25,1.5){};
    		\node[vertex] (22) at (1.75,1.5){};
    		\node[vertex] (211) at (1,0){};
    		\node[vertex] (223) at (2.75,0){};
        		\draw (11)--(1)--(0)--(2);
        		\draw (211)--(21)--(2)--(22)--(223);
        		\draw (211)--(22);
                \node[] at (.5,4.5){$\emptyset$};
                \node[] at (-1.5,3){$1$};
                \node[] at (1.5,3){$2$};
                \node[] at (-1.5,1.5){$11$};
                \node[] at (.65,1.5){$21$};
                \node[] at (2.15,1.5){$22$};
                \node[] at (1.5,0){$211$};
                \node[] at (3.25,0){$223$};
    \end{tikzpicture}
\end{center}
    \subcaption{Realisation of $G_3(p,q)$}
\end{minipage}
    \caption{The number of individuals in the third generation is smaller in the graph process $\mathcal G(p,q)$ than in the graph process $\mathcal G(p,q')$, while $q'>q>0$.}
    \label{fig:qnonmonotone}
\end{figure}
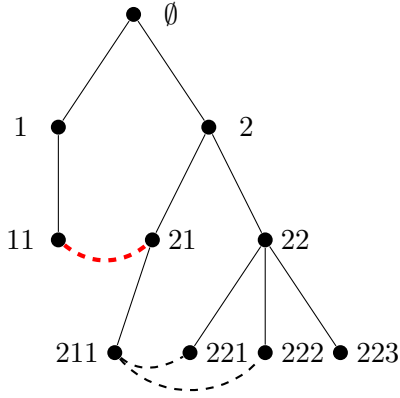
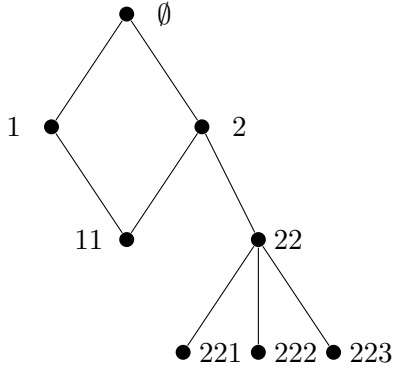
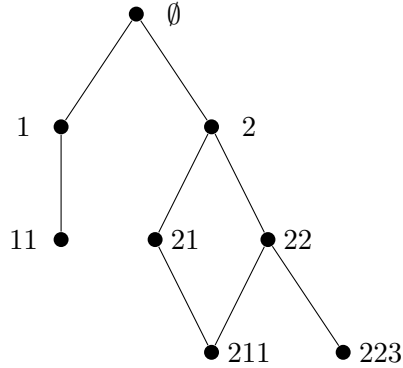

\end{document}